\documentclass[12pt]{amsart}
\usepackage{amsfonts, amssymb}
\usepackage{amsmath}
\usepackage{amsthm}
\usepackage{cite}
\numberwithin{equation}{section}

\newtheorem{theorem}{Theorem}[section]
\newtheorem{lemma}[theorem]{Lemma}

\begin{document}

\title[Bank--Laine functions]{Bank--Laine functions with preassigned\\
number of zeros}

\author{Yueyang Zhang}
\address{School of Mathematics and Physics, University of Science and Technology Beijing, No.~30 Xueyuan Road, Haidian, Beijing, 100083, P.R. China}
\email{zhangyueyang@ustb.edu.cn}
\thanks{The author is supported by the Fundamental Research Funds for the Central Universities~{(FRF-TP-19-055A1)} and a Project supported by the National Natural Science Foundation of China~{(12301091)}}



\date{\today}

\commby{}

\begin{abstract}
A Bank--Laine function $E$ is written as $E=f_1f_2$ for two normalized solutions $f_1$ and $f_2$ of the second order differential equation $f''+Af=0$, where $A$ is an entire function. In this paper, we first complete the construction of Bank--Laine functions by Bergweiler and Eremenko. Then, letting $n\in \mathbb{N}$ be a positive integer, we show the existence of entire functions $A$ for which the associated Bank--Laine functions $E=f_1f_2$ have preassigned exponent of convergence of number of zeros $\lambda(E)$ of three types:

\begin{enumerate}
\item [(1)]
for every two numbers $\lambda_1,\lambda_2\in[0,n]$ such that $\lambda_1\leq \lambda_2$, there exists an entire function $A$ of order $\rho(A)=n$ such that $E=f_1f_2$ satisfies $\lambda(f_1)=\lambda_1$, $\lambda(f_2)=\lambda_2$ and $\lambda(E)=\lambda_2\leq \rho(E)=n$;

\item [(2)]
for every number $\rho\in(n/2,n)$ and $\lambda\in[0,\infty)$, there exists an entire function $A$ of order $\rho(A)=\rho$ such that $E=f_1f_2$ satisfies $\lambda(f_1)=\lambda$, $\lambda(f_2)=\infty$ and, moreover, $E_c=f_1(cf_1+f_2)$ satisfies $\lambda(E_c)=\infty$ for any constant $c$;

\item [(3)]
for every number $\lambda\in[0,n]$, there exists an entire function $A$ of order $\rho(A)=n$ such that $E=f_1f_2$ satisfies $\lambda(f_1)=\lambda$, $\lambda(f_2)=\infty$ and, moreover, $E_c=f_1(cf_1+f_2)$ satisfies $\lambda(E_c)=\infty$ for any constant $c$.
\end{enumerate}
The construction for the three types of Bank--Laine functions requires new developments of the method of quasiconformal surgery by Bergweiler and Eremenko.

\end{abstract}

\maketitle


\section{Introduction}\label{intro} 


In the last 1980s, Bank and Laine~\cite{Banklaine1982,Banklaine1982-2} initiated the study on the complex oscillation of the second order differential equation
\begin{equation}\label{bank-laine0}
f''+Af=0,
\end{equation}
where $A$ is an entire function, under the framework of Nevanlinna theory (see~\cite{Hayman1964Meromorphic} for the standard notation and basic results).
It is elementary that all nontrivial solutions of equation \eqref{bank-laine0} are entire functions. Let $f_1$ and $f_2$ be two linearly independent solutions of \eqref{bank-laine0}. When $A$ is a polynomial, the product $E=f_1f_2$ satisfies $\lambda(E)=\rho(E)$ and $\rho(E)$ is an integer order or half-integer order. When $A$ is transcendental and of finite order, an application of the lemma on the logarithmic derivative~\cite[Section~2.2]{Hayman1964Meromorphic} easily yields that all nontrivial solutions of \eqref{bank-laine0} satisfy $\rho(f)=\infty$. Bank and Laine proved the following

\begin{theorem}[\cite{Banklaine1982,Banklaine1982-2}]\label{maintheorem1}
Let $E$ be the product of two linearly independent solutions $f_1$ and $f_2$ of the differential equation \eqref{bank-laine0}, where $A$ is a transcendental entire function of finite order. Then
\begin{itemize}
  \item [(1)] if $\rho(A)\not\in \mathbb{N}$, then $\lambda(E) \geq \rho(A)$;
  \item [(2)] if $\rho(A)<1/2$, then $\lambda(E)=\infty$.
\end{itemize}

\end{theorem}

The condition $\rho(A)<1/2$ in (2) in Theorem~\ref{maintheorem1} was extended to $\rho(A)=1/2$ by Shen~\cite{Shen1985} and Rossi~\cite{Rossi1986} independently. They actually showed that if $1/2\leq \rho(A)<1$, then the conclusion $\lambda(E) \geq \rho(A)$ can be improved to satisfy the inequality
\begin{equation}\label{bank-laine1}
\frac{1}{\rho(A)}+\frac{1}{\lambda(E)}\leq 2.
\end{equation}
A problem left by their work--known as \emph{the Bank--Laine conjecture}--is whether $\lambda(E)=\infty$ whenever $\rho(A)\not\in \mathbb{N}$. This conjecture has attracted much interest; see the surveys \cite{Gundersen2014,lainetohge2008} and references therein. On the other hand, when $A$ has integer order, there are examples of $A$ (say $A(z)=e^z-q^2/16$ for an odd integer $q$) such that the product $E=f_1f_2$ of two linearly independent solutions $f_1$ and $f_2$ of \eqref{bank-laine0} satisfies $\lambda(E)<\infty$ or even $\lambda(E)=0$; see~\cite[Theorem~5.22]{Laine1993} and also more complex oscillation results on the second order differential equation \eqref{bank-laine0} there.

An entire function $E$ is called a \emph{Bank--Laine function} if $E(z)=0$ implies either $E'(z)=1$ or $E'(z)=-1$. In particular, $E$ is called \emph{special} if $E(z)=0$ implies that $E'(z)=1$. We say that two linearly independent solutions $f_1$ and $f_2$ of \eqref{bank-laine0} are \emph{normalized} if the \emph{Wronskian determinant}~$W(f_1,f_2)=f_1f_2'-f_1'f_2=1$. Then a Bank--Laine function $E$ can be written as $E=f_1f_2$ for two normalized solutions $f_1$ and $f_2$ of the second order differential equation \eqref{bank-laine0}; see~\cite[Proposition~6.4]{Laine1993}. Most study on the Bank--Laine conjecture has concentrated on constructing Bank--Laine functions with certain properties. It is easy to see that all Bank--Laine functions have order~$\geq1$ by Theorem~\ref{maintheorem1} and the inequality \eqref{bank-laine1}. Bergweiler and Eremenko~\cite{Bergweilereremenko2017,Bergweilereremenko2019} disproved the Bank--Laine conjecture by using the method of quasiconformal surgery to construct, for each $\rho\in(1/2,\infty)$, an entire function $A$ such that $\rho(A)=\rho$ and the differential equation \eqref{bank-laine0} has two linearly independent solutions $f_1,f_2$ such that $\lambda(E)=\rho(E)<\infty$ and one of $f_1$ and $f_2$ has no zeros. Their construction in \cite{Bergweilereremenko2017} shows for the first time that there exists a special Bank--Laine function $E$ such that $\rho(E)$ is not an integer order or half-integer order. We refer to~\cite{Ahlfors2006,lehtoVirtanen2008} for the basic results of quasiconformal mapping theory.

For two linearly independent solutions $f_1$ and $f_2$ of equation \eqref{bank-laine0}, it is well known that the ratio $F=f_2/f_1$ satisfies the \emph{Schwarzian differential equation}
\begin{equation}\label{schwarzdifferential}
\begin{split}
S(F):=\left(\frac{F''}{F'}\right)'-\frac{1}{2}\left(\frac{F''}{F'}\right)^2=\frac{F'''}{F'}-\frac{3}{2}\left(\frac{F''}{F'}\right)^2=2A,
\end{split}
\end{equation}
where the expression $S(F)$ is called the \emph{Schwarzian derivative} of $F$. Then $F$ is a locally univalent meromorphic function and the Schwarzian derivative is invariant under the linear fractional transformation of $F$. This implies a one-to-one correspondence between the equivalence class of locally univalent meromorphic functions and entire functions. A normalized pair $(f_1,f_2)$ of solutions of \eqref{bank-laine0} can be recovered from $F$ by the formulas $f_1^2=1/F'$ and $f_2^2=F^2/F'$. The quotient $F/F'=f_1f_2=E$ is then a Bank--Laine function and all Bank--Laine functions arise this way. Zeros of a Bank--Laine function $E$ are poles and zeros of $F$, and $E$ is special if and only if $F$ is entire. By gluing different locally univalent meromorphic functions together using quasiconformal surgery, Bergweiler and Eremenko obtain the desired locally univalent entire functions $F$ and then recover the corresponding coefficients $A$ from equation~\eqref{schwarzdifferential}. In~\cite{Zhang2024}, the present author completed Bergweiler and Eremenko's construction in \cite{Bergweilereremenko2019} and, in particular, prove the following

\begin{theorem}[\cite{Zhang2024}]\label{maintheorem2}
Let $\lambda\in(1,\infty)$ and $\delta\in[0,1]$. Then there exists a Bank--Laine function $E$ such that $\lambda(E)=\rho(E)=\lambda$ and $E=f_1f_2$ for two entire functions $f_1$ and $f_2$ such that $\lambda(f_1)=\delta\lambda$ and $\lambda(f_2)=\lambda$.
\end{theorem}

The construction in \cite{Zhang2024} starts from a family of locally univalent meromorphic functions $F=f_2/f_1$ with $f_1$ and $f_2$ being two normalized solutions such that $\lambda(f_i)<\infty$, $i=1,2$, of the second order differential equation
\begin{equation*}
f''-\left[\frac{1}{4}e^{2z}+\left(\frac{k_2-k_1}{2}\right)e^{z}+\frac{(k_1+k_2+1)^2}{4}\right]f=0,
\end{equation*}
where $k_1,k_2$ are two nonnegative integers (see also~\cite{ChiangIsmail2006,zhang2021,zhang2021-1,Zhang2022}). This equation admits two normalized solutions $f_1$ and $f_2$ of the form
\begin{equation*}
\left\{
  \begin{array}{ll}
    f_1=\left(\sum_{i=0}^{k_1}a_ie^{iz}\right)e^{-\frac{k_1+k_2+1}{2}z}e^{-\frac{1}{2}e^{z}},       &\\
    f_2=\left(\sum_{j=0}^{k_2}b_je^{jz}\right)e^{-\frac{k_1+k_2+1}{2}z}e^{\frac{1}{2}e^{z}}
  \end{array}
\right.
\end{equation*}
with
\begin{equation*}
\left\{
  \begin{array}{ll}
    a_i=\frac{k_2!}{(i+k_2)!}a_0,       & \quad  i=1,\cdots,k_1,\\
    b_j=\frac{(-1)^jk_1!}{(j+k_1)!}b_0, & \quad  j=1,\cdots,k_2
  \end{array}
\right.
\end{equation*}
and $a_0$ and $b_0$ are two constants such that $a_{k_1}b_{k_2}=\frac{k_1!k_2!}{[(k_1+k_2)!]^2}a_0b_0=1$. Let
\begin{equation*}
\begin{split}
F=\frac{f_2}{f_1}=\left(\frac{\sum_{j=0}^{k_2}b_je^{jz}}{\sum_{i=0}^{k_1}a_ie^{iz}}\right)e^{e^{z}}.
\end{split}
\end{equation*}
We choose $a_0=b_0$. Let $A_0=B_0=1$ and $A_i=a_i/a_0$ and $B_j=b_j/b_0$. For every pair $(m,2n)$ of two nonnegative integers $m$ and $n$, we consider the rational function
\begin{equation}\label{recuree1pm8}
\begin{split}
R_{m,n}(z)=\frac{P_n(z)}{Q_m(z)}=\frac{\sum_{j=0}^{2n}B_jz^j}{\sum_{i=0}^{m}A_iz^i}.
\end{split}
\end{equation}
Thus we may write
\begin{equation*}
\begin{split}
h_{m,n}(z)=R_{m,n}(z)e^{z}=\left(\frac{\sum_{j=0}^{2n}B_jz^j}{\sum_{i=0}^{m}A_iz^i}\right)e^{z}.
\end{split}
\end{equation*}
It follows that $A_{m}B_{2n}=\frac{m!(2n)!}{[(m+2n)!]^2}=\frac{1}{\binom{m+2n}{m}(m+2n)!}$.
Then the meromorphic (entire when $m=0$) function
\begin{equation}\label{recuree1pm9}
\begin{split}
g_{m,n}(z)=h_{m,n}(e^z)=\left(\frac{\sum_{j=0}^{2n}B_je^{jz}}{\sum_{i=0}^{m}A_ie^{iz}}\right)e^{e^{z}}
\end{split}
\end{equation}
satisfies
\begin{equation}\label{recuree1pm10}
\begin{split}
g'_{m,n}(z)=\frac{1}{\binom{m+2n}{m}(m+2n)!}\frac{e^{e^{z}+(m+2n+1)z}}{\left(\sum_{i=0}^{m}A_ie^{iz}\right)^2}.
\end{split}
\end{equation}
Since $A_j$ are all positive numbers, $g'_{m,n}(z)\not=0$ for all $z\in \mathbb{C}$ and $g_{m,n}$ is increasing on $\mathbb{R}$, and satisfies $g_{m,n}(x)\to 1$ as $x\to -\infty$ as well as $g_{m,n}(x)\to \infty$ as $x\to \infty$. By the recursive formulas for $A_i$, we see that if $n\geq 1$, then $A_1+A_2+\cdots+A_{m}<1$ for all $m\geq 1$, which implies that $g_{m,n}$ has no poles on the imaginary axis. If $n=0$ and $m\geq 1$, it is possible that $g_{m,n}$ has some poles on the imaginary axis. Bergweiler and Eremenko glued the functions $g_{m,n}$ in \eqref{recuree1pm9} with $m=0$ for two or infinitely many different even numbers $2n$. Our construction in \cite{Zhang2024} considered the generic choice of $m$ and $n$. For the underlying ideas behind the gluing, see the description in~\cite[Section~2]{Bergweilereremenko2019}.

In the proof of Theorem~\ref{maintheorem2} in \cite{Zhang2024}, for each number $\rho\in(1/2,1)$, we let $\gamma=1/(2\rho-1)$ so that $1/\rho+1/(\rho\gamma)=2$. The hypothesis $1/2<\rho<1$ corresponds to $\gamma>1$ as well as $\lambda=\rho\gamma>1$. The proof of Theorem~\ref{maintheorem2} actually shows that, for every $\delta\in[0,1]$, there exists an entire function $A$ of order $\rho(A)=\rho$ such that the associated Bank--Laine function $E=f_1f_2$ for two normalized solutions $f_1$ and $f_2$ of the differential equation \eqref{bank-laine0} satisfies $\lambda(f_1)=\delta\rho\gamma$ and $\lambda(f_2)=\rho\gamma$, respectively. In particular, when $\delta=0$ the function $f_1$ can be chosen to have no zeros so that $E$ is special. Theorem~\ref{maintheorem2} completes the construction for Bank--Laines in~\cite[Corollary~1.1]{Bergweilereremenko2019}. As pointed out in~\cite{Zhang2024}, we can extend the results to the case that $\rho\in(n/2,n)$ for any positive integer $n$ so that there exists an entire function $A$ of order $\rho(A)=\rho$ such that the associated Bank--Laine function $E=f_1f_2$ satisfies $\lambda(E)=\lambda$ and $n/\rho+n/\lambda=2$ by following exactly the same process as in \cite[Section~5]{Bergweilereremenko2019}, thus completing the construction in \cite[Corollary~1.2]{Bergweilereremenko2019}; see also the comments in Subsection~\ref{Completion of the proof-ex} and Subsection~\ref{Completion of the proof-add}.
A more general assumption $\rho\in(1/2,1]$ was considered in \cite{Zhang2024}. Unfortunately, the proof for Theorem~\ref{maintheorem2} there is not complete because the proof for the two lemmas in \cite[Lemma~2.1 and Lemma~2.2]{Zhang2024} contains gaps in both of the two cases $\gamma=1$ and $\gamma>1$. One of the purposes of this paper is to fill these gaps. The case $\gamma=1$ will be included in Theorem~\ref{maintheorem4} below and the corrected versions of the two lemmas will be provided in Section~\ref{Proof of maintheorem5}.

Here and henceforth, we use the notation $S_1(r)\asymp S_2(r)$, provided that there are two positive constants $c_1$ and $c_2$ such that two positive quantities $S_1(r)$ and $S_2(r)$ defined for $r\in(0,\infty)$ satisfy $c_2 S_2(r)\leq S_1(r)\leq c_1S_2(r)$ for all large~$r$. With this notation, we note from the proof of Theorem~\ref{maintheorem2}, together with the modifications in Section~\ref{Proof of maintheorem5}, that when $\delta=1$ and $\gamma>1$, the numbers of zeros of $f_1$ and $f_2$ actually satisfy
\begin{equation}\label{bank-laine1kktf}
n(r,0,f_1)\asymp (\log r)^{-3}r^{\lambda} \quad \text{and} \quad n(r,0,f_2)\asymp r^{\lambda}.
\end{equation}
Thus, for the Bank--Laine functions constructed in Theorem~\ref{maintheorem2}, $f_1$ always has less number of zeros than $f_2$. It seems that the method in \cite{Zhang2024} does not allow to remove the term $(\log r)^{-3}$ in \eqref{bank-laine1kktf}. As a complement to Theorem~\ref{maintheorem2}, we shall first consider gluing two locally univalent meromorphic functions $g_{m,n}$ in \eqref{recuree1pm9} as in \cite{Bergweilereremenko2017} and prove the following

\begin{theorem}\label{maintheorem3}
Let $M$ and $N$ be two positive integers. Then there exists an entire function $A$ of order
\begin{equation*}
\rho=1+\frac{\log^2(M/N)}{4\pi^2}
\end{equation*}
such that the associated Bank--Laine function $E=f_1f_2$ for two normalized solutions $f_1$ and $f_2$ of \eqref{bank-laine0} satisfies $\lambda(E)=\rho(E)=\rho$ and $\lambda(f_1)=\lambda(f_2)=\rho$.
\end{theorem}

In the proof of Theorem~\ref{maintheorem3}, we have actually shown that the numbers of zeros of $f_1$ and $f_2$ actually satisfy $n(r,0,f_1) \asymp r^{\rho}$ and $n(r,0,f_2) \asymp r^{\rho}$ respectively. Thus $n(r,0,f_1)\asymp n(r,0,f_2)$. We remark that the two numbers $M$ and $N$ can also be chosen suitably so that $f_1$ has no zeros and thus the corresponding Bank--Laine function $E$ is special. We may even choose $M$ and $N$ suitably so that $f_2$ has no zeros.

We may also use the method of quasiconformal surgery in~\cite{Bergweilereremenko2019,Zhang2024} to construct Bank--Laine functions $E=f_1f_2$ such that $n(r,0,f_1)\asymp n(r,0,f_2)$ when $\rho(A)=\rho\in(n/2,n)$ for an integer $n\geq 2$. To see this, we only need to note that the function $G_1$ defined in \cite[Section~5]{Bergweilereremenko2019} can be replaced by $1/G_0$ since the gluing there only use the fact that $G_1=1/G_0$ for $z\in J$. There are two major differences between these two kinds of Bank--Laine functions. First, the Bank--Laine functions $E$ constructed in Theorem~\ref{maintheorem2} and its extensions have completely regular growth while the Bank--Laine functions $E$ constructed in Theorem~\ref{maintheorem3} satisfy $E(\Gamma_1(t))\to 0$ and $E(\Gamma_2(t))\to\infty$ as $t\to\infty$ on certain logarithmic spirals $\Gamma_1$ and $\Gamma_2$. Second, the Bank--Laine functions $E$ constructed in Theorem~\ref{maintheorem2} and its extensions satisfy $\rho(A)=\rho\in(n/2,n)$ for an integer $n\geq 1$ and $\lambda(E)=\rho(E)=\lambda$ for a number $\lambda$ such that $n/\rho+n/\lambda=2$ and thus $\lambda(E)=\rho(E)>\rho(A)$ in general while the Bank--Laine functions $E$ constructed in Theorem~\ref{maintheorem3} always satisfy $\lambda(E)=\rho(E)=\rho(A)$. Bank--Laine functions $E$ such that $n/\rho(A)+n/\lambda(E)<2$ can also be constructed by an extension of the method in~\cite{Bergweilereremenko2019}, which are obtained after this paper is finished.

By looking at Theorem~\ref{maintheorem1}-Theorem~\ref{maintheorem3}, one notes that the existence of Bank--Laine functions in the case $\rho(A)$ is an integer is still not clear. Langley~\cite{Langley1998-99} pointed out, using \emph{Lagrange interpolation}, we can make Bank--Laine functions $E=Pe^{Q}$ with $P,Q$ being polynomials. However, it is not known if there exist Bank--Laine functions $E$ such that $\lambda(E)=\lambda<\rho(E)=n$ for any preassigned number $\lambda$ and positive integer $n$. It seems that the method of quasiconformal surgery in \cite[Subsection~5.3]{Bergweilereremenko2019} is invalid when dealing with this case. In this paper we shall describe a new method of quasiconformal surgery for construction of such kind of Bank--Laine functions. We prove the following

\begin{theorem}\label{maintheorem4}
Let $n\in \mathbb{N}$ be a positive integer and $\lambda_1,\lambda_2\in[0,n]$ be two numbers such that $\lambda_1\leq \lambda_2$. Then there exists an entire function $A$ of order $\rho(A)=n$ such that the associated Bank--Laine function $E=f_1f_2$ for two normalized solutions $f_1$ and $f_2$ of \eqref{bank-laine0} satisfies $\lambda(f_1)=\lambda_1$, $\lambda(f_2)=\lambda_2$ and $\lambda(E)=\lambda_2\leq \rho(E)=n$.
\end{theorem}

In the proof of Theorem~\ref{maintheorem4}, we have actually shown that the numbers of zeros of $f_1$ satisfies $n(r,0,f_1)\asymp (\log r)^2$ when $\lambda_1=0$ or $n(r,0,f_1)\asymp r^{\lambda_1}$ when $0<\lambda_1\leq n$. The number of zeros of $f_2$ has similar asymptotic behaviors. By slightly modifying the proof of Theorem~\ref{maintheorem4}, $f_1$ or $f_2$ can be chosen to have only finitely many zeros in the plane, as pointed out in Subsection~\ref{Preliminary lemmas}. The Bank--Laine function $E$ constructed in Theorem~\ref{maintheorem4} are the first nontrivial examples such that $\lambda(E)\leq \rho(A)$ and $\lambda(E)=\lambda$ is not zero or integer. These examples show that Bank--Laine function $E$ of integer order indeed have value distribution properties not covered by Theorem~\ref{maintheorem1}.

As in \cite[Theorem~3]{Zhang2024}, we can describe the asymptotic behaviors of the Bank--Laine function $E$ in Theorem~\ref{maintheorem3} and also the corresponding coefficient function $A$ of the second order differential equation~\eqref{bank-laine0} in the case $n=1$. We have the following

\begin{theorem}\label{maintheorem4-asym}
When $n=1$, the functions $A$ and $E$ in Theorem~\ref{maintheorem4} can be chosen so that
\begin{equation*}
\begin{split}
\frac{1}{2}\log |A(re^{i\theta})|\sim \log\frac{1}{|E(re^{i\theta})|} \sim \frac{1}{l}r\cos(\theta) \quad \text{for} \quad |\theta|\leq (1-\varepsilon)\frac{\pi}{2},
\end{split}
\end{equation*}
where $l$ is an integer such that $l=1$ when $\lambda_1\leq \lambda_2<1$ and $l=3$ when $\lambda_1<\lambda_2=1$ and $l=4$ when $\lambda_1=\lambda_2=1$, while
\begin{equation*}
\begin{split}
\log |E(-re^{i\theta})|\sim \frac{1}{l}r\cos(\theta) \quad \text{for} \quad |\theta|\leq (1-\varepsilon)\frac{\pi}{2}
\end{split}
\end{equation*}
and
\begin{equation*}
\begin{split}
|A(-re^{i\theta})|\sim \frac{1}{4l^2} \quad \text{for} \quad |\theta|\leq (1-\varepsilon)\frac{\pi}{2}
\end{split}
\end{equation*}
uniformly as $r\to\infty$, for any $\varepsilon\in(0,1)$.
\end{theorem}

The proof for Theorem~\ref{maintheorem4-asym} is only a slight modification of that in \cite[Theorem~3]{Zhang2024} for the case $\rho=\sigma=1$ and thus will be omitted. The proof is based on the fact that the Schwarzian derivative $S(F)$ can be factored as $S(F)=B(F/F')/2$, where
\begin{equation*}
\begin{split}
B(E):=-2\frac{E''}{E}+\left(\frac{E'}{E}\right)^2-\frac{1}{E^2}.
\end{split}
\end{equation*}
See Bank and Laine~\cite{Banklaine1982,Banklaine1982-2}. The general solution of the differential equation $B(E)=4A$, that is, of the equation
\begin{equation}\label{recuree1pm7}
\begin{split}
4A=-2\frac{E''}{E}+\left(\frac{E'}{E}\right)^2-\frac{1}{E^2}
\end{split}
\end{equation}
is a product of two normalized solutions of \eqref{bank-laine0}. The results in \cite[Theorem~3]{Zhang2024} shows that the functions $A$ and $E$ constructed in Theorem~\ref{maintheorem3} are actually of completely regular growth (in the sense of Levin and Pfluger). This is obvious for $E$ and $A$ when $\lambda_1\leq \lambda_2<n$. In fact, by the classical Hadamard factorization theorem, Theorem~\ref{maintheorem4-asym} together with its extension to the general case in Subsection~\ref{Completion of the proof} implies that $E(z)=H(z)e^{p(z)}$, where $H$ is the canonical product formed by the zeros of $E$ and $p$ is a polynomial of degree $n$ of the form $p(z)=z^n/l+\cdots$.

While the Bank--Laine functions $E=f_1f_2$ constructed in previous three theorems all satisfy $\rho(E)<\infty$, by the linear independence of $f_1$ and $f_2$, the product $E_c=f_1(cf_1+f_2)$ with a nonzero constant $c$ is also a Bank--Laine function and satisfies $\lambda(E_c)=\infty$. In the rest of this paper, we construct of Bank--Laine functions $E=f_1f_2$ such that $\rho(E)=\infty$ but $\lambda(f_1)<\infty$. More specifically, we shall first prove the following

\begin{theorem}\label{maintheorem5}
Let $n\in \mathbb{N}$ be a positive integer and $\rho\in (n/2,n)$ and $\lambda\in[0,\infty)$. Then there exists an entire function $A$ of order $\rho(A)=\rho$ such that the associated Bank--Laine function $E=f_1f_2$ for two normalized solutions $f_1$ and $f_2$ of \eqref{bank-laine0} satisfies $\lambda(f_1)=\lambda$ and $\lambda(f_2)=\infty$. Moreover, the Bank--Laine function $E_c=f_1(cf_1+f_2)$ satisfies $\lambda(E_c)=\infty$ for any constant $c$.
\end{theorem}

We see that the Bank--Laine function $E=f_1f_2$ constructed in Theorem~\ref{maintheorem5} satisfies $\lambda(E)=\rho(E)=\infty$. In the proof, we have actually shown that the numbers of zeros of $f_1$ and $f_2$ satisfy $n(r,0,f_1)\asymp r^{\lambda}$ and $\log n(r,0,f_2)\asymp r^{\rho}$ respectively. To prove Theorem~\ref{maintheorem5}, we not only glue the functions $g_{m,n}$ as in the proof of Theorem~\ref{maintheorem2} but also glue two functions $\hat{g}_{m,n}=(g_{m,n}+1)/2$ and $\hat{g}_{m,n}=g_{m,n}$ along horizontal lines for certain pair $(m,2n)$ of two nonnegative integers $m$ and $n$. If we fix $\hat{g}_{m_2,n_2}=g_{m_2,n_2}$ for one pair $(m_2,2n_2)$ in the upper half-plane and $\check{g}_{m_1,n_1}=(g_{m_1,n_1}+1)/2$ for another pair $(m_1,2n_1)$ in the lower half-plane, then we may follow the same process as in the proof of Theorem~\ref{maintheorem3} to obtain: for a dense set of orders $\rho\in(1,\infty)$, there exists an entire function $A$ of order $\rho(A)=\rho$ such that the associated Bank--Laine function $E=f_1f_2$ for two normalized solutions $f_1$ and $f_2$ of \eqref{bank-laine0} satisfies $\lambda(f_1)<\infty$ and $\lambda(f_2)=\infty$. Moreover, the Bank--Laine function $E_c=f_1(cf_1+f_2)$ satisfies $\lambda(E_c)=\infty$ for any nonzero constant $c$.

By the proof of Theorem~\ref{maintheorem5}, the entire function $A$ and the associated Bank--Laine function $E$ inherit all the properties of those in Theorem~\ref{maintheorem2} in the lower half-plane in the case $\rho\in (1/2,1)$, as well as the asymptotic behaviors as in \cite[Theorem~3]{Zhang2024}. In fact, by checking the proof for \cite[Theorem~3]{Zhang2024} together with the definition $\hat{g}_{m,n}=(g_{m,n}+1)/2$, we easily find that these asymptotic behaviors also hold for $z=re^{i\theta}$ such that $|\theta|<(1-\varepsilon)\pi/(2\sigma)$ with $\sigma=\rho/(2\rho-1)$. In the region $\{z=re^{i\theta}: 0\leq \theta\leq\pi/(2\rho)\}$ of the plane, the function $E$ has much complicated behavior. There must be an infinite sequence $(z_n)$, $z_n=r_ne^{i\theta_n}$, $0\leq \theta_n\leq \pi/(2\rho)$ and $r_n\to\infty$, such that $\log |E(r_ne^{i\theta_n})|\geq cr_n^{\rho}$ for some positive constant $c$ and all $r_n$.

Theorem~\ref{maintheorem5} can be viewed as an extension of Theorem~\ref{maintheorem2}. If we consider the existence of Bank--Laine functions such that $\rho(E)=\infty$ but $\lambda(f_1)<\infty$ in the case $A$ has integer order, then a combination of the methods in the proof of Theorem~\ref{maintheorem4} and Theorem~\ref{maintheorem5} yields an extension of Theorem~\ref{maintheorem4}, namely the following

\begin{theorem}\label{maintheorem6}
Let $n\in \mathbb{N}$ be a positive integer and $\lambda\in[0,n]$. Then there exists an entire function $A$ of order $\rho(A)=n$ such that the associated Bank--Laine function $E=f_1f_2$ for two normalized solutions $f_1$ and $f_2$ of \eqref{bank-laine0} satisfies $\lambda(f_1)=\lambda$ and $\lambda(f_2)=\infty$. Moreover, the Bank--Laine function $E_c=f_1(cf_1+f_2)$ satisfies $\lambda(E_c)=\infty$ for any constant $c$.
\end{theorem}

We see that the Bank--Laine functions $E=f_1f_2$ constructed in Theorem~\ref{maintheorem6} satisfies $\lambda(E)=\rho(E)=\infty$. In the proof, we have actually shown that the numbers of zeros of $f_1$ and $f_2$ satisfy $n(r,0,f_1)\asymp (\log r)^2$ or $n(r,0,f_1)\asymp r^{\lambda}$ and $\log n(r,0,f_2)\asymp r^{n}$ respectively. Also, the entire function $A$ and the associated Bank--Laine function $E$ inherit all the properties of those in Theorem~\ref{maintheorem4} in the lower half-plane in the case $\rho(A)=1$, as well as the asymptotic behaviors as in~Theorem~\ref{maintheorem4-asym}.

A question concerning the entire function $A$ in the second order differential equation \eqref{bank-laine0} was asked in \cite[Question~4.5]{Gundersen2017}: \emph{Is it possible to characterize the transcendental entire functions $A$ where \eqref{bank-laine0} admits a non-trivial solution $f_1$ with $\lambda(f_1)<\infty$, such that for every solution $f_2$ of \eqref{bank-laine0} that is linearly independent with $f_1$ we have $\lambda(f_2)=\infty$}? The constructions in Theorem~\ref{maintheorem5} and Theorem~\ref{maintheorem6} present such entire functions $A$ of finite order. The constructions there are quite flexible. For example, as noted in the proof of Theorem~\ref{maintheorem6}, since the number $n(\dot{r},0,g_{m,n}+1)$ of zeros of $g_{m,n}+1$ in only one fixed strip $\Pi_k=\{x+iy: x>0, 2\pi(k-1)<y<2\pi k\}$ satisfies $\log n(\dot{r},0,g_{m,n}+1)\asymp r$, it is enough to fix only one $g_{m,n}+1$ in the upper half-plane in the gluing to guarantee that the resulting Bank--Laine function $E$ satisfies $\lambda(E)=\infty$ in the proof of Theorem~\ref{maintheorem5} and Theorem~\ref{maintheorem6}. We have actually shown that the number $\limsup_{r\to\infty}(\log\log n(r,0,E)/\log r)$ for the Bank--Laine functions $E$ constructed there can also be preassigned.

In the remainder of this paper, we provide the proof for Theorem~\ref{maintheorem3}, Theorem~\ref{maintheorem4} Theorem~\ref{maintheorem5} and Theorem~\ref{maintheorem6} in Section~\ref{Proof of maintheorem3}-Section~\ref{Proof of maintheorem6}, respectively. In the following sections, we keep the permanent notation: $(f_1,f_2)$ is a normalized pair of solutions of \eqref{bank-laine0}, the quotient $F=f_2/f_1$ is a solution of $S(F)=2A$, and $E=f_1f_2=F/F'$ is a solution of \eqref{recuree1pm7}. In different sections, the notations $\phi,\psi,\tau,G$ and $F$ will represent different functions and the notations $M,N$ and $\kappa$ will represent different numbers.

\section{Proof of Theorem~\ref{maintheorem3}}\label{Proof of maintheorem3} 

For every pair $(m,2n)$ of two nonnegative integers $m$ and $n$, we consider the rational function $R_{m,n}$ in \eqref{recuree1pm8}. In this section, we fix two distinct pairs $(m_1,2n_1)$ and $(m_2,2n_2)$ and will sometimes omit them from notation. We shall assume that at least one of $m_1$ and $m_2$ is nonzero and at least one of $n_1$ and $n_2$ is nonzero. Though now $g_{m,n}$ may have some poles, the process of gluing $g_{m_1,n_1}$ and $g_{m_2,n_2}$ is almost the same as that in \cite{Bergweilereremenko2017}: by restricting $g_{m_1,n_1}$ to the lower half-plane $\check{\mathbb{H}}$ and $g_{m_2,n_2}$ to the upper half-plane $\hat{\mathbb{H}}$, we then paste them together, using a quasiconformal surgery, producing a locally univalent meromorphic function $F$. Then our Bank--Laine function will be $E=F/F'$ and thus $A=B(E)/4$ as in \eqref{recuree1pm7}. For our purpose of this paper, we will use some different notations from the ones used in \cite{Bergweilereremenko2017}.

By the properties of $g_{m,n}$ described in the introduction, we see that there exists an increasing diffeomorphism $\phi: \mathbb{R}\to \mathbb{R}$ such that
\begin{equation}\label{recuree-1orqr-1}
g_{m_2,n_2}(x)=(g_{m_1,n_1}\circ\phi)(x)
\end{equation}
for all $x\in \mathbb{R}$. Denote $M=m_2+2n_2+1$ and $N=m_1+2n_1+1$ for simplicity. Let
\begin{equation*}
\begin{split}
\kappa=\frac{M}{N}=\frac{m_2+2n_2+1}{m_1+2n_1+1}.
\end{split}
\end{equation*}
Though $g_{m,n}$ in \eqref{recuree1pm9} has some poles, we may use the arguments as in \cite{Bergweilereremenko2017} to show

\begin{lemma}\label{Lemma 0}
The asymptotic behaviors of the diffeomorphism $\phi$ satisfy:
\begin{eqnarray}
&\phi(x)=x+O(e^{-x/2}),        \quad\quad\quad\quad\quad &\phi'(x)=1+O(e^{-x/2}), \quad\quad x\to \infty,  \label{recuree1pm11}\\
&\phi(x)=\kappa x+c+O(e^{-\delta|x|/2}),    \quad\quad &\phi'(x)=\kappa+O(e^{-\delta|x|/2}),  \quad x\to -\infty,             \label{recuree1pm12}
\end{eqnarray}
with
\begin{equation*}
\begin{split}
c=\frac{1}{N}\log \left[\frac{\binom{m_1+2n_1}{m_1}}{\binom{m_2+2n_2}{m_2}}\frac{N!}{M!}\right] \quad \text{and} \quad \delta=\frac{1}{2}\min\{1,\kappa\}.
\end{split}
\end{equation*}

\end{lemma}

\begin{proof}
In order to prove \eqref{recuree1pm11}, we note that
\begin{equation*}
\begin{split}
\log g_{m_2,n_2}(x)=e^{x}+O(x)=e^{x}\left[1+O(xe^{-x})\right], \quad x\to \infty.
\end{split}
\end{equation*}
The equation $g_{m_2,n_2}(x)=(g_{m_1,n_1}\circ\phi)(x)$ easily implies that $2x/3\leq \phi(x)\leq 2x$ for large $x$. Thus we also have
\begin{equation*}
\begin{split}
\log g_{m_1,n_1}(\phi(x))&=e^{\phi(x)}\left[1+O(\phi(x)e^{-\phi(x)})\right]=e^{\phi(x)}\left[1+O(\phi(x)e^{-2x/3})\right], \quad x\to \infty.
\end{split}
\end{equation*}
Combining the last two equations we obtain
\begin{equation*}
\begin{split}
e^{\phi(x)-x}=1+O(xe^{-2x/3}), \quad x\to \infty,
\end{split}
\end{equation*}
from which the first statement in \eqref{recuree1pm11} easily follows. For the second statement in \eqref{recuree1pm11} we use
\begin{equation}\label{recuree1pm16  bef0}
\begin{split}
\phi'=\frac{g'_{m_2,n_2}}{g_{m_2,n_2}}\frac{g_{m_1,n_1}\circ\phi}{g'_{m_1,n_1}\circ\phi}=e^{Mx-N\phi(x)}\frac{A_{m_2}B_{2n_2}}{A_{m_1}B_{2n_1}}\frac{\sum_{i=0}^{m_1}A_ie^{i\phi(x)}}{\sum_{i=0}^{m_2}A_ie^{ix}}\frac{\sum_{j=0}^{2n_1}B_je^{j\phi(x)}}{\sum_{j=0}^{2n_2}B_je^{jx}},
\end{split}
\end{equation}
so that, together with the first statement in \eqref{recuree1pm11},
\begin{equation}\label{recuree1pm17 bef1}
\begin{split}
\phi'&=e^{Mx-N\phi(x)+(m_1+2n_1)\phi(x)-(m_2+2n_2)x}\frac{1+O(e^{-\phi(x)})}{1+O(e^{-x})}\frac{1+O(e^{-\phi(x)})}{1+O(e^{-x})}\\
&=e^{O(e^{-x/2})}\left[1+O(e^{-\phi(x)})+O(e^{-x})\right]=1+O(e^{-x/2}), \quad x\to \infty.
\end{split}
\end{equation}
In order to prove \eqref{recuree1pm12} we note from \eqref{recuree1pm9} and \eqref{recuree1pm10} that
\begin{equation*}
\begin{split}
h'_{m,n}(w)=\frac{1}{\binom{m+2n}{m}(m+2n)!}\frac{e^ww^{m+2n}}{Q_m(w)^2}.
\end{split}
\end{equation*}
Since $h_{m,n}(w)$ is analytic on $[0,\infty)$, then by taking the derivatives of $h'_{m,n}(w)$ successively we find that $h'_{m,n}(0)=\cdots=h^{(m+2n)}_{m,n}(0)=0$. Thus by Peano's version of Taylor's theorem we have
\begin{equation*}
\begin{split}
h_{m,n}(w)=R_{m,n}(w)e^{w}=1+\frac{w^{m+2n+1}}{\binom{m+2n}{m}(m+2n+1)!}+O(w^{m+2n+2}), \quad w\to 0.
\end{split}
\end{equation*}
Hence
\begin{equation*}
\begin{split}
g_{m,n}(x)&=1+\frac{e^{(m+2n+1)x}}{\binom{m+2n}{m}(m+2n+1)!}+O(e^{(m+2n+2)x})\\
&=1+\frac{e^{(m+2n+1)x}}{\binom{m+2n}{m}(m+2n+1)!}\left[1+O(e^{x})\right], \quad x\to-\infty.
\end{split}
\end{equation*}
The equation $g_{m_2,n_2}(x)=g_{m_1,n_1}(\phi(x))$ now yields
\begin{equation*}
\begin{split}
\frac{\binom{m_2+2n_2}{m_2}}{\binom{m_1+2n_1}{m_1}}\frac{M!}{N!}e^{N\phi(x)-Mx}=1+O(e^{x})+O(e^{\phi(x)}), \quad x\to-\infty
\end{split}
\end{equation*}
and hence
\begin{equation*}
\begin{split}
\phi(x)=\frac{M}{N}x+\frac{1}{N}\log \left[\frac{\binom{m_1+2n_1}{m_1}}{\binom{m_2+2n_2}{m_2}}\frac{N!}{M!}\right]+O(e^{x})+O(e^{\phi(x)}), \quad x\to-\infty,
\end{split}
\end{equation*}
which gives the first statement in \eqref{recuree1pm12}. For the second statement in \eqref{recuree1pm12}, recall that $A_{m}B_{2n}=\frac{m!(2n)!}{[(m+2n)!]^2}=\frac{1}{\binom{m+2n}{m}(m+2n)!}$. Similarly as for the estimates in \eqref{recuree1pm17 bef1}, \eqref{recuree1pm16  bef0} now takes the form
\begin{equation*}
\begin{split}
\phi'(x)=\frac{A_{m_2}B_{2n_2}}{A_{m_1}B_{2n_1}}e^{Mx-N\phi(x)}\frac{1+O(e^{\phi(x)})}{1+O(e^{x})}\frac{1+O(e^{\phi(x)})}{1+O(e^{x})}
\end{split}
\end{equation*}
and thus, by the first statement in \eqref{recuree1pm12}, we obtain
\begin{equation*}
\begin{split}
\phi'(x)&=\frac{A_{m_2}B_{2n_2}}{A_{m_1}B_{2n_1}}e^{Mx-N(\kappa x+c+O(e^{-\delta|x|/2}))}\left[1+O(e^{\phi(x)})+O(e^{x})\right]\\
&=\frac{A_{m_2}B_{2n_2}}{A_{m_1}B_{2n_1}}e^{-Nc+O(e^{-\delta|x|/2})}\left[1+O(e^{\phi(x)})+O(e^{x})\right]=\kappa+O(e^{-\delta|x|/2}).
\end{split}
\end{equation*}

\end{proof}

Now we begin to define a quasiregular function in the same way as in \cite{Bergweilereremenko2019}. For the convenience of calculating the number of zeros of the desired Bank--Laine functions in the final part of the proof and also of the applications of the arguments in the later proof of Theorem~\ref{maintheorem4}-Theorem~\ref{maintheorem6}, below we shall keep as many as details.

Let $D=\mathbb{C}\setminus\mathbb{R}_{\leq 0}$, and $p:D\to \mathbb{C}$, $p(z)=z^{\mu}$, the principal branch of the power. Here $\mu$ is a complex number such that
\begin{equation}\label{recuree1pm23}
\begin{split}
\mu=\text{Re} (\mu)+i\text{Im}(\mu)=\frac{2\pi}{4\pi^2+\log^2\kappa}(2\pi-i\log \kappa).
\end{split}
\end{equation}
We note that it is possible that $\kappa=1$ with the two pairs $(m_1,2n_1)$ and $m_2,2n_2$. As shown in \cite{Bergweilereremenko2019}, $p$ maps $D$ onto the complement $\mathbb{G}$ of a logarithmic spiral $\Gamma$, with
\begin{equation}\label{recuree1pm24  bef1}
\begin{split}
p(x+i0)=p(\kappa x-i0), \quad  x<0.
\end{split}
\end{equation}
In particular, when $\kappa=1$, $p$ is just the identity in the plane and $\Gamma=\mathbb{R}_{\leq 0}$. The inverse map $h=p^{-1}$ is a conformal homeomorphism $h:\mathbb{G}\to D$. Let $\Gamma'=p(\mathbb{R}_{\geq 0})$. The two logarithmic spirals $\Gamma$ and $\Gamma'$ divide the plane into two parts $\hat{\mathbb{G}}$ and $\check{\mathbb{G}}$, which are images under $p$ of the upper and lower half-planes, respectively.

The function $V$ defined by
\begin{equation*}
V(z)=\left\{
      \begin{array}{ll}
      (g_{m_2,n_2}\circ h)(z), & \quad z\in \hat{\mathbb{G}}, \\
      (g_{m_1,n_1}\circ h)(z), & \quad z\in \check{\mathbb{G}}
      \end{array}
    \right.
\end{equation*}
is meromorphic in $\hat{\mathbb{G}}\cup \check{\mathbb{G}}$ and has a jump discontinuity on $\Gamma$ and $\Gamma'$. To remove this discontinuity, we consider the strip $\Pi=\{z: |\text{Im}\ z|<1\}$ and define a quasiconformal homeomorphism $\psi: \mathbb{C}\to \mathbb{C}$ to be the identity outside of $\Pi$ and for $y=\text{Im}\ z \in (-1,1)$,
\begin{equation}\label{recuree1pm27  bef2}
\psi(x+iy)=\left\{
           \begin{array}{ll}
           \phi(x)+|y|(x-\phi(x))+iy,     & \quad x\geq 0, \\
           \phi(x/\kappa)+|y|(x-\phi(x/\kappa)+iy), & \quad x<0.
           \end{array}
           \right.
\end{equation}
Then $\psi$ commutes with the complex conjugation and satisfies
\begin{equation}\label{recuree1pm28}
\begin{split}
\psi(x)=\phi(x), \quad x>0 \quad \text{and} \quad \psi(\kappa x)=\phi(x), \quad x<0.
\end{split}
\end{equation}
The Jacobian matrix $J_{\psi}$ of $\psi$ is given for $x>0$ and $0<|y|<1$ by
\begin{equation*}
J_{\psi}(x+iy)=\left(
                 \begin{array}{cc}
                 \phi'(x)+|y|(1-\phi(x)) & \pm(x-\phi(x)) \\
                 0 & 1 \\
                 \end{array}
               \right),
\end{equation*}
and we see using \eqref{recuree1pm11} that
\begin{equation*}
J_{\psi}(x+iy)=\left(
                 \begin{array}{cc}
                 1 & 0 \\
                 0 & 1 \\
                 \end{array}
               \right), \quad 0<|y|<1, \quad x\to \infty.
\end{equation*}
Similarly, using \eqref{recuree1pm12} we find that
\begin{equation*}
J_{\psi}(x+iy)=\left(
                 \begin{array}{cc}
                 1 & \mp c \\
                 0 & 1 \\
                 \end{array}
               \right), \quad 0<|y|<1, \quad x\to-\infty.
\end{equation*}
We conclude that $\psi$ is quasiconformal in the plane.

Now we modify $V$ to obtain a continuous function. Define $U:\mathbb{C}\to \mathbb{C}$ by
\begin{equation}\label{recuree1pm32}
U(z)=\left\{
      \begin{array}{ll}
      (g_{m_2,n_2}\circ h)(z),            & \quad z\in \hat{\mathbb{G}}\cup\Gamma\cup\Gamma'\cup\{0\}, \\
      (g_{m_1,n_1}\circ \psi \circ h)(z), & \quad z\in \check{\mathbb{G}}.
      \end{array}
    \right.
\end{equation}
It follows from \eqref{recuree1pm24  bef1} and \eqref{recuree1pm27  bef2} that $U$ is continuous on $\Gamma$ and $\Gamma'$ and quasiregular in the plane. The existence theorem for solutions of the Beltrami equation \cite[{\S}~V.1]{lehtoVirtanen2008} shows that there exists a quasiconformal homeomorphism $\tau: \mathbb{C}\to \mathbb{C}$ with the same Beltrami coefficient as $U$. The function $F=U\circ\tau^{-1}$ is then meromorphic.
Moreover, as shown in \cite{Bergweilereremenko2017}, the Beltrami coefficient of $U$ (and hence of $\tau$) satisfies the hypotheses of the Teichm\"uller--Wittich--Belinskii theorem~\cite[\S~V.6]{lehtoVirtanen2008} (see also~\cite{Shishikura2018}). This theorem shows that $\tau$ is conformal at $\infty$ and may thus be normalized to satisfy
\begin{equation}\label{recuree1pm34}
\begin{split}
\tau(z)\sim z, \quad z\to\infty.
\end{split}
\end{equation}
We set $E=F/F'$. As $F'(z)\not=0$ for all $z\in \mathbb{C}$ by construction, $E$ is entire. As all poles and zeros of $F$ are simple, all residues of $F/F'$ are equal to $1$ or $-1$, so $E'(z)=1$ or $E'(z)=-1$ at every zero of $E$, which implies the Bank--Laine property. Clearly, $E=f_1f_2$ for two normalized solutions of the second order differential equation \eqref{bank-laine0} with an entire coefficient $A$.

We shall use a slightly different method from that in~\cite{Bergweilereremenko2017} to show that $E$, as well as the entire function $A$, has order $\rho=1/\text{Re}(\mu)$. Below we shall first estimate the numbers of poles and zeros of $F$ respectively.

For the meromorphic function $g_{m_2,n_2}$, if $m_2\not=0$, then the number of poles of $g_{m_2,n_2}$ satisfies $n(r,\infty,g_{m_2,n_2})\sim m_2r/\pi$ as $r\to\infty$. Thus, by the periodicity of $g_{m_2,n_2}$, the number of poles of $g_{m_2,n_2}$ in the upper half-pane $\hat{\mathbb{H}}$ satisfies $\hat{n}(r,\infty,g_{m_2,n_2})\sim m_2r/(2\pi)$ as $r\to\infty$. We note that the image of the circle $\{z:|z|=r\}$ under $h(z)=z^{1/\mu}$ is the part of the logarithmic spiral which connects two points on the negative real axis and intersects the positive real axis at $r^{1/\text{Re}(\mu)}$. By the definitions of $p$ and $h$ we may write $z=re^{i\theta}$ on the circle such that $z_0=re^{i\theta_0}$ is a point on $\Gamma'$ and
\begin{equation*}
\begin{split}
h(z)=\exp\left(\frac{1}{\mu}\log z\right)=\exp\left(\frac{(\text{Im}(\mu)) \theta+(\text{Re}(\mu))\log r}{|\mu|^2}+i\frac{(\text{Re}(\mu))\theta-(\text{Im}(\mu))\log r}{|\mu|^2}\right).
\end{split}
\end{equation*}
When $\kappa>1$, noting that $\text{Re}(\mu)>0$ and $\text{Im}(\mu)<0$, we may write $-(\theta_0+\pi)\leq \theta\leq -(\theta_0-\pi)$, where $\theta_0$ satisfies $(\text{Re}(\mu))\theta_0=(\text{Im}(\mu)) \log r$. A simple computation shows that
\begin{equation}\label{recuree1pm47}
\begin{split}
e^{(\text{Im}(\mu))\pi}r^{1/\text{Re}(\mu)}\leq |h(re^{i\theta})|\leq e^{-(\text{Im}(\mu))\pi}r^{1/\text{Re}(\mu)}.
\end{split}
\end{equation}
Therefore, when $r$ is large, the number of poles of $g_{m_2,n_2}\circ h$ in $\hat{\mathbb{G}}$ satisfies
\begin{equation*}
\begin{split}
c_2r^{1/\text{Re}(\mu)} \leq \hat{n}(r,\infty,g_{m_2,n_2}\circ h)\leq c_1r^{1/\text{Re}(\mu)}
\end{split}
\end{equation*}
for two positive constants $c_1$ and $c_2$ both dependent on $\kappa$. Similarly, if $m_1\not=0$, the number of poles of $g_{m_1,n_1}\circ h$ in $\check{\mathbb{G}}$ satisfies
\begin{equation}\label{recuree1pm48}
\begin{split}
c_4r^{1/\text{Re}(\mu)} \leq \check{n}(r,\infty,g_{m_1,n_1}\circ h) \leq c_3 r^{1/\text{Re}(\mu)}
\end{split}
\end{equation}
for two positive constants $c_3$ and $c_4$ both dependent on $\kappa$. By the definition of $\tau(x+iy)$ in \eqref{recuree1pm27  bef2} and the asymptotic behaviors for $\phi$ in \eqref{recuree1pm11} and \eqref{recuree1pm12}, we see that the number of poles of $g_{m_1,n_1}\circ \psi \circ h$ in $\check{\mathbb{G}}$ satisfies
\begin{equation}\label{recuree1pm49}
\begin{split}
c_6r^{1/\text{Re}(\mu)} \leq \check{n}(r,\infty,g_{m_1,n_1}\circ \psi \circ h) \leq c_5r^{1/\text{Re}(\mu)}
\end{split}
\end{equation}
for two positive constants $c_5$ and $c_6$ both dependent on $\kappa$. When $\kappa=1$ or $\kappa<1$, we can also obtain two similar estimates as in \eqref{recuree1pm48} and \eqref{recuree1pm49}. By summarizing the above estimates together with the relation in \eqref{recuree1pm34}, we conclude that $n(r,\infty,F)\asymp r^{1/\text{Re}(\mu)}$. It follows that $N(r,\infty,F)\asymp r^{1/\text{Re}(\mu)}$. Similarly, when at least one of $n_1$ and $n_2$ is nonzero, for the zeros of $F$ we also obtain $n(r,0,F)\asymp r^{1/\text{Re}(\mu)}$ and thus $N(r,0,F)\asymp r^{1/\text{Re}(\mu)}$. Since $E=F/F'$ and since the poles and zeros of $F$ are simple, we have
\begin{equation}\label{recuree1pm50}
\begin{split}
N\left(r,0,E\right)\asymp r^{1/\text{Re}(\mu)}.
\end{split}
\end{equation}

Next we estimate $m(r,1/E)$. We note that there is a set $\Omega_1$ of finite linear measure such that, for all $|z|\not\in \Omega_1$, the functions $P_{n_i}(e^{z})$ and $Q_{m_i}(e^z)$ both satisfy
\begin{equation*}
    \left\{
      \begin{array}{ll}
      \exp(-(2n_i+\varepsilon)|z|)\leq \left|P_{n_i}(e^{z})\right|\leq \exp((2n_i+\varepsilon)|z|),&\\
      \exp(-(m_i+\varepsilon)|z|)\leq \left|Q_{m_i}(e^{z})\right|\leq \exp((m_i+\varepsilon)|z|),&
      \end{array}
    \right.
\end{equation*}
where $i=1,2$ and $\varepsilon>0$ is a small constant. This implies that
\begin{equation*}
\begin{split}
|g_{m_i,n_i}(z)|\leq \exp((m_i+\varepsilon)|z|+(2n_i+\varepsilon)|z|)\left|\exp(e^z)\right|
\end{split}
\end{equation*}
for $i=1,2$ and all $z\in \mathbb{C}$ such that $|z|\not\in \Omega_1$. Denote $L_i=m_i+2n_i$, $i=1,2$, for simplicity. Clearly, this implies that there is a set $\Omega_2$ of finite linear measure such that
\begin{equation}\label{recuree1pm53}
    \left\{
      \begin{array}{ll}
      |g_{m_2,n_2}(z)|\leq \exp((L_2+2\varepsilon)|z|)\left|\exp\left(e^{z}\right)\right| \quad \text{for} \quad z\in \hat{\mathbb{H}} \quad \text{and} \quad |z|\not\in \Omega_2, & \\
      |g_{m_1,n_1}(z)|\leq \exp((L_1+2\varepsilon)|z|)\left|\exp\left(e^{z}\right)\right| \quad \text{for} \quad z\in \check{\mathbb{H}} \quad \text{and} \quad |z|\not\in \Omega_2. & \\
      \end{array}
    \right.
\end{equation}
We deduce from \eqref{recuree1pm32}, \eqref{recuree1pm47} and \eqref{recuree1pm53} that there is a set $\Omega_3$ of finite linear measure such that
\begin{equation*}
\begin{split}
\log|g_{m_2,n_2}\circ h| \leq c(L_2+2\varepsilon)|z|^{1/\text{Re}(\mu)}+\exp(c|z|^{1/\text{Re}(\mu)}) \quad \text{for} \quad  z\in \hat{\mathbb{G}} \quad \text{and} \quad  |z|\not\in \Omega_3,
\end{split}
\end{equation*}
where $c$ is a suitably chosen positive constant and, together with the asymptotic behaviors for $\phi$ in \eqref{recuree1pm11} and \eqref{recuree1pm12} and also the definition $\psi(x+iy)$ in \eqref{recuree1pm28}, that
\begin{equation*}
\begin{split}
\log|g_{m_1,n_1} \circ \psi \circ h| \leq c(L_1+2\varepsilon)|z|^{1/\text{Re}(\mu)}+\exp(c|z|^{1/\text{Re}(\mu)}) \quad \text{for}\quad  z\in \check{\mathbb{G}} \quad  \text{and} \quad |z|\not\in \Omega_3.
\end{split}
\end{equation*}
The above two estimates and~\eqref{recuree1pm34} yield
\begin{equation}\label{recuree1pm57}
\begin{split}
\log|F(z)|\leq \exp((c+o(1))|z|^{1/\text{Re}(\mu)})
\end{split}
\end{equation}
as $|z|\to\infty$ outside a set $\Omega_3$ of finite linear measure. Together with the relation in \eqref{recuree1pm57}, the lemma on the logarithmic derivative~\cite[Section~2.2]{Hayman1964Meromorphic} now implies that $E=F/F'$ satisfies
\begin{equation}\label{recuree1pm58}
\begin{split}
m\left(r,\frac{1}{E}\right)=O(r^{1/\text{Re}(\mu)}),
\end{split}
\end{equation}
outside a set of finite linear measure. Thus, by \eqref{recuree1pm50} and the first main theorem of Nevanlinna theory \cite[Section~1.3]{Hayman1964Meromorphic}, the characteristic function $T(r,E)$ of $E$ satisfies
\begin{equation*}
\begin{split}
T\left(r,E\right)\asymp r^{1/\text{Re}(\mu)}.
\end{split}
\end{equation*}
Therefore, $E=f_1f_2$ is a Bank--Laine function such that $\lambda(E)=\rho(E)=1/\text{Re}(\mu)=\rho$, as desired. Moreover, the two normalized solutions $f_1$ and $f_2$ of the second order linear differential equation \eqref{bank-laine0} both satisfy $n(r,0,f_i)\asymp r^{1/\text{Re}(\mu)}$, $i=1,2$.

Finally, we determine the order of the entire function $A$. Since $E$ is of finite order, it follows by the relation \eqref{recuree1pm7} together with \eqref{recuree1pm58} and the lemma on the logarithmic derivative that
\begin{equation*}
\begin{split}
m(r,A)=m\left(r,\frac{1}{E}\right)+O(\log r)=O(r^{1/\text{Re}(\mu)}).
\end{split}
\end{equation*}
Thus the first main theorem implies that $\rho(A)\leq 1/\text{Re}(\mu)$. On the other hand, we note that the function $U$ defined in \eqref{recuree1pm32} is meromorphic in $\hat{\mathbb{G}}$ and, in particular, is analytic on $\Gamma$ and $\Gamma'$ respectively. By taking the derivative of $F$ in $\tau(\hat{\mathbb{G}})$, we get
\begin{equation}\label{recuree1pm59-h}
\begin{split}
\frac{F'}{F}=\left(\frac{g'_{m_2,n_2}}{g_{m_2,n_2}}\circ h \circ \tau^{-1}\right)(h'\circ \tau^{-1})(\tau^{-1})'.
\end{split}
\end{equation}
We may choose a curve in $\tau(\hat{\mathbb{G}})$, say $\Gamma''$, so that $\tau^{-1}(\Gamma'')$ is close to the curve $\Gamma'$. For a point $z\in\tau^{-1}(\Gamma'')$, we have the Cauchy formula
\begin{equation*}
\begin{split}
\tau'(z)=\frac{1}{2\pi i}\int_{C_z}\frac{\tau(\zeta)}{(\zeta-z)^2}d\zeta
\end{split}
\end{equation*}
with a circle $C_z$ centered at $z$. Choosing the radius $\beta(z)$ of this circle to be finite, say $\beta(z)=1$, we have $|\tau'(z)|\leq 2|z|$ when $|z|$ is large. Since $E=F/F'$, then by the definition of $g_{m,n}$ in \eqref{recuree1pm9} as well as the relations in \eqref{recuree1pm34} and in \eqref{recuree1pm59-h}, we see that $1/E$ satisfies
\begin{equation}\label{recuree1pm59-h-1}
\begin{split}
\frac{1}{|E(z)|}\geq \left(d_1e^{\frac{1}{2}|z|^{1/\text{Re}(\mu)}}\right)\left(\frac{d_2|z|^{1/\text{Re}(\mu)-1}}{|\mu|}\right)\frac{1}{2|z|}
\end{split}
\end{equation}
for two positive constants $d_1,d_2$ and all $z\in\Gamma''$ such that $|z|$ is large.
Further, since $E$ is of finite order, by an estimate in \cite[Corollary~2]{gundersen:88} we may suppose that, for a small $\varepsilon>0$, there is an infinite sequence $(z_n)$ such that $z_n\in\Gamma''$ and $|E'(z_n)/E(z_n)|\leq |z_n|^{1/\text{Re}(\mu)-1+\varepsilon}$ and $|E''(z_n)/E(z_n)|\leq |z_n|^{2(1/\text{Re}(\mu)-1+\varepsilon)}$ for all $z_n$ such that $|z_n|$ is large. Together with these two estimates as well as the inequality in \eqref{recuree1pm59-h-1}, the relation in \eqref{recuree1pm7} implies that $\log |A(z_n)|\geq d_3|z_n|^{1/\text{Re}(\mu)}$ for some positive constant $d_3$ and all $z_n\in\Gamma''$ such that $|z_n|$ is large, implying that $\rho(A)\geq 1/\text{Re}(\mu)$. Hence $\rho(A)=1/\text{Re}(\mu)$. This completes the proof of Theorem~\ref{maintheorem3}.

\section{Proof of Theorem~\ref{maintheorem4}}\label{Proof of maintheorem4} 

\subsection{Preliminary lemmas}\label{Preliminary lemmas} 

We first prove a technical lemma which will be used to determine the number of zeros of the desired Bank--Laine functions. In this subsection, we restrict to consider the case $\lambda_1\leq \lambda_2<1$. We shall choose two sequences $(m_k)$ and $(n_k)$ associated with $\lambda_1$ and $\lambda_2$, respectively.

\begin{lemma}\label{Lemma a}
Let $\lambda\in[0,1)$. Then there is a sequence $(m_k)$ of integers~0 and~1 only, with $m_k=0$ for $k=1,2,3$, such that the function $h:[0,\infty)\to[0,\infty)$ which satisfies $h(0)=0$ and which is linear in the intervals $[2\pi (k-1), 2\pi k]$ and has slopes $m_k$ there satisfies
\begin{equation}\label{recuree1}
\begin{split}
h(x)=(\log x)^2+O(1),  \quad x\to\infty
\end{split}
\end{equation}
when $\lambda=0$, or
\begin{equation}\label{recuree2}
\begin{split}
h(x)=x^{\lambda}+O(1), \quad x\to\infty
\end{split}
\end{equation}
when $0<\lambda<1$. Moreover, when $\lambda=0$, the subsequence $(m_{k_i})$ such that $m_{k_i}=1$ for all $k_i$ satisfies
\begin{equation}\label{recuree3}
\begin{split}
2\pi k_{i}=\left[1+O\left(1/k_i\right)\right]e^{(2\pi i)^{1/2}}, \quad i\to\infty
\end{split}
\end{equation}
and, when $0<\lambda<1$, the subsequence $(m_{k_i})$ such that $m_{k_i}=1$ for all $k_i$ satisfies
\begin{equation}\label{recuree4}
\begin{split}
2\pi k_{i}=(2\pi i)^{1/\lambda}+O(1), \quad i\to\infty.
\end{split}
\end{equation}

\end{lemma}

\begin{proof}
The statement in \cite[Lemma~2.2]{Zhang2024} contains an error, so here we present a complete proof for the lemma. Moreover, we will construct the sequence $(m_k)$ including infinitely many elements $1$ for the case $\lambda=0$.

When $0\leq \lambda < 1$, we set $h(0)=0$ and choose $m_1=m_2=m_3=0$. We note that
\begin{equation*}
\begin{split}
(\log(2\pi k))^2-(\log(2\pi(k-1)))^2 \leq 2\pi \quad \text{for} \quad k\geq 1
\end{split}
\end{equation*}
and
\begin{equation*}
\begin{split}
(2\pi k)^{\lambda}-(2\pi(k-1))^{\lambda} \leq 2\pi \quad \text{for} \quad k\geq 1
\end{split}
\end{equation*}
when $0<\lambda<1$. Suppose that $k\geq 3$, and $h(2\pi k_{i-1})$ is already defined. Then we define
\begin{equation}\label{recuree5-bef}
\begin{split}
h(2\pi k_{i}):=2\pi i,
\end{split}
\end{equation}
where $k_{i}$ is the largest positive integer such that
\begin{equation}\label{recuree5}
\begin{split}
[\log(2\pi(k_{i}-1))]^2<2\pi i\leq [\log(2\pi k_{i})]^2
\end{split}
\end{equation}
when $\lambda=0$, and the largest positive integer such that
\begin{equation}\label{recuree6}
\begin{split}
(2\pi(k_{i}-1))^{\lambda}<2\pi i\leq (2\pi k_{i})^{\lambda}
\end{split}
\end{equation}
when $0<\lambda<1$. Then we interpolate $h$ linearly between $2\pi (k_{i}-1)$ and $2\pi k_i$ in the way that the slopes are equal to $m_{k_i}=1$ in the intervals $[2\pi (k_i-1),2\pi k_i]$ and the slopes are equal to $m_{k}=0$ otherwise. In particular, when $m_{k_i}=1$, since $2\pi(k_i-1)<x<2\pi k_i$ we have $x=2\pi k_i+O(1)$. By the relation $\log(1+x)\to x$ as $x\to 0$, we see that the two estimates in \eqref{recuree3} and \eqref{recuree4} follow from \eqref{recuree5} and \eqref{recuree6}, respectively. Moreover, by the choice of $(m_k)$, we easily see that $h(x)=(\log x)^2+O(1)$ when $\lambda=0$ and also that $h(x)=x^{\lambda}+O(1)$ when $0<\lambda<1$. Thus the assertions on $h$ in \eqref{recuree1} and \eqref{recuree2} follow.

\end{proof}

If we choose the sequence $(m_k)$ in Lemma~\ref{Lemma a} to be such that only finitely many of the elements $m_k$ are equal to~1, then $f_1$ or $f_2$ in the constructed Bank--Laine function $E=f_1f_2$ has only finitely many zeros.

Let $\lambda_1,\lambda_2\in[0,1)$ be two numbers such that $\lambda_1\leq \lambda_2$. By Lemma~\ref{Lemma a}, for the number $\lambda_1\in[0,1)$, there is a sequence $(m_k)$ of $0$ and $1$ only and a function $h_1:[0,\infty)\to[0,\infty)$ which satisfies $h_1(0)=0$ and which is linear in the intervals $[2\pi (k-1), 2\pi k]$ and has slopes $m_k$ there satisfies
\begin{equation*}
\begin{split}
h_1(x)=(\log x)^2+O(1),  \quad x\to\infty
\end{split}
\end{equation*}
when $\lambda_1=0$, and satisfies
\begin{equation*}
\begin{split}
h_1(x)=x^{\lambda_1}+O(1), \quad x\to\infty
\end{split}
\end{equation*}
when $0<\lambda_1<1$. Moreover, when $\lambda_1=0$, the subsequence $(m_{k_i})$ such that $m_{k_i}=1$ for all $k_i$ satisfies
\begin{equation*}
\begin{split}
2\pi k_{i}=\left[1+O\left(1/k_i\right)\right]e^{(2\pi i)^{1/2}}, \quad i\to\infty
\end{split}
\end{equation*}
and, when $0<\lambda_1<1$, the subsequence $(m_{k_i})$ such that $m_{k_i}=1$ for all $k_i$ satisfies
\begin{equation*}
\begin{split}
2\pi k_{i}=(2\pi i)^{1/\lambda_1}+O(1), \quad i\to\infty.
\end{split}
\end{equation*}
Similarly, for the number $\lambda_2\in[0,1)$, there is a sequence $(n_k)$ of $0$ and $1$ only and a function $h_2:[0,\infty)\to[0,\infty)$ which satisfies $h_2(0)=0$ and which is linear in the intervals $[2\pi (k-1), 2\pi k]$ and has slopes $2n_k$ there satisfies
\begin{equation*}
\begin{split}
h_2(x)=2(\log x)^2+O(1),  \quad x\to\infty
\end{split}
\end{equation*}
when $\lambda_2=0$, and satisfies
\begin{equation*}
\begin{split}
h_2(x)=2x^{\lambda_2}+O(1), \quad x\to\infty
\end{split}
\end{equation*}
when $0<\lambda_2<1$. Moreover, when $\lambda_2=0$, the subsequence $(n_{k_j})$ such that $n_{k_j}=1$ for all $k_j$ satisfies
\begin{equation*}
\begin{split}
2\pi k_{j}=\left[1+O\left(1/k_j\right)\right]e^{(2\pi j)^{1/2}}, \quad j\to\infty
\end{split}
\end{equation*}
and, when $0<\lambda_2<1$, the subsequence $(n_{k_j})$ such that $n_{k_j}=1$ for all $k_j$ satisfies
\begin{equation*}
\begin{split}
2\pi k_{j}=(2\pi j)^{1/\lambda_2}+O(1), \quad j\to\infty.
\end{split}
\end{equation*}
For each pair $(m_k,2n_k)$, we denote $N_k=m_k+2n_k+1$. Then, for the two sequences $(m_k)$ and $(n_k)$, a simple calculation shows that the function $H:[0,\infty)\to[0,\infty)$ which satisfies $H(0)=0$ and which is linear in the intervals $[2\pi (k-1), 2\pi k]$ and has slopes $N_k$ there satisfies
\begin{equation}\label{recuree1-Big}
\begin{split}
H(x)=x+\omega(x)
\end{split}
\end{equation}
where $\omega:[0,\infty)\to[0,\infty)$ is a function satisfying $\omega(x) \sim 3(\log x)^2$ as $x\to\infty$ when $\lambda_1=\lambda_2=0$ or $\omega(x) \sim 2x^{\lambda_2}$ as $x\to\infty$ when $\lambda_1<\lambda_2$ or $\omega(x) \sim 3x^{\lambda_2}$ as $x\to\infty$ when $\lambda_1=\lambda_2>0$. We see that there is a function $g:[0,\infty)\to[0,\infty)$ which satisfies $g(0)=0$ and which is linear in the intervals $[2\pi (k-1), 2\pi k]$ such that
\begin{equation*}
\begin{split}
H(g(x))=x.
\end{split}
\end{equation*}
Such a function $g$ exists and, by the estimates in \eqref{recuree1-Big}, we see that $g(x)=x(1+o(1))$ as $x\to\infty$. A more canful calculation shows that the distance between the curve determined by the function $y=g(x)$ and the curve determined by the function $y=x$ tends to infinity as $x\to\infty$. This makes it impossible to define a quasiregular function in the way as in~\cite[Subsection~2.2]{Zhang2024}.

Denote $\mathcal{N}_k=\sum_{j=1}^{k}N_j$ and $\mathcal{N}_0=0$. We shall modify the previous functions $h_1$, $h_2$ and $H$ in the intervals $[2\pi \mathcal{N}_{k-1},2\pi \mathcal{N}_k]$ to be linear there and has slops $m_{k}/N_{k}$, $2n_{k}/N_{k}$ and $1$, respectively. Then, for the constant $\lambda_1\in[0,1)$ and the associated sequence $(m_k)$, there is a function $H_1:[0,\infty)\to[0,\infty)$ which satisfies $H_1(0)=0$ and which is linear in the intervals $[2\pi \mathcal{N}_{k-1},2\pi\mathcal{N}_k]$ and has slopes $m_k/N_k$ there satisfies
\begin{equation*}
\begin{split}
H_1(x)=h_1(x)+O(1),  \quad x\to\infty.
\end{split}
\end{equation*}
Moreover, when $\lambda_1=0$, the subsequence $(m_{k_i})$ such that $m_{k_i}=1$ for all $k_i$ satisfies
\begin{equation}\label{recuree3-a3-a3}
\begin{split}
c_{2}e^{(2\pi i)^{1/2}}\leq 2\pi k_{i}\leq c_{1}e^{(2\pi i)^{1/2}}
\end{split}
\end{equation}
for two positive constants $c_{1},c_{2}$ and, when $0<\lambda_1<1$, the subsequence $(m_{k_i})$ such that $m_{k_i}=1$ for all $k_i$ satisfies
\begin{equation}\label{recuree4-a4-a4}
\begin{split}
c_{4}(2\pi i)^{1/\lambda_1}\leq 2\pi k_{i}\leq c_{3}(2\pi i)^{1/\lambda_1}
\end{split}
\end{equation}
for two positive constants $c_{3},c_{4}$. Also, for the constant $\lambda_2\in[0,1)$ and the associated sequence $(n_k)$, there is a function $H_2:[0,\infty)\to[0,\infty)$ which satisfies $H_2(0)=0$ and which is linear in the intervals $[2\pi \mathcal{N}_{k-1},2\pi\mathcal{N}_k]$ and has slopes $2n_k/N_k$ there satisfies
\begin{equation*}
\begin{split}
H_2(x)=h_2(x)+O(1),  \quad x\to\infty.
\end{split}
\end{equation*}
Moreover, when $\lambda_2=0$, the subsequence $(n_{k_j})$ such that $n_{k_j}=1$ for all $k_j$ satisfies
\begin{equation}\label{recuree3-b3-b3}
\begin{split}
c_{6}e^{(2\pi j)^{1/2}}\leq 2\pi k_{j}\leq c_{5}e^{(2\pi j)^{1/2}}
\end{split}
\end{equation}
for two positive constants $c_{5},c_{6}$ and, when $0<\lambda_2<1$, the subsequence $(n_{k_j})$ such that $n_{k_j}=1$ for all $k_j$ satisfies
\begin{equation}\label{recuree4-b4-b4}
\begin{split}
c_{8}(2\pi j)^{1/\lambda_2}\leq 2\pi k_{j}\leq c_{7}(2\pi j)^{1/\lambda_2}
\end{split}
\end{equation}
for two positive constants $c_{7},c_{8}$.

Let now $m,n,\mathfrak{m},\mathfrak{n}\in \mathbb{N}$ and $N=m+2n+1$ and $M=\mathfrak{m}+2\mathfrak{n}+1$. Let
\begin{equation*}
\begin{split}
\kappa=\frac{M}{N}=\frac{\mathfrak{m}+2\mathfrak{n}+1}{m+2n+1}.
\end{split}
\end{equation*}
Then we need to study the asymptotic behaviors of the increasing diffeomorphism  $\phi:\mathbb{R}\to\mathbb{R}$ defined by
\begin{equation}\label{diffeo  fajr  a}
\begin{split}
g_{\mathfrak{m},\mathfrak{n}}(x)=(g_{m,n}\circ\phi)(x),
\end{split}
\end{equation}
which has appeared in \eqref{recuree-1orqr-1}. We will consider the functions $\phi$ for the case that $m=m_k$, $n=n_k$, $\mathfrak{m}=m_{k+1}$ and $\mathfrak{n}=n_{k+1}$ and also that $N=N_k$ and $M=N_{k+1}$. Note that $N_k\leq 4$ for all $k$ and thus $\kappa=\kappa(k)$ is bounded for all $k$. Thus the asymptotic behaviors described in Lemma~\ref{Lemma 0} apply to all $\phi_k$ such that $g_{m_{k+1},n_{k+1}}(x)=g_{m_{k},n_{k}}(\phi_k(x))$. In order to simplify the formulas, we will sometimes suppress the dependence of $\phi$ from $m,n$ and $\mathfrak{m},\mathfrak{n}$ from the notation.

For technical reasons, we will not work with $g_{m,n}$ defined in \eqref{recuree1pm9} directly. Recall that by \eqref{recuree1pm10} the function $g_{m,n}:\mathbb{R}\to(1,\infty)$ is an increasing homeomorphism. Thus there exists $s_{m,n}\in \mathbb{R}$ such that
\begin{equation}\label{diffeo  fajr  a-1}
\begin{split}
g_{m,n}(s_{m,n})=2.
\end{split}
\end{equation}
For each pair of $(m,2n)$, as in \cite{Bergweilereremenko2019} we shall work with the function $g_{m,n}(z+s_{m,n})$. Since $N\leq 4$ in any case, we have the following

\begin{lemma}\label{Lemma b}
The constant $s_{m,n}$ in \eqref{diffeo  fajr  a-1} satisfies $s_{m,n}=O(1)$ as $k\to\infty$.
\end{lemma}

\begin{proof}
The proof for large $N$ has been given in \cite[Lemma~2.4]{Zhang2024}. We write $s_{m,n}=\log N+s$. Then equation \eqref{recuree1pm9} yields
\begin{equation}\label{trans 2-a-1}
\begin{split}
\log\left(\sum_{j=0}^{2n}B_jN^je^{js}\right)+\log\left(\sum_{i=0}^{m}A_iN^ie^{is}\right)+Ne^{s}=\log2.
\end{split}
\end{equation}
The left hand side of \eqref{trans 2-a-1} tends to $0$ if we let $s\to-\infty$ and tends to $\infty$ if we let $s\to\infty$. Thus we must have $s=O(1)$.

\end{proof}

\subsection{Definition of a quasiregular map}\label{Definition of a quasiregular map} 

Now we begin to define a quasiregular map by gluing the functions $g_{m,n}$ defined in \eqref{recuree1pm9} for two suitably chosen sequences $(m_k)$ and $(n_k)$. By modifying the functions $g_{m,n}$ suitably, we obtain closely related functions $u_{m,n}$ and $v_{m,n}$ and then glue restrictions of these maps to strips along horizontal lines to obtain quasiregular maps $U$ and $V$ which are defined in the right and left half-planes respectively. In particular, we have $U(iy)=V(iy)$ on the imaginary axis, resulting in a quasiregular map $G$ in the plane. The maps $U$, $V$ and $G$ will commute with the complex conjugation, so it will be enough to define them in the upper half-plane.

Instead of the mapping $z^{\mu}$ used in the proof of Theorem~\ref{maintheorem2} and Theorem~\ref{maintheorem3}, we shall consider the affine mappings in regions of the plane so that the order of the resulting Bank--Laine function remains unchanged. This is based on the fact that the affine mapping from a rectangle region to another rectangle region is "most nearly conformal"~\cite{Ahlfors2006}. Instead of gluing the functions $g_{m,n}$ in \eqref{recuree1pm9} directly, for each pair $(m,2n)$ of nonnegative integers $m$ and $n$, we shall first modify the functions $g_{m,n}$ slightly and define
\begin{equation}\label{diffeo  fajr  a a;lrk-1}
\begin{split}
g^{+}_{m,n}(z)=g_{m,n}\left(\frac{x}{l}+i\frac{y}{N}\right),
\end{split}
\end{equation}
where $l$ is a positive integer dependent on the sequence $(m_k)$ and $(n_k)$ we shall choose later, and
\begin{equation}\label{diffeo  fajr  a a;lrk-2}
\begin{split}
g^{-}_{m,n}(z)=g_{m,n}\left(\frac{x+i y}{N}\right).
\end{split}
\end{equation}
Actually we shall restrict the functions in \eqref{diffeo  fajr  a a;lrk-1} to the right half-plane $\mathbb{H}^+=\{z\in\mathbb{C}: \text{Re}\ z>0\}$ and the functions in \eqref{diffeo  fajr  a a;lrk-2} to the left half-plane $\mathbb{H}^-=\{z\in\mathbb{C}: \text{Re}\ z<0\}$. Thus we may write
\begin{equation}\label{diffeo  fajr  ayt}
\begin{split}
\tilde{g}_{m,n}(z)=(g_{m,n}\circ\chi_{m,n})(z)
\end{split}
\end{equation}
for a map $\chi_{m,n}:\mathbb{C}\to \mathbb{C}$ such that $\chi_{m,n}(x+iy)=x/l+iy/N$ in the right half-plane $\mathbb{H}^+$ and $\chi_{m,n}(x+iy)=(x+iy)/N$ in the left half-plane $\mathbb{H}^-$ respectively. Note that $\chi_{m,n}$ is an affine mapping when restricted to the right or left half-plane and also that $\chi_{m,n}$ is continuous on the imaginary axis.

Now, for the two numbers $\lambda_1,\lambda_2\in[0,1]$ such that $\lambda_1\leq\lambda_2$, we choose the associated sequences $(m_k)$ and $(n_k)$ in the following way:
\begin{enumerate}
  \item [(I)] if $\lambda_1\leq \lambda_2<1$, then we choose the sequence $(m_k)$ and $(n_k)$ constructed in Subsection~\ref{Preliminary lemmas} respectively;
  \item [(II)] if $\lambda_1<\lambda_2=1$, then we choose the sequence $(m_k)$ constructed in Subsection~\ref{Preliminary lemmas} and $(n_k)$ in the way that $n_k=0$ for $k=1,2,3$ and $n_k=1$ for all $k\geq 4$ respectively;
  \item [(III)] if $\lambda_1=\lambda_2=1$, then we choose the sequence $(m_k)$ and $(n_k)$ in the way that $m_k=0$ for $k=1,2,3$ and $m_k=1$ for all $k\geq 4$ and $n_k=0$ for $k=1,2,3$ and $n_k=1$ for all $k\geq 4$ respectively.
\end{enumerate}
Corresponding to the two sequences $(m_k)$ and $(n_k)$ in each case, we shall choose the integer $l$ in $\tilde{g}_{m,n}$ in \eqref{diffeo  fajr  ayt} such that $l=1$ in the case (I) and $l=3$ in the case (II) and $l=4$ in the case (III). Then the explicit expression for $\tilde{g}_{m,n}$ in \eqref{diffeo  fajr  ayt} can be given using $g^{+}_{m,n}$ defined in \eqref{diffeo  fajr  a a;lrk-1} and $g^{-}_{m,n}$ defined in \eqref{diffeo  fajr  a a;lrk-2} as follows. On each strip
\begin{equation*}
\begin{split}
\hat{\Xi}_k=\{x+iy: \quad y\geq0, \quad 2\pi\mathcal{N}_{k-1}<y< 2\pi\mathcal{N}_k\}
\end{split}
\end{equation*}
of the upper half-plane, we shall define
\begin{equation}\label{diffeo  fajr  a-j}
\begin{split}
\tilde{g}_{m_k,n_k}(z)=\left\{
\begin{array}{ll}
g_{m_k,n_k}\left(\frac{x}{l}+i\frac{y-2\pi\mathcal{N}_{k-1}}{N_k}\right),  &  x\geq 0, \\
g_{m_k,n_k}\left(\frac{x+iy-2i\pi\mathcal{N}_{k-1}}{N_k}\right), &  x\leq 0.
\end{array}
\right.
\end{split}
\end{equation}
Note that the two expressions in the definition of $\tilde{g}_{m_k,n_k}$ in \eqref{diffeo  fajr  a-j} coincide on the imaginary axis in each of three cases (I), (II) and (III).

Now we begin to glue the functions $\tilde{g}_{m,n}$ defined in \eqref{diffeo  fajr  ayt} on the horizontal lines. We shall always assume that $(m_k)$ and $(n_k)$ are the two sequences chosen in either of the three cases. In either case, the choice of the integer $l$ guarantees that the resulting functions $U$, $V$ and $G$ are analytic in almost all parts of the plane.

We begin by constructing the map $U$. We shall use the explicit expressions for $\tilde{g}_{m_k,n_k}$ in \eqref{diffeo  fajr  a-j}. Recall from equation \eqref{diffeo  fajr  ayt} that $\tilde{g}_{m,n}=g_{m,n}\circ\chi_{m,n}$ for an affine mapping $\chi_{m,n}$ in the right half-plane. Instead of $\tilde{g}_{m,n}$, we consider the map
\begin{equation*}
\begin{split}
u_{m,n}: \{z\in\mathbb{C}: \text{Re}\ z\geq0 \}\to\mathbb{C}, \quad u_{m,n}(z)=g_{m,n}(\chi_{m,n}(z)+s_{m,n}).
\end{split}
\end{equation*}
Note that $u_{m,n}$ is increasing on the real line and maps $[0,\infty)$ onto $[2,\infty)$. For the two chosen sequences $(m_k)$ and $(n_k)$, we write $N_k=m_k+2n_k+1$. Put $U(z)=u_{m_k,n_k}(z)$ in the half-strip
\begin{equation}\label{halfstrip-1}
\begin{split}
\hat{\Xi}_k^{+}=\{x+iy: x>0, \quad  2\pi\mathcal{N}_{k-1}<y<2\pi\mathcal{N}_k\}.
\end{split}
\end{equation}
The function $U$ may be discontinuous on the horizontal lines $\mathcal{L}^{+}_{k-1}=\{z=x+iy: x>0, y=2\pi\mathcal{N}_{k-1}\}$ and $\mathcal{L}^{+}_k=\{z=x+iy: x>0, y=2\pi\mathcal{N}_{k}\}$ if $m_k=1$ or $n_k=1$. In order to obtain a continuous function we consider the function $\psi_k:[0,\infty)\to[0,\infty)$ defined by $u_{m_{k+1},n_{k+1}}(x)=u_{m_k,n_k}(\psi_k(x))$. The function $\psi_k$ is closely related to the function $\phi$ considered in \eqref{diffeo  fajr  a}. In fact, denoting by $\phi_k$ the function $\phi$ corresponding to $(m_k,2n_k)$ and $(m_{k+1},2n_{k+1})$, we have
\begin{equation}\label{recuree1pm-diff-00}
\begin{split}
\psi_k(x)=l\phi_k\left(\frac{x}{l}+s_{m_{k+1},n_{k+1}}\right)-ls_{m_{k},n_{k}}.
\end{split}
\end{equation}
We then define $U:\{z\in\mathbb{C}: \text{Re}\ z\geq 0\}\to \mathbb{C}$ by interpolating between $u_{m_{k+1},n_{k+1}}$ and $u_{m_{k},n_{k}}$ as follows: if $2\pi \mathcal{N}_{k-1}\leq y<2\pi\mathcal{N}_k$, say $y=2\pi \mathcal{N}_{k-1}+2\pi N_kt$ where $0\leq t<1$, then we put
\begin{equation*}
\begin{split}
U(x+iy)=u_{m_{k},n_{k}}((1-t)x+t\psi_k(x)+iy)=u_{m_{k},n_{k}}(x+iy+t(\psi_k(x)-x)).
\end{split}
\end{equation*}
The function $U$ defined in this way is continuous on the horizontal lines $\mathcal{L}^{+}_k$ for any $k$ in the right half-plane.

We now define a function $V$ in the left half-plane. In order to do so, we define
\begin{equation*}
\begin{split}
v_{m,n}:\{z\in\mathbb{C}: \text{Re}\ z< 0\}\to \mathbb{C},  \quad v_{m,n}(z)=g_{m,n}\left(\chi_{m,n}(z)+s_{m,n}\right).
\end{split}
\end{equation*}
Note that $v_{m,n}$ maps $(-\infty,0]$ monotonically onto $(1,2]$. For the two sequences $(m_k)$ and $(n_k)$ as before, this time we would like to define $V(z)=v_{m_k,n_k}(z)$ in the half-strip
\begin{equation}\label{halfstrip-2}
\begin{split}
\hat{\Xi}_k^{-}=\{x+iy: x<0, \ 2\pi \mathcal{N}_{k-1}\leq y<2\pi \mathcal{N}_k\},
\end{split}
\end{equation}
but again this function would be discontinuous on the horizontal lines $\mathcal{L}^{-}_{k-1}=\{z=x+iy: x<0, y=2\pi\mathcal{N}_{k-1}\}$ and $\mathcal{L}^{-}_k=\{z=x+iy: x<0, y=2\pi\mathcal{N}_{k}\}$ if $m_k=1$ or $n_k=1$. In order to obtain a continuous function we again interpolate between $v_{m_{k+1},n_{k+1}}$ and $v_{m_k,n_k}$. Similarly as before we consider the map $\psi_k:(-\infty,0]\to (-\infty,0]$ defined by $v_{m_{k+1},n_{k+1}}(x)=v_{m_k,n_k}(\psi_k(x))$. More precisely, denoting by $\phi_k$ the function $\phi$ corresponding to $(m_k,2n_k)$ and $(m_{k+1},2n_{k+1})$, we have
\begin{equation}\label{recuree1pm-diff-01}
\begin{split}
\psi_k(x)=N_k\phi_k\left(\frac{x}{N_{k+1}}+s_{m_{k+1},n_{k+1}}\right)-N_ks_{m_{k},n_{k}}.
\end{split}
\end{equation}
Then we define $V:\{z\in \mathbb{C}:
\text{Re}\ z\leq 0\}\to \mathbb{C}$ by interpolating between $v_{m_{k+1},n_{k+1}}$ and $v_{m_k,n_k}$ as follows: if $2\pi \mathcal{N}_{k-1}\leq y<2\pi \mathcal{N}_k$, say $y=2\pi \mathcal{N}_{k-1}+2\pi N_kt$ where $0\leq t<1$, then we put
\begin{equation*}
\begin{split}
V(x+iy)=v_{m_{k},n_{k}}((1-t)x+t\psi_k(x)+iy)=v_{m_{k},n_{k}}(x+iy+t(\psi_k(x)-x)).
\end{split}
\end{equation*}
This map $V$ is continuous on the horizontal lines $\mathcal{L}^{-}_{k}$ for any $k$ in the left half-plane.

Note that $U$ and $V$ coincide on the imaginary axis. With the definitions of the maps $U$ and $V$, we define the map
\begin{equation}\label{recuree1pm-diff-1}
\begin{split}
G(z)=\left\{
       \begin{array}{ll}
         U(z),   &  \text{if} \quad |\arg z|\leq \frac{\pi}{2}; \\
         V(z),   &  \text{if} \quad |\arg(-z)|\leq \frac{\pi}{2},
       \end{array}
      \right.
\end{split}
\end{equation}
which is continuous on the horizontal lines $\mathcal{L}^{+}_k$ in the right half-plane and $\mathcal{L}^{-}_{k}$ in the left half-plane for any integer $k$ respectively. In the section below, we will show that the map $G$ defined in \eqref{recuree1pm-diff-1} is quasiregular in the plane and also satisfies the hypothesis of the Teichm\"uller--Wittich--Belinskii theorem.

\subsection{Estimation of the dilatation}\label{Estimation of the dilatation} 

Now we begin to estimate the dilatation of the functions $U$, $V$ and $G$ defined in previous section. For a quasiregular map $f$, let
\begin{equation*}
\begin{split}
\mu_f(z)=\frac{f_{\overline{z}}(z)}{f_z(z)} \ \ \text{and} \ \ K_f(z)=\frac{1+|\mu_f(z)|}{1-|\mu_f(z)|}.
\end{split}
\end{equation*}
In order to apply the Teichm\"uller--Wittich--Belinskii theorem, we have to estimate~$K_G(z)-1$. We note that
\begin{equation*}
\begin{split}
K_G(z)-1=\frac{2|\mu_G(z)|}{1-|\mu_G(z)|}=\frac{2|\mu_G(z)|(1+|\mu_G(z)|)}{1-|\mu_G(z)|^2}\leq \frac{4|\mu_G(z)|}{1-|\mu_G(z)|^2}.
\end{split}
\end{equation*}
By the explicit expressions of $\tilde{g}_{m_k,n_k}$ in \eqref{diffeo  fajr  a-j}, the arguments for estimating the dilatation of $G$ will be similar in either of the three cases (I), (II) and (III). Below we shall first consider the first case (I).

We begin by estimating $K_{U}-1$ in the right half-plane $\mathbb{H}^+=\{z\in\mathbb{C}: \text{Re}\ z>0\}$. Let $2\pi \mathcal{N}_{k-1}\leq y<2\pi\mathcal{N}_k$ so that $0\leq t=(y-2\pi\mathcal{N}_{k-1})/(2\pi N_k)<1$. Then $U(z)=u_{m_k,n_k}(q(z))$ with
\begin{equation}\label{recuree1pm59 final-3}
\begin{split}
q(x+iy)=x+iy+t(\psi_k(x)-x)=x+iy+\frac{y-2\pi\mathcal{N}_{k-1}}{2\pi N_k}(\psi_k(x)-x),
\end{split}
\end{equation}
where $\psi_k(x)$ takes the form in \eqref{recuree1pm-diff-00}. Thus
\begin{equation}\label{recuree1pm60 final-4}
\begin{split}
q_z(z)=1+a(z)-ib(z) \ \ \text{and} \ \ q_{\overline{z}}=a(z)+ib(z)
\end{split}
\end{equation}
with
\begin{equation}\label{recuree1pm61 final-5}
\begin{split}
a(x+iy)=\frac{t}{2}(\psi_k'(x)-1) \ \ \text{and} \ \ b(x+iy)=\frac{1}{4\pi N_k}(\psi_k(x)-x).
\end{split}
\end{equation}
Note that if $a(z)<0$, then
\begin{equation*}
\begin{split}
a(z)=\frac{t}{2}(\psi'_k(x)-1)\geq \frac{1}{2}(\psi'_k(x)-1)
\end{split}
\end{equation*}
and thus $1+2a(z)\geq \psi'_k(z)>0$. Then we have $1+2a(z)\geq \min\{1,\psi'_k(z)\}>0$ as well as $1+a(z)>0$ in any case.

We deduce from \eqref{recuree1pm59 final-3} that
\begin{equation*}
\begin{split}
|\mu_U(z)|^2=|\mu_q(z)|^2=\frac{a(z)^2+b(z)^2}{(1+a(z))^2+b(z)^2}
\end{split}
\end{equation*}
and thus
\begin{equation*}
\begin{split}
K_U(z)-1 &\leq \frac{4|\mu_q(z)|}{1-|\mu_q(z)|^2}\leq \frac{4\sqrt{(1+a(z))^2+b(z)^2}\sqrt{a(z)^2+b(z)^2}}{1+2a(z)}\\
&\leq \frac{4(1+a(z)+|b(z)|)(a(z)+|b(z)|)}{1+2a(z)}\leq \frac{4(1+r(x))r(x)}{\min\{1,\psi_k'(x)\}},
\end{split}
\end{equation*}
where
\begin{equation}\label{recuree1pm68 final-7}
\begin{split}
r(x)=|\psi_k'(x)-1|+|\psi_k(x)-x|.
\end{split}
\end{equation}

In order to estimate $K_U(z)-1$ for $z\in \hat{\Xi}_k^{+}$ defined in \eqref{halfstrip-1}, we recall from Lemma~\ref{Lemma 0} that
\begin{equation}\label{recuree1pm71 fina-1}
\begin{split}
|\phi_{k}(x)-x|=o(1), \quad x\to\infty.
\end{split}
\end{equation}
By the definition of $\psi_k(x)$ in \eqref{recuree1pm-diff-00} we have
$\psi_{k}(x)-x=l\phi_{k}(x/l+s_{m_{k+1},n_{k+1}})-l(x/l+s_{m_{k+1},n_{k+1}})+l(s_{m_{k+1},n_{k+1}}-s_{m_k,n_k})$ and, together with Lemma~\ref{Lemma b} and the estimate in \eqref{recuree1pm71 fina-1}, thus
\begin{equation}\label{recuree1-main eqiqu-uy-1-pr-1}
\begin{split}
|\psi_{k}(x)-x|=O(1), \quad x\to\infty.
\end{split}
\end{equation}
Further, by Lemma~\ref{Lemma 0} and Lemma~\ref{Lemma b} we also have
\begin{equation*}
\begin{split}
|\psi'_{k}(x)-1|=o(1), \quad x\to\infty.
\end{split}
\end{equation*}
The last two equations imply that the functions $r(x)$ defined by \eqref{recuree1pm68 final-7} satisfies $r(x)=O(1)$ for $z\in \hat{\Xi}_k^{+}$. Note that the affine mapping $\chi_{m_k,n_k}:\mathbb{H}^{+}\to \mathbb{H}^{+}$ defined in \eqref{diffeo  fajr  ayt} satisfies $\chi_{m_k,n_k}(z)=x/l+iy/N_k$ and so $|\mu_{\chi_{m_k,n_k}}|=(N_k-l)/(N_k+l)<1$. Thus $U$ is quasiregular in the right half-plane $\mathbb{H}^+$ and
\begin{equation*}
\begin{split}
K_U(z)-1=O(1), \quad x\to\infty
\end{split}
\end{equation*}
for $z\in \hat{\Xi}_k^{+}$. We denote
\begin{equation*}
\begin{split}
S_{1,k_i}=\{x+iy: x>0, \ 2\pi\mathcal{N}_{k_i-1}<y<2\pi \mathcal{N}_{k_i}, \ m_{k_i}=1\}
\end{split}
\end{equation*}
and
\begin{equation*}
\begin{split}
S_{2,k_j}=\{x+iy: x>0, \ 2\pi\mathcal{N}_{k_j-1}<y<2\pi \mathcal{N}_{k_j}, \ n_{k_j}=1\}.
\end{split}
\end{equation*}
Recall that $m_k=n_k=0$ for $k=1,2,3$. Let $\varepsilon>0$ be small. Define $\gamma$ as follows: $\gamma=\varepsilon$ when $\lambda_2=0$ and $\gamma=\lambda_2$ when $\lambda_2>0$. Then, together with the estimates in \eqref{recuree3-a3-a3} and \eqref{recuree4-a4-a4}, we have
\begin{equation*}
\begin{split}
\int_{S_{1,k_i}}\frac{dxdy}{x^2+y^2}\leq \int_{S_{1,k_i}}\frac{dxdy}{x^2+4\pi^2\mathcal{N}_{k_i-1}^2}\leq 2\pi N_{k_i}\int_0^\infty\frac{dx}{x^2+4\pi^2\mathcal{N}_{k_i-1}^2}=\frac{\pi N_{k_i}}{2\mathcal{N}_{k_i-1}}.
\end{split}
\end{equation*}
By the choice of the two sequence of $(m_k)$ and $(n_k)$ we easily see that
\begin{equation}\label{recuree1pm79-0qjw-1}
\begin{split}
\mathcal{N}_k \sim lk \quad \text{as} \quad k\to\infty.
\end{split}
\end{equation}
Since $N_{k_i}\leq 4$ for all $k_i$, this yields
\begin{equation*}
\begin{split}
\int_{S_{1,k_i}}\frac{dxdy}{x^2+y^2}\leq \frac{4\pi}{c(2\pi i)^{1/\gamma}}
\end{split}
\end{equation*}
for $i\geq 1$ and some positive constant $c$. Thus
\begin{equation}\label{recuree1pm81}
\begin{split}
\int_{S_{1,k_i}}\frac{K_U(z)-1}{x^2+y^2}dxdy \leq \frac{c_1}{(2\pi i)^{1/\gamma}}
\end{split}
\end{equation}
for all $i\geq 1$ and some constant $c_1$. Similarly, we also have
\begin{equation}\label{recuree1pm81 fuzu-1}
\begin{split}
\int_{S_{1,k_i-1}}\frac{K_U(z)-1}{x^2+y^2}dxdy\leq \frac{c_2}{(2\pi i)^{1/\gamma}}
\end{split}
\end{equation}
for all $i\geq 1$ and some constant $c_2$. Moreover, together with the estimates in \eqref{recuree3-b3-b3} and \eqref{recuree4-b4-b4}, we can also deduce that
\begin{equation}\label{recuree1pm81-w}
\begin{split}
\int_{S_{1,k_j}}\frac{K_U(z)-1}{x^2+y^2}dxdy \leq \frac{c_3}{(2\pi j)^{1/\gamma}}
\end{split}
\end{equation}
for all $j\geq 1$ and some constant $c_3$, and that
\begin{equation}\label{recuree1pm81 fuzu-1-w}
\begin{split}
\int_{S_{1,k_j-1}}\frac{K_U(z)-1}{x^2+y^2}dxdy\leq \frac{c_4}{(2\pi j)^{1/\gamma}}
\end{split}
\end{equation}
for all $j\geq 1$ and some constant $c_4$. Note that $U$ is meromorphic in $\mathbb{H}^{+}$ apart from the half-strips $S_{1,k_i-1}$, $S_{1,k_i}$, $S_{1,k_j-1}$ and $S_{1,k_j}$. Then the last four inequalities \eqref{recuree1pm81}-\eqref{recuree1pm81 fuzu-1-w} imply that
\begin{equation}\label{recuree1pm81 fuzu-1-w-final}
\begin{split}
\int_{z\in \mathbb{H}^{+}}\frac{K_U(z)-1}{x^2+y^2}dxdy=2\sum_{k=1}^{\infty}\int_{\Xi^{+}_{k}}\frac{K_U(z)-1}{x^2+y^2}dxdy<\infty.
\end{split}
\end{equation}

The estimate of $K_V(z)-1$ in the left half-plane $\mathbb{H}^-=\{z\in\mathbb{C}: \text{Re}\ z<0\}$ is similar. Now we have $V(z)=V_{m_k,n_k}(q(z))$ where $q(z)$ has the same form as in \eqref{recuree1pm59 final-3}. Moreover, \eqref{recuree1pm60 final-4} and \eqref{recuree1pm61 final-5} both hold for the same $a(x+iy)$ and $b(x+iy)$, but $\psi_k(x)$ taking the form in \eqref{recuree1pm-diff-01}. Instead of \eqref{recuree1pm61 final-5} we obtain
\begin{equation*}
\begin{split}
K_V(z)-1 &\leq \frac{4(1+r(x))r(x)}{\min\{1,\psi_k'(x)\}}
\end{split}
\end{equation*}
with
\begin{equation}\label{recuree1pm95}
\begin{split}
r(x)=\left|\psi'_k(x)-1\right|+\frac{1}{N_k}\left|\psi_k(x)-x\right|.
\end{split}
\end{equation}
Now, by the definition of $\psi_k(x)$ in \eqref{recuree1pm-diff-01}, we have
\begin{equation*}
\begin{split}
& \frac{1}{N_k}\left|\psi_k(x)-x\right|=\left|\phi_k\left(\frac{x}{N_{k+1}}+s_{m_{k+1},n_{k+1}}\right)-\frac{x}{N_{k}}-s_{m_{k},n_{k}}\right|\\
=&\
\left|\phi_k\left(\frac{x}{N_{k+1}}+s_{m_{k+1},n_{k+1}}\right)-\frac{N_{k+1}}{N_{k}}\left(\frac{x}{N_{k+1}}+s_{m_{k+1},n_{k+1}}\right)+\frac{N_{k+1}}{N_{k}}s_{m_{k+1},n_{k+1}}-s_{m_{k},n_{k}}\right|
\end{split}
\end{equation*}
and
\begin{equation*}
\begin{split}
\left|\psi'_k(x)-1\right|=\frac{N_k}{N_{k+1}}\left|\phi'_k\left(\frac{x}{N_{k+1}}+s_{m_{k+1},n_{k+1}}\right)-\frac{N_{k+1}}{N_k}\right|.
\end{split}
\end{equation*}
We thus have by Lemma~\ref{Lemma 0} and Lemma~\ref{Lemma b} that
\begin{equation}\label{recuree1-main eqiqu-uy-1-pr-2}
\begin{split}
\frac{1}{N_k}\left|\psi_k(x)-x\right|=O(1), \quad x\to-\infty
\end{split}
\end{equation}
and also that
\begin{equation*}
\begin{split}
&\left|\psi'_k(x)-1\right|=o(1), \quad x\to-\infty.
\end{split}
\end{equation*}
The last two inequalities imply that the function $r(x)$ defined in \eqref{recuree1pm95} satisfies $r(x)=O(1)$ for $z\in \hat{\Xi}_k^{-}$ defined in \eqref{halfstrip-2}. Now we denote
\begin{equation*}
\begin{split}
T_{1,k_i}=\{x+iy: x<0, \ 2\pi\mathcal{N}_{k_i-1}<y<2\pi \mathcal{N}_{k_i}, \ m_{k_i}=1\}
\end{split}
\end{equation*}
and
\begin{equation*}
\begin{split}
T_{2,k_j}=\{x+iy: x<0, \ 2\pi\mathcal{N}_{k_j-1}<y<2\pi \mathcal{N}_{k_j}, \ n_{k_j}=1\}.
\end{split}
\end{equation*}
Note that the affine mapping $\chi_{m_k,n_k}:\mathbb{H}^{-}\to \mathbb{H}^{-}$ defined in \eqref{diffeo  fajr  ayt} satisfies $\chi_{m_k,n_k}(z)=(x+iy)/N_k$ and thus is analytic there. Then by the same arguments as the estimate for $K_U-1$ in the right half-plane, we also conclude that $V$ is quasiregular in the left half-plane and that
\begin{equation*}
\begin{split}
K_V(z)-1=O(1), \quad x\to-\infty
\end{split}
\end{equation*}
for $z\in T_{1,k_i}$, $z\in T_{1,k_i+1}$, $z\in T_{1,k_j}$ and $z\in T_{1,k_j+1}$. This finally implies that
\begin{equation}\label{recuree1pm81-uu-final}
\begin{split}
\int_{z\in \mathbb{H}^{-}}\frac{K_V(z)-1}{x^2+y^2}dxdy =2\sum_{k=1}^{\infty}\int_{\Xi_k^{-}}\frac{K_V(z)-1}{x^2+y^2}dxdy <\infty.
\end{split}
\end{equation}
The two inequalities in \eqref{recuree1pm81 fuzu-1-w-final} and \eqref{recuree1pm81-uu-final} shows that, if $r>0$, then
\begin{equation}\label{recuree1pm116}
\begin{split}
\int_{|z|>r}\frac{K_G(z)-1}{x^2+y^2}dxdy <\infty.
\end{split}
\end{equation}

For the case (II), since $l=3$ and since the sequence $(n_k)$ are chosen in the way that $n_k=1$ for all $k\geq 4$, we also have $N_k\leq4$ for all $k\geq4$. In this case, we may estimate the dilatation of $U$ in the half-stripes $S_{1,k_i-1}$ and $S_{1,k_i}$ of the right half-plane and the dilatation of $V$ in the half-stirps $T_{1,k_i-1}$ and $T_{1,k_i}$ of the left half-plane similarly as in the case (I). We may also estimate the dilatation of $U$ in the stripes $\hat{\Xi}_k^{+}$ for the case $k=2,3$ in a similar way. We omit the details. Moreover, by the choice of the two sequences $(m_k)$ and $(n_k)$, in the stripe $\hat{\Xi}_1^{+}$ we have
\begin{equation*}
\begin{split}
\int_{\hat{\Xi}_1^{+},|z|>1}\frac{K_U(z)-1}{x^2+y^2}dxdy\leq O(1)\cdot 2\pi\int_{x>1}\frac{dx}{x^2}=O(1)\cdot 2\pi.
\end{split}
\end{equation*}
Together with this estimate we finally conclude that the inequality in \eqref{recuree1pm116} also holds.

For the case (III), since $l=4$ and since the two sequences $(m_k)$ and $(n_k)$ are chosen in the way that $m_k=n_k=1$ for all $k\geq 4$, we have $N_k=4$ for all $k\geq4$. Thus we only need to estimate the dilatation of $U$ in the stripes $\hat{\Xi}_k^{+}$ for the case $k=1,2,3$ to conclude that the inequality in \eqref{recuree1pm116} also holds.

Therefore, in either of the three cases (I), (II) and (III), the function $G$ defined in \eqref{recuree1pm-diff-1} is quasiregular in the plane and satisfies the hypothesis of the Teichm\"uller--Wittich--Belinskii theorem. This theorem, together with the existence theorem for quasiconformal mappings, yields that there exists a quasiconformal homeomorphism $\tau:\mathbb{C}\to\mathbb{C}$ and a meromorphic function $F$ such that
\begin{equation}\label{recuree1pm117 final-8}
\begin{split}
G(z)=F(\tau(z)) \quad  \text{and} \quad \tau(z)\sim z \quad \text{as} \quad z\to\infty.
\end{split}
\end{equation}
Then $F=G\circ \tau^{-1}$ is a locally univalent meromorphic function. Define $E=F/F'$. Then $E$ is a Bank--Laine function. Clearly, $E=f_1f_2$ for two normalized solutions of the second order differential equation \eqref{bank-laine0} with an entire coefficient $A$.

\subsection{Proof for the case~$n=1$}\label{Completion of the proof} 

We first prove Theorem~\ref{maintheorem4} for the case $n=1$. We shall consider the three cases (I), (II) and (III) together. We will use similar arguments as in the proof of Theorem~\ref{maintheorem3} to estimate the numbers of poles and zeros of $F$ respectively and also to determine the orders of $E$ and $A$.

We first estimate the number of poles of $F$. Note that $\mathcal{N}_k\sim lk$ as $k\to\infty$ by \eqref{recuree1pm79-0qjw-1}. Also note that the function $\psi_k$ in \eqref{recuree1pm-diff-00} and \eqref{recuree1pm-diff-01} both satisfy $|\psi_k(x)-x|=O(1)$ by  the estimates in \eqref{recuree1-main eqiqu-uy-1-pr-1} and \eqref{recuree1-main eqiqu-uy-1-pr-2}. Let $r>0$ and choose $k\in \mathbb{N}$ such that $2\pi\mathcal{N}_{k-1} <r \leq 2\pi \mathcal{N}_k$. It follows from the construction of $U$ and $V$ that
\begin{equation}\label{recuree1pm118  fajr19}
\left\{
\begin{array}{ll}
\sum_{i=1}^{k}m_i \leq n(r,\infty,U)\leq 2\sum_{i=1}^{k}m_i, &\\
\sum_{i=1}^{k}m_i \leq n(r,\infty,V)\leq 2\sum_{i=1}^{k}m_i &
\end{array}
\right.
\end{equation}
for large $r$. By the choice of $k_i$ in \eqref{recuree5} or \eqref{recuree6} in the proof of Lemma~\ref{Lemma a}, we easily see that $\sum_{i=1}^{k}m_i \sim (2\pi)^{-1} (\log (2\pi k))^2$ as $k\to\infty$ when $\lambda_1=0$ and that $\sum_{i=1}^{k}m_i \sim (2\pi)^{\lambda_1-1} k^{\lambda_1}$ as $k\to\infty$ when $0<\lambda_1<1$. Obviously, the latter asymptotic relation also holds for the associated sequence $(m_k)$ when $\lambda_1=1$. Then the two double inequalities in \eqref{recuree1pm118  fajr19} imply that $n(r,\infty,G)\asymp (\log r)^2$ when $\lambda_1=0$ and that $n(r,\infty,G)\asymp r^{\lambda_1}$ when $0<\lambda_1\leq 1$. By the relations in \eqref{recuree1pm117 final-8}, we have $n(r,\infty,F)\asymp (\log r)^2$ when $\lambda_1=0$ and $n(r,\infty,F)\asymp r^{\lambda_1}$ when $0<\lambda_1\leq 1$. Thus $\lambda(f_1)=\lambda_1$ in all three cases.

The estimate for the number of zeros of $F$ is similar. Again, let $r>0$ and choose $k\in \mathbb{N}$ such that $2\pi\mathcal{N}_{k-1} <r \leq 2\pi \mathcal{N}_k$. It follows from the construction of $U$ and $V$ that
\begin{equation}\label{recuree1pm118  fajr19jpoi}
\left\{
\begin{array}{ll}
2\sum_{j=1}^{k}n_j\leq n(r,0,U)\leq 4\sum_{i=j}^{k}n_j, &\\
2\sum_{j=1}^{k}n_j\leq n(r,0,V)\leq 4\sum_{j=1}^{k}n_j &
\end{array}
\right.
\end{equation}
for large $r$. By the same arguments as before, we have $2\sum_{j=1}^{k}n_j \sim (2\pi)^{-1} (\log (2\pi k))^2$ as $k\to\infty$ when $\lambda_2=0$ and that $2\sum_{j=1}^{k}n_j \sim (2\pi)^{\lambda_2-1} k^{\lambda_2}$ as $k\to\infty$ when $0<\lambda_2\leq 1$. Then the two double inequalities in \eqref{recuree1pm118  fajr19jpoi} imply that $n(r,0,G)\asymp (\log r)^2$ when $\lambda_2=0$ and that $n(r,0,G)\asymp r^{\lambda_2}$ when $0<\lambda_2\leq 1$. By the relations in \eqref{recuree1pm117 final-8}, we have $n(r,0,F)\asymp (\log r)^2$ when $\lambda_2=0$ and $n(r,0,F)\asymp r^{\lambda_2}$ when $0<\lambda_2\leq 1$. Thus $\lambda(f_2)=\lambda_2$ in all three cases.

Next, let $\varepsilon>0$ be a small constant. Since $m_k,n_k\leq 1$ for all $k$ by the choice of the two sequences $(m_k)$ and $(n_k)$, there is a set $\Omega_1$ of finite linear measure such that the functions $P_{n_k}(e^{z})$ and $Q_{m_k}(e^z)$ satisfy
\begin{equation}\label{growth-1}
\left\{
\begin{array}{ll}
\exp(-(2n_k+\varepsilon)r)\leq\left|P_{n_k}(e^{z})\right|\leq \exp((2n_k+\varepsilon)r), &\\
\exp(-(m_k+\varepsilon)r)\leq\left|Q_{m_k}(e^{z})\right|\leq \exp((m_k+\varepsilon)r) &
\end{array}
\right.
\end{equation}
both hold for all $z\in \mathbb{C}$ such that $|z|\not\in \Omega_1$ and all pairs $(m_k,2n_k)$. This implies that
\begin{equation}\label{growth-2}
\begin{split}
|g_{m_k,n_k}(z)|\leq \exp((m_k+\varepsilon)r+(2n_k+\varepsilon)r)\left|\exp(e^z)\right|
\end{split}
\end{equation}
for all $z\in \mathbb{C}$ such that $|z|\not\in \Omega_1$ and all pairs $(m_k,2n_k)$. For simplicity, we denote $L_k=m_k+2n_k$. Clearly this implies that there is a set $\Omega_2$ of finite linear measure such that
\begin{equation}\label{growth-3}
\left\{
\begin{array}{ll}
|v_{m_k,n_k}(z)|\leq \exp((L_k+2\varepsilon)r/N_k)\left|\exp(e^{z/N_k})\right| & \text{for} \quad z\in \mathbb{H}^{-} \quad \text{and} \quad |z|\not\in \Omega_2,\\
|u_{m_k,n_k}(z)|\leq \exp((L_k+2\varepsilon)r/l)\left|\exp(e^{z/l})\right| & \text{for} \quad z\in \mathbb{H}^{+} \quad \text{and} \quad |z|\not\in \Omega_2
\end{array}
\right.
\end{equation}
for all pairs $(m_k,2n_k)$.

Let $z=x+iy\in \mathbb{H}^{-}$ with $\text{Im}\ z\geq 0$ and $k\in \mathbb{N}$ with $2\pi\mathcal{N}_{k-1}\leq \text{Im}\ z<2\pi\mathcal{N}_k$. With $t=(y-2\pi\mathcal{N}_{k-1})/{2\pi N_k}$ and $\psi_k$ defined in \eqref{recuree1pm-diff-01}, we have $0\leq t<1$ and
\begin{equation}\label{growth-4}
\begin{split}
|V(x+iy)|=|v_{m_k,n_k}((1-t)x+t\psi_k(x)+iy)|\leq \exp((L_k+3\varepsilon)r/N_k).
\end{split}
\end{equation}
Since $L_k\leq 3$ and $N_k=m_k+2n_k+1\geq 1$ for all $k$, thus we have
\begin{equation}\label{recuree1pm125  fajr20}
\begin{split}
|V(z)|\leq \exp((3+3\varepsilon)r)
\end{split}
\end{equation}
for $z\in \mathbb{H}^{-}$ and $|z|\not\in \Omega_2$. For $z\in \mathbb{H}^{+}$ and $|z|\not\in \Omega_1$, by the inequality in \eqref{growth-2} we have
\begin{equation*}
\begin{split}
|g_{m_k,n_k}(z)|\leq \exp((L_k+2\varepsilon)r/l)\exp\left(e^{(\text{Re}\  z)/l}\right).
\end{split}
\end{equation*}
Again, assuming that $\text{Im}\ z\geq 0$, we choose $k\in \mathbb{N}$ such that $2\pi \mathcal{N}_{k-1}\leq \text{Im}\ z <2\pi \mathcal{N}_k$. Then, with $U(z)=u_{m_k,n_k}(q(z))=g_{m_k,n_k}(q(z)+s_{m_k,n_k})$ where $q(z)$ is defined by \eqref{recuree1pm59 final-3}, we have
\begin{equation}\label{recuree1pm126  fajr21hug}
\begin{split}
|U(z)|\leq \exp((L_k+2\varepsilon)(\text{Re}(q(z))/l+s_{m_k,n_k})\exp\left(e^{\text{Re}(q(z))/l+s_{m_k,n_k}}\right)
\end{split}
\end{equation}
for $z\in \mathbb{H}^{+}$ and $|z|\not\in \Omega_2$. Since $m_k$ and $n_k$ are either $0$ or $1$ and since $s_{m_k,n_k}=O(1)$ by Lemma~\ref{Lemma b}, we deduce from \eqref{recuree1pm59 final-3} and \eqref{recuree1pm126  fajr21hug} that
\begin{equation*}
\begin{split}
\log|U(z)|&\leq (L_k+2\varepsilon)\text{Re}(q(z))/l+s_{m_k,n_k}+\exp(\text{Re}(q(z))/l+s_{m_k,n_k})\\
&\leq \exp((1+o(1))|z|/l)+O(1)=\exp((1+o(1))|z|/l)
\end{split}
\end{equation*}
as $z\to\infty$ in $\mathbb{H}^{+}$ and $|z|\not\in \Omega_2$. Together with \eqref{recuree1pm125  fajr20} we conclude that
\begin{equation*}
\begin{split}
\log|G(z)|\leq \exp((1+o(1))|z|/l)
\end{split}
\end{equation*}
as $|z|\to\infty$ outside a set $\Omega_3$ of finite linear measure. Hence \eqref{recuree1pm117 final-8} yields that
\begin{equation*}
\begin{split}
\log|F(z)|\leq \exp((1+o(1))|z|/l)
\end{split}
\end{equation*}
as $|z|\to\infty$ outside a set $\Omega_4$ of finite linear measure. The lemma on the logarithmic derivative now implies that $E=F/F'$ satisfies
\begin{equation}\label{prox esti-1}
\begin{split}
m\left(r,\frac{1}{E}\right)=O(r)
\end{split}
\end{equation}
as $|z|\to\infty$ outside a set of finite linear measure. Moreover, since all zeros of $E$ are simple and since $\lambda_1\leq \lambda_2$, the previous estimates on the poles and zeros of $F$ imply that, when $\lambda_2=0$,
\begin{equation}\label{zero esti-1}
\begin{split}
N(r,0,E)\leq c_1(\log r)^3
\end{split}
\end{equation}
for some positive constant $c_1$ and all large $r$ or, when $0<\lambda_2\leq 1$,
\begin{equation}\label{zero esti-2}
\begin{split}
c_3 r^{\lambda_2} \leq N(r,0,E)\leq c_2 r^{\lambda_2}
\end{split}
\end{equation}
for two positive constants $c_2,c_3$ and all large $r$. By the estimate in \eqref{prox esti-1} together with the estimates in \eqref{zero esti-1} or \eqref{zero esti-2}, the first main theorem implies that $T(r,E)=T(r,1/E)+O(1)=O(r)$. Thus $\lambda(E)\leq \rho(E)\leq 1$. In particular, if $\lambda_2=1$, then the estimate in \eqref{zero esti-2} shows that $\lambda(E)=\rho(E)=1$. If $\lambda_2<1$, then by the classical Hadamard factorization theorem we have $E(z)=w(z)e^{p(z)}$ for an entire function $w(z)$ of order $\lambda_2$ and a linear polynomial $p(z)$. This together with the lemma on the logarithmic derivative easily yields that $T(r,E'/E)=o(1)T(r,E)$ and $T(r,E''/E)=o(1)T(r,E)$ where $r\to\infty$. Then the relation in \eqref{recuree1pm7} implies that $T(r,A)=T(r,1/E)+o(1)T(r,E)=T(r,E)+o(1)T(r,E)$. Since an independent estimate below which only uses the fact that $E$ is of finite order yields $\rho(A)=1$, we also have $\rho(E)=1$.

When $E$ is of finite order, it follows by the relation \eqref{recuree1pm7} together with the lemma on the logarithmic derivative that
\begin{equation*}
\begin{split}
m(r,A)=m\left(r,\frac{1}{E}\right)+O(\log r)=O(r).
\end{split}
\end{equation*}
Thus the first main theorem implies that $\rho(A)\leq 1$. On the other hand, we note that $g_{1,1}$ is analytic in $\hat{\Xi}_1^{+}$ and in $\check{\Xi}_{-1}^{+}$ as well as in the positive real axis $\mathbb{R}^{+}$. Recall from the definition of $\tilde{g}_{m,n}$ in \eqref{diffeo  fajr  ayt} that $\chi_{1,1}(z)=x/l+iy$ in the right half-plane. Together with the function $\tau$ in \eqref{recuree1pm117 final-8}, we have $|(\chi_{1,1}\circ\tau^{-1})(z)|\leq 2|z|/l$ for $z$ such that $|z|$ is large. Denote $\tilde{\chi}_{1,1}=\chi_{1,1}+s_{1,1}$. Then by the same arguments as in the proof of Theorem~\ref{maintheorem3} we may choose a curve in $(\tau\circ\tilde{\chi}_{1,1}^{-1})(\hat{\Xi}_1^{+}\cup\check{\Xi}_{-1}^{+}\cup \mathbb{R}^{+})$, say $\mathcal{L}^{+}$, such that $(\tilde{\chi}_{1,1}\circ\tau^{-1})(\mathcal{L}^{+})$ is the positive real axis and also that $\tau\circ\tilde{\chi}_{1,1}^{-1}$ satisfies $|(\tau\circ\tilde{\chi}_{1,1}^{-1})'(z)|\leq l|2z|$ for all $z\in (\tilde{\chi}_{1,1}\circ\tau^{-1})(\mathcal{L}^{+})$ such that $|z|$ is large. Then, by the relation in \eqref{recuree1pm7} as well as the definitions of $u_{m,n}$ in Subsection~\ref{Definition of a quasiregular map} and the relation $F=G\circ \tau^{-1}$,
we finally obtain that $\log |A(z_n)|\geq d|z_n|$ for some positive constant $d$ and an infinite sequence $(z_n)$ such that $z_n\in\mathcal{L}^{+}$ and $|z_n|$ is large, implying that $\rho(A)\geq 1$. Hence $\rho(A)=1$. This proves Theorem~\ref{maintheorem4} for the case $n=1$.

\subsection{Proof for the case~$n\geq 2$}\label{Completion of the proof-ex} 

In this subsection, we show how to extend the results in the case $n=1$ to the case $n\geq 2$ for any positive integer $n$ with the same process as in \cite[Section~5]{Bergweilereremenko2019}. For two numbers $\Lambda_1,\Lambda_2\in[0,n]$ such that $\Lambda_1\leq \Lambda_2$, we let $\lambda_1=\Lambda_1/n$ and $\lambda_2=\Lambda_2/n$. Then, we choose two sequences $(m_k)$ and $(n_k)$ associated with $\lambda_1$ and $\lambda_2$ in the same way as in Subsection~\ref{Definition of a quasiregular map}. Note that $m_k=n_k=0$, $k=1,2,3$, for the two sequences $(m_k)$ and $(n_k)$ in any case.

Let $\chi_{1,1}:\mathbb{C}\to\mathbb{C}$ be such that $\chi_{1,1}(z)=x/l+iy$ when $x\geq 0$ and $\chi_{1,1}(z)=x+iy$ when $x\leq 0$. Here $l=1$ when $\lambda_1\leq \lambda_2<1$ and $l=3$ when $\lambda_1<\lambda_2=1$ and $l=4$ when $\lambda_1=\lambda_2=1$. Then we choose the numbers $\rho_0$ and $\sigma_0$ in~\cite[Section~5]{Bergweilereremenko2019} to be $\rho_0=\sigma_0=1$ and the quasiconformal mapping $Q$ there to be $Q(z)=\chi_{1,1}(z)$.

We define a similar function $G_0$ as in previous subsection. Denote by $J^{+}$ and $J^{-}$ the preimages of the half-strips $\{z\in \mathbb{C}:\text{Re}\ z\geq0, 0\leq \text{Im}\ z\leq 2\pi\}$ and $\{z\in \mathbb{C}:\text{Re}\ z\leq0, 0\leq \text{Im}\ z\leq 2\pi\}$ under $z\mapsto \chi_{1,1}(z)$ respectively. With $s_0=\log\log 2$, we then have
\begin{equation}\label{big-integer-order-1}
G_0(z)=
\left\{
  \begin{array}{ll}
    \exp\exp(\chi_{1,1}(z)+s_0), & \text{if} \quad  z\in J^{+}, \\
    \exp\exp(z+s_0),       & \text{if} \quad  z\in J^{-}.
  \end{array}
\right.
\end{equation}
We note that $J:=J^{+}\cup J^{-}$ is bounded by the real axis and a line in the upper half-plane which, as the real axis, is mapped to $(1,\infty)$ by $G_0$. In addition to $G_0$, we will also consider a modification $G_1$ of $G_0$ defined as follows. Let $U$ and $V$ be as in Subsection~\ref{Definition of a quasiregular map}. For $\text{Re}\ z\geq 0$ we define
\begin{equation*}
U_1(z)=
\left\{
  \begin{array}{lll}
    U(z-\pi i),         & \text{if} \quad        \text{Im}\ z\geq \pi, \\
    \exp(-\exp(z+s_0),  & \text{if} \quad -\pi < \text{Im}\ z <\pi, \\
    U(z+\pi i),         & \text{if} \quad         \text{Im}\ z\leq -\pi
  \end{array}
\right.
\end{equation*}
and for $\text{Re}\ z<0$ we define
\begin{equation*}
V_1(z)=
\left\{
  \begin{array}{lll}
    V(z-\pi i),         & \text{if} \quad \text{Im}\ z\geq \pi, \\
    \exp(-\exp(z+s_0),  & \text{if} \quad -\pi < \text{Im}\ z <\pi, \\
    V(z+\pi i),         & \text{if} \quad \text{Im}\ z\leq -\pi.
  \end{array}
\right.
\end{equation*}
The map $G_1$ is then obtained by gluing $U_1$ and $V_1$ in the same way in which $U$ and $V$ were glued to obtain $G_0$. Thus we obtain
\begin{equation}\label{big-integer-order-4}
G_1(z)=
\left\{
  \begin{array}{ll}
    \exp(-\exp(\chi_{1,1}(z)+s_0)),  & \text{if} \quad  z\in J^{+}, \\
    \exp(-\exp(z+s_0)),        & \text{if} \quad  z\in J^{-}.
  \end{array}
\right.
\end{equation}
It follows that $G_1(z)=1/G_0(z)$ for $z\in J$. (As mentioned in the introduction, we may define $G_1(z)=1/G_0(z)$ for all $z\in \mathbb{C}$ so that the resulting Bank--Laine function $E=f_1f_2$ satisfies $\lambda(E)\leq n$ and $\lambda(f_1)=\lambda(f_2)$.)

For $j=1,\cdots,2n$ we put
\begin{equation*}
\begin{split}
\Sigma_j=\left\{z\in \mathbb{C}: (j-1)\frac{\pi}{n}<\arg z<j\frac{\pi}{n}\right\}.
\end{split}
\end{equation*}
Here we have used the notation $n$ instead of the $N$ in~\cite[Section~5]{Bergweilereremenko2017}. Let $H=\{z\in \mathbb{C}:\text{Im}\ z>0\}$ be the upper half-plane and $L=H/J$. Denote by $\overline{L}$ the reflection of $L$ on the real axis; that is, $\overline{L}=\{z\in \mathbb{C}: \overline{z}\in L\}$. Denote by $D_j$ the preimage of $L$ or $\overline{L}$ under the map $z\mapsto z^n$ in $\Sigma_j$ for $j=1,\cdots,2n$. With the functions $\chi_{1,1}=Q$ and $G_0$ and $G_1$ above, the process of defining a quasiregular function $G$ is the same as that in~\cite[Subsection~5.2-Subsection~5.3]{Bergweilereremenko2017}. Here we describe the definition briefly. We define the quasiregular map $G:\mathbb{C}\to\mathbb{C}$ which satisfies $G(z)=G_0(z^n)$ for $z\in D_j$ if $2\leq j\leq 2n-1$ and $G(z)=G_1(z^n)$ for $z\in D_1$ and $z\in D_{2n}$. In the remaining part of the plane we define $G$ by a suitable interpolation which will require only \eqref{big-integer-order-1} and \eqref{big-integer-order-4}. For a suitable $r_0>0$, we define $G$ in certain neighborhoods of the rays $R_j=\{z\in \mathbb{C}: \arg z=j\pi/n,|z|>r_0\}$. Near the rays $R_j$ such that $j\not=1$ and $j\not=2n-1$, we define $G=G_0(\varphi(z^n))$ for $2\leq j\leq 2n-2$ and $G=G_1(\varphi(z^n))$ for $j=2n$ for a quasiconformal mapping $\varphi$ in the plane such that $\varphi(z^n)\sim z^n$ as $z\to\infty$. In particular, $\varphi(z^n)=z^n$ when $j$ is odd. (Here we have put $D_{2n+1}=D_1$. In similar expressions the index $j$ will also be taken modulo $2n$.)
Near the rays $R_1$ and $R_{2n-1}$, we define $G=G_0(\varphi(z^n))$ and $G=G_1(\tau_1(\varphi(z^n)))$ or $G=G_1(\tau_2(\varphi(z^n)))$ for two quasiconfromal mappings $\tau_1(z)=x+i(y+\pi)/2$ and $\tau_2(z)=\overline{\tau_1(\overline{z})}$. Of course, $\varphi$, $\tau_1$ and $\tau_2$ are chosen so that $G(z)$ is continuous.
Finally, the definition of $G$ can be extended to the remaining part of the plane. We shall omit the details.

Then, together with the calculations in Subsection~\ref{Estimation of the dilatation}, we may follow the process in~\cite[Subsection~5.4]{Bergweilereremenko2017} to show that $G$ is quasiregular in the plane and satisfies the hypothesis of the Teichm\"uller--Wittich--Belinskii theorem. This theorem, together with the existence theorem for quasiconformal mappings, yields that there exists a quasiconformal homeomorphism $\tau:\mathbb{C}\to\mathbb{C}$ and a meromorphic function $F$ such that
\begin{equation*}
\begin{split}
G(z)=F(\tau(z)) \quad  \text{and} \quad \tau(z)\sim z \quad \text{as} \quad z\to\infty.
\end{split}
\end{equation*}
We shall also omit the details. Then $E=F/F'$ is a Bank--Laine function and $E=f_1f_2$ for two normalized solutions of the second order linear differential equation \eqref{bank-laine0} with an entire function $A$.

As in the case $n=1$, we may estimate the numbers of poles and zeros of $G_0(z^n)$ in the sectors $D_2,D_3,\cdots,D_{2n-1}$ and that of $G_1(z^n)$ in the sectors $D_1,D_{2n}$, respectively. Since $G$ has no poles or zeros in the region $\mathbb{C}-\sum_{j=1}^{2n}D_j$, this finally yields that $n(r,\infty,F)\asymp (\log r)^2$ when $\Lambda_1=0$ and that $n(r,\infty,F)\asymp r^{\Lambda_1}$ when $0<\Lambda_1\leq n$ and also that $n(r,0,F)\asymp (\log r)^2$ when $\Lambda_2=0$ and that $n(r,0,F)\asymp r^{\Lambda_2}$ when $0<\Lambda_2\leq n$. Moreover, by the estimates in Subsection~\ref{Completion of the proof} in the case $n=1$, we easily see that $\log |F|\leq \exp((1+o(1))|z|^n)$ as $z\to\infty$ in $\sum_{j=1}^{2n}D_j$. By the construction of $G$, we also find that the interpolation near the rays $R_j$ yields the same estimate. This implies that $\rho(E)\leq n$ and also $\rho(A)\leq n$. On the other hand, we may also estimate $|A|$ from below along certain curve in similar way as in the case $n=1$. For example, letting $\eta:\mathbb{C}\to \mathbb{C}$ be the map such that $\eta(z)=z^n$, we may choose a curve in $(\tau\circ\eta^{-1}\circ\tilde{\chi}_{1,1}^{-1})(\hat{\Xi}_1^{+}\cup\check{\Xi}_{-1}^{+}\cup \mathbb{R}^{+})$, say $Y_2'$, so that $(\tilde{\chi}_{1,1}\circ\eta\circ\tau^{-1})(Y_2')$ is the positive real axis. For a point $z$ in the real axis, we may use the Cauchy formula as in the proof of Theorem~\ref{maintheorem3} to show that $|(\tau\circ\eta^{-1}\circ\tilde{\chi}_{1,1}^{-1})'(z)|\leq l|2z|^{1/n}$ when $|z|$ is large. Then, by the same arguments as in the case $n=1$, we obtain that $\rho(A)\geq n$ and further that $\rho(E)=\rho(A)=n$. We shall also omit the details. This completes the proof of Theorem~\ref{maintheorem4}.

\section{Proof of Theorem~\ref{maintheorem5}}\label{Proof of maintheorem5} 

\subsection{Preliminary lemmas}\label{Preliminary lemmas-add} 

In this subsection, we provide some preliminary lemmas which are corrected versions of those in~\cite[Section~2]{Zhang2024}. Since the construction of the Bank--Laine functions $E$ in the case $\lambda(E)=1$ is included in Theorem~\ref{maintheorem4}, here we only consider the case where $\lambda(E)>1$.

\begin{lemma}\label{Lemma1}
Let $\gamma\in(1,\infty)$ and $\delta\in[0,1]$ be two numbers and let $\alpha:[0,\infty)\to[0,\infty)$ be a function such that $\alpha(x)=[\log(x+2\pi)]^{-3}$. Then there exists a nonnegative sequence $(m_k)$ of integers, with $m_k=0$ for $k=1,2,3$, such that the function $\mathfrak{h}:[0,\infty)\to[0,\infty)$ which satisfies $\mathfrak{h}(0)=0$ and which is linear on each interval $[2\pi (k-1), 2\pi k]$ and has slopes $m_k$ there satisfies either
\begin{enumerate}
 \item [(1)]
when $0\leq \delta\gamma \leq 1$, there is a subsequence $(m_{k_i})$ such that $m_{k_i}=1$ for $k\geq k_1$ and $m_{k}=0$ for other $k$ such that
\begin{equation*}
\begin{split}
\mathfrak{h}(x)=x^{\delta\gamma}+O(1), \quad x\to\infty
\end{split}
\end{equation*}
and, in particular, $m_k=0$ for all $k\geq k_1$ when $\delta\gamma=0$ and $m_k=1$ for all $k\geq k_1$ when $\delta\gamma=1$ and $k_{i}\geq c i^{1/\delta\gamma}$ for a fixed positive constant $c$ when $0<\delta\gamma<1$; or

\item [(2)]
when $\delta\gamma>1$, $(m_k)$ is increasing such that $m_{k+1}-m_k$ is an odd integer when $m_{k+1}>m_k$ such that
\begin{equation*}
\begin{split}
\mathfrak{h}(x)=\alpha(x)x^{\delta\gamma}+O\left(\alpha(x)x^{\delta\gamma-2}\right)+O(1), \quad x\to\infty
\end{split}
\end{equation*}
and
\begin{equation*}
\begin{split}
m_k=\delta\gamma\alpha(k)(2\pi k)^{\delta\gamma-1}+O\left(\alpha(k)k^{\delta\gamma-2}\right)+O(1), \quad k\to\infty.
\end{split}
\end{equation*}
\end{enumerate}

\end{lemma}

\begin{lemma}\label{Lemma2}
Let $\gamma\in(1,\infty)$ be the constant in Lemma~\ref{Lemma1} and $(m_k)$ be the sequence there. Then there exists a sequence $(n_k)$ of nonnegative integers, with $n_k=0$ for $k=1,2,3$, such that the function $h:[0,\infty)\to[0,\infty)$ which satisfies $h(0)=0$ and which is linear in the intervals $[2\pi (k-1), 2\pi k]$ and has slopes $N_k=m_k+2n_k+1$ there satisfies
\begin{equation*}
\begin{split}
h(x)=x^{\gamma}+O(x^{\gamma-2})+O(1), \quad x\to\infty
\end{split}
\end{equation*}
and
\begin{equation}\label{recuree1pm16  fajr-2}
\begin{split}
N_k=m_k+2n_k+1=\gamma(2\pi k)^{\gamma-1}+O\left(k^{\gamma-2}\right)+O(1), \quad k\to\infty.
\end{split}
\end{equation}
Moreover, the function $g(x)$ defined on $[0,\infty)$ by $h(g(x))=x^{\gamma}$ satisfies
\begin{equation}\label{recuree1pm17  fajr-1}
\begin{split}
g(x)=x+O(x^{-1})+O(x^{1-\gamma}), \quad x\to\infty
\end{split}
\end{equation}
and
\begin{equation}\label{recuree1pm18   fajr0}
\begin{split}
g'(x)=1+O(x^{-1})+O(x^{1-\gamma}), \quad x\to\infty,
\end{split}
\end{equation}
where $g'$ denotes either the left or right derivative of $g$.

\end{lemma}

For two given numbers $\gamma\in(1,\infty)$ and $\delta\in[0,1]$, we may follow the proof
of~\cite[Lemma~2.2]{Zhang2024}, together with the comments in Lemma~\ref{Lemma a}, to obtain the desired sequence $(m_k)$ in Lemma~\ref{Lemma1}. Also, for such a sequence $(m_k)$, we may follow the proof in~\cite[Lemma~2.1]{Zhang2024} to find a positive sequence $(N_k)$ such that $N_k=1$ for $k=1,2,3$ and $N_k$ has opposite parity to $m_k$ for $k\geq 4$. The desired sequence $(n_k)$ in Lemma~\ref{Lemma2} is then defined in the way that $2n_k+1=N_k-m_k$ for $k\geq 1$. The assumptions that $m_k=0$ for $k=1,2,3$ and that $n_k=0$ for $k=1,2,3$ can be satisfied by slightly modifying the original proofs in~\cite{Zhang2024}. Also note that the sequence $(m_k)$ here can be chosen to have the properties in Lemma~\ref{Lemma a} in the case $\delta=0$, but we shall not do this in this section.

In the original version of Lemma~\ref{Lemma2} (i.e.,\cite[Lemma~2.1]{Zhang2024}), we have chosen the slopes of $h$ in the interval $[2\pi(k-1),2\pi k]$ to be $2n_k+1$. However, this yields the function $g$ defined by $h(g(x))+\mathfrak{h}(g(x))=x^{\gamma}$ there satisfies $|g(x)-x|\to\infty$ as $x\to\infty$. For such a function $g$, the corresponding function $Q$ in the definition of $G$ in \cite[Subsection~2.2]{Zhang2024} is not quasiconformal in the plane. Now, by Lemma~\ref{Lemma2}, we have the relations in \eqref{recuree1pm17  fajr-1} and \eqref{recuree1pm18   fajr0} so that $|g(x)-x|$ and $|g'(x)-1|$ both tend to $0$ as $x\to\infty$.

Below we begin to prove Theorem~\ref{maintheorem5} under the assumption $\rho\in(1/2,1)$. For a pair $(m,2n)$, we shall first modify $g_{m,n}$ in \eqref{recuree1pm9} slightly to define
\begin{equation}\label{recuree1pm1309hb-1-kjh}
\begin{split}
\hat{g}_{m,n}(z)=\frac{1}{2}[g_{m,n}(z)+1].
\end{split}
\end{equation}
Actually we shall restrict the functions in \eqref{recuree1pm1309hb-1-kjh} to the region $\{z=x+iy, y\geq4\pi\}$ of the upper half-plane and to the rest of the plane the functions
\begin{equation}\label{recuree1pm1309hb-1-kjh-j87}
\begin{split}
\check{g}_{m,n}(z)=g_{m,n}(z).
\end{split}
\end{equation}
Thus we may write
\begin{equation}\label{recuree1pm1309hb-1-kjh-1}
\begin{split}
\overline{g}_{m,n}(z)=(\overline{\chi}_{m,n}\circ g_{m,n})(z)
\end{split}
\end{equation}
for an affine mapping $\overline{\chi}: \mathbb{C}\to \mathbb{C}$ such that either $\overline{\chi}_{m,n}(z)=(z+1)/2$ or $\overline{\chi}_{m,n}(z)=z$. Let $\gamma\in(1,\infty)$ and $\delta\in[0,1]$. We choose the sequence $(m_k)$ and $(n_k)$ constructed in Lemma~\ref{Lemma1} and Lemma~\ref{Lemma2} respectively. Actually, the function $\overline{\chi}_{m_k,n_k}$ will be chosen such that $\overline{\chi}_{m_k,n_k}(z)=(z+1)/2$ when $k\geq 3$ and $\overline{\chi}_{m_k,n_k}(z)=z$ otherwise. The assumptions that $m_k=n_k=0$ for $k=1,2,3$ make it convenient to extend the construction of the entire function $A$ in the case $\rho(A)=\rho\in(1/2,1)$ to the case $\rho(A)=\rho\in(n/2,n)$ for any positive integer $n$ as in \cite[Section~5]{Bergweilereremenko2019}.

Let $m,n,\mathfrak{m},\mathfrak{n}\in \mathbb{N}$ and put $N=m+2n+1$ and $M=\mathfrak{m}+2\mathfrak{n}+1$. By the definition of $\overline{\chi}_{m,n}$ in \eqref{recuree1pm1309hb-1-kjh-1}, for each pair $(m,2n)$, we shall consider the function $\overline{\phi}:\mathbb{R}\to\mathbb{R}$ defined by
\begin{equation}\label{diffeo}
\begin{split}
(\overline{\chi}_{\mathfrak{m},\mathfrak{n}}\circ g_{\mathfrak{m},\mathfrak{n}})(x)=\overline{g}_{\mathfrak{m},\mathfrak{n}}(x)=(\overline{g}_{m,n}\circ\overline{\phi})(x)=(\overline{\chi}_{m,n}\circ g_{m,n}\circ\overline{\phi})(x).
\end{split}
\end{equation}
We will consider the functions $\overline{\phi}$ for the case that $m=m_k$, $n=n_k$, $\mathfrak{m}=m_{k+1}$ and $\mathfrak{n}=n_{k+1}$, that is, $M=N_{k+1}$ and $N=N_k$. Thus we will consider the asymptotic behaviors of $\overline{\phi}_k$ as $k\to\infty$, but in order to simplify the formulas in most situations below we suppress the dependence of $\overline{\phi}$ from $m,n$ and $\mathfrak{m},\mathfrak{n}$ from the notation.

We first consider one pair $(m,2n)$ and suppose that $\overline{\chi}_{\mathfrak{m},\mathfrak{n}}(z)=(z+1)/2$ and $\overline{\chi}_{m,n}(z)=z$. Denote also $\kappa=M/N$. Noting that $e^{Kx-e^x}=O(e^{-x})$ for any positive constant $K$ as $x\to\infty$ and $e^{-e^x}=1+O(e^{x})$ as $x\to-\infty$, we may replace the term $\sum_{j=0}^{2n_2}B_je^{jx}$ in the dominator in equation \eqref{recuree1pm16  bef0} by $e^{-e^x}\sum_{i=0}^{m_2}A_ie^{ix}+\sum_{j=0}^{2n_2}B_je^{jx}$ and follow exactly the same arguments as in the proof of Lemma~\ref{Lemma 0} to obtain that
\begin{eqnarray}
&\overline{\phi}(x)=x+O(e^{-x/2}),        \quad\quad\quad\quad\quad &\overline{\phi}'(x)=1+O(e^{-x/2}), \quad\quad x\to \infty,  \label{recuree1pm12-new-1}\\
&\overline{\phi}(x)=\kappa x+c+O(e^{-\delta|x|/2}),    \quad\quad &\overline{\phi}'(x)=\kappa+O(e^{-\delta|x|/2}),  \quad x\to -\infty            \label{recuree1pm12-new-2}
\end{eqnarray}
with
\begin{equation*}
\begin{split}
c=\frac{1}{N}\log \left[\frac{\binom{m_1+2n_1}{m_1}}{\binom{m_2+2n_2}{m_2}}\frac{N!}{M!}\right] -\frac{1}{N}\log 2 \quad \text{and} \quad \delta=\frac{1}{2}\min\{1,\kappa\}.
\end{split}
\end{equation*}
In fact, writing $\overline{g}_{m_{k+1},n_{k+1}}(x)=\overline{g}_{m_k,n_k}(\overline{\phi}_k(x))$, since in any case we have chosen $m_k=n_k=0$, $k=1,2,3$, for the two sequences $(m_k)$ and $(n_k)$, we have $\overline{g}_{m_{3},n_{3}}=[g_{m_3,n_3}+1]/2$ and $\overline{g}_{m_2,n_2}=g_{m_2,n_2}$ and thus
\begin{equation}\label{diffeo-s-2}
\overline{g}_{m_3,n_3}(x)=(\overline{g}_{m_2,n_2}\circ\overline{\phi}_{2})(x).
\end{equation}
Thus we will only use the asymptotic behaviors in \eqref{recuree1pm12-new-1} and \eqref{recuree1pm12-new-2} for the case $M=N=1$. Otherwise, we always have $\overline{\phi}_k=\phi_k$, which is an increasing diffeomorphism on $\mathbb{R}$ appearing in \eqref{recuree-1orqr-1} and \eqref{diffeo  fajr  ayt} but defined here for the pair $(m_k,2n_k)$ when $k\not=3$. By the construction of $(m_k)$ and $(n_k)$, we have $n_k\to\infty$ as $k\to\infty$ and thus the asymptotic behaviors of $\phi_k$ becomes more complicated. Below we study the asymptotic behaviors of $\overline{\phi}_k$ as $k\to\infty$.

For simplicity, we shall use the notation $\beta(x)=\delta\alpha(k)(2\pi k)^{(\delta-1)\gamma}$. Note that $\beta(k)=o(1)$ as $k\to\infty$. We see that
\begin{equation}\label{recuree1pm71}
\begin{split}
\frac{m_k}{m_k+2n_k+1}=(1+o(1))\beta(k)=o(1), \quad k\to\infty
\end{split}
\end{equation}
and thus
\begin{equation}\label{recuree1pm72}
\begin{split}
\frac{2n_k+1}{m_k+2n_k+1}=1-(1+o(1))\beta(k)=1+o(1), \quad k\to\infty.
\end{split}
\end{equation}
Moreover, with the notation $N_k=m_k+2n_k+1$, we deduce from \eqref{recuree1pm16  fajr-2} that
\begin{equation}\label{recuree1pm73}
\begin{split}
\frac{N_{k+1}}{N_k}=1+O\left(\frac{1}{k}\right), \quad k\to\infty
\end{split}
\end{equation}
and, denoting $\mathcal{N}_k=\sum_{j=1}^{k}N_j$, we have from \eqref{recuree1pm16  fajr-2} that
\begin{equation}\label{recuree1pm109}
\begin{split}
\mathcal{N}_{k}\sim (2\pi)^{\gamma-1}k^{\gamma}, \quad k\to\infty.
\end{split}
\end{equation}
The above estimates \eqref{recuree1pm71}--\eqref{recuree1pm109} are used in proving the lemmas below. We first have the following

\begin{lemma}\label{Lemma3}
Let $m,n\in \mathbb{N}$ and put $N=m+2n+1$. Let $y>0$. Then
\begin{equation}\label{recuree1pm24}
\begin{split}
\log(h_{m,n}(y)-1)=&\ -\log \binom{m+2n}{m} N!+y+N\log y\\
&\ -2\log Q_m(y)-\log\left(1+\frac{y}{N}\right)+R(y,N),
\end{split}
\end{equation}
where $R(y,N)\leq Cm/N$ for all $y>0$ and all $N\geq N_0$ for some integer $N_0$ and some positive constant $C$.
\end{lemma}

We see that the error term $R(y,N)$ in \eqref{recuree1pm24} is uniformly bounded for all $y>0$ and all $N\geq N_0$ for an integer $N_0$. In particular, by \eqref{recuree1pm71} we see that $R(y,N)\leq 2C\beta(k)$ for all large $N$. In terms of $\overline{g}_{m,n}$ defined in \eqref{recuree1pm1309hb-1-kjh-1}, equation in \eqref{recuree1pm24} takes the form
\begin{equation}\label{recuree1pm25}
\begin{split}
\log(\overline{g}_{m,n}(x)-1)=&\ -\log \binom{m+2n}{m} N!+e^x+Nx\\
&\ -2\log Q_m(e^x)-\log\left(1+\frac{e^x}{N}\right)+R(e^x,N)-\log 2
\end{split}
\end{equation}
when $\overline{\chi}_{m,n}(z)=(z+1)/2$, or takes the form
\begin{equation}\label{recuree1pm25-fur}
\begin{split}
\log(\overline{g}_{m,n}(x)-1)=&\ -\log \binom{m+2n}{m} N!+e^x+Nx\\
&\ -2\log Q_m(e^x)-\log\left(1+\frac{e^x}{N}\right)+R(e^x,N)
\end{split}
\end{equation}
when $\overline{\chi}_{m,n}(z)=z$. In the proof of the following lemmas, we shall consider both the two estimates in \eqref{recuree1pm25} and \eqref{recuree1pm25-fur} for $\overline{g}_{m,n}$ with different pairs $(m,2n)$.

\begin{proof}[Proof of Lemma~\ref{Lemma3}]
The proof of \cite[Lemma~2.3]{Zhang2024} contains some errors, here we correct them.  It is easy to see from \eqref{recuree1pm9} and \eqref{recuree1pm10} that
\begin{equation*}
\begin{split}
h'_{m,n}(y)=\frac{1}{\binom{m+2n}{m}(m+2n)!}\frac{e^yy^{m+2n}}{Q_m(y)^2}.
\end{split}
\end{equation*}
Then, as in the proof of \cite[Lemma~2.3]{Zhang2024}, we obtain by \eqref{recuree1pm10} and Lagrange's version of Taylor's theorem that
\begin{equation}\label{recuree1pm27}
\begin{split}
h_{m,n}(y)-1=\frac{1}{\binom{m+2n}{m}(m+2n+1)!}\frac{e^yy^{m+2n+1}}{Q_m(y)^2}\left(1-U(y)\right),
\end{split}
\end{equation}
where
\begin{equation}\label{recuree1pm27 fur1}
\begin{split}
U(y)=\frac{Q_m(y)^2}{e^yy^{m+2n+1}}\int_0^{y}\frac{e^{u}u^{m+2n+1}[Q_m(u)-2Q'_m(u)]}{Q_m(u)^3}du.
\end{split}
\end{equation}
Note that $N=m+2n+1$. We write
\begin{equation}\label{recuree1pm27 fur2}
\begin{split}
U(y)=\frac{Q_m(y)^2}{e^yy^{N}}\int_0^{y}\frac{e^{u}u^{N}}{Q_m(u)^2}du-2\frac{Q_m(y)^2}{e^yy^{N}}\int_0^{y}\frac{e^{u}u^{N}Q'_m(u)}{Q_m(u)^3}du=G_1-2G_2.
\end{split}
\end{equation}
For $G_1$ in \eqref{recuree1pm27 fur2}, we write
\begin{equation}\label{recuree1pm27 fur2 rqr1}
\begin{split}
G_1=y\int_0^{1}e^{y(s-1)}s^{N}\frac{Q_m(y)^2}{Q_m(sy)^2}du=yI.
\end{split}
\end{equation}
Then from the proof of \cite[Lemma~2.3]{Zhang2024}, we have
\begin{equation}\label{recurfai   5 fur4}
\begin{split}
\left|G_1-\frac{y}{y+N}\right|=\left|Iy-\frac{y}{y+N}\right|\leq \frac{(2m+11)y}{(y+N-2m)^2}.
\end{split}
\end{equation}
Now, for $G_2$ in \eqref{recuree1pm27 fur2}, by the recursive formulas of the coefficients $A_i$ in $g_{m,n}$ in \eqref{recuree1pm9}, we see that $\sum_{i=0}^{\infty}A_i<\infty$ and also that $\sum_{i=0}^{\infty}iA_i<\infty$. Then it is easy to see that there is a constant $C_1$ independent from $m$ such that $C_1mQ_m(y)\geq yQ'_m(y)$ in both of the two cases $0<y \leq 1$ and $y>1$. Thus, for all $y>0$, we have
\begin{equation*}
\begin{split}
\frac{Q'_m(y)}{Q_m(y)}\leq \frac{C_1m}{y}.
\end{split}
\end{equation*}
It follows that
\begin{equation}\label{recurfai   8}
\begin{split}
G_2=\frac{Q_m(y)^2}{e^yy^{N}}\int_0^{y}\frac{e^{u}u^{N}Q'_m(u)}{Q_m(u)^3}du\leq \frac{c}{y}\frac{Q_m(y)^2}{e^yy^{N-1}}\int_0^{y}\frac{e^{u}u^{N-1}}{Q_m(u)^2}du.
\end{split}
\end{equation}
For the right-hand side term above, by the same reasoning as for $G_1$, we have
\begin{equation}\label{recurfai   8-iu}
\begin{split}
G_2&\leq \frac{C_1}{y}\left(\frac{y}{y+N-1}+\frac{(2m+11)y}{(y+N-2m-1)^2}\right)\\
&=C_2\left(\frac{m}{y+N}+\frac{(2m+11)m}{(y+N-2m)^2}\right)
\end{split}
\end{equation}
where $C_2$ is a suitably chosen positive constant. By combining \eqref{recuree1pm27 fur1} and \eqref{recurfai   5 fur4}, we have
\begin{equation*}
\begin{split}
\log(1-U(y))=&\ \log \left(1-G_1+2G_2\right)=\log \left(\frac{N}{y+N}-r(y,N)+2G_2\right)\\
=&\ -\log \left(1+\frac{y}{N}\right)+\log \left(1-\frac{y+N}{N}\left(r(y,N)-2G_2\right)\right),
\end{split}
\end{equation*}
where $r(y,N)$ satisfies
\begin{equation}\label{recurfai   10}
\begin{split}
|r(y,N)|\leq \frac{(2m+11)y}{(y+N-2m)^2}.
\end{split}
\end{equation}
Together with \eqref{recuree1pm27 fur2 rqr1}, \eqref{recurfai   8-iu} and \eqref{recurfai   10}, we have
\begin{equation*}
\begin{split}
\frac{y+N}{N}(|r(y,N)|+2G_2)\leq \frac{2m+11}{N}\left[\frac{(y+N)(y+C_2m)}{(y+N-2m)^2}+\frac{C_2m}{2m+11}\right].
\end{split}
\end{equation*}
By the estimate in \eqref{recuree1pm71} and the definition of $\beta(k)$, we see that there is an integer $N_0$ and a positive constant $C_3$ such that for all $N\geq N_0$ and all $y>0$,
\begin{equation}\label{recurfai   12}
\begin{split}
\frac{y+N}{N}(r(y,N)+2G_2) \leq C_3\frac{m}{N}\leq \frac{1}{2}.
\end{split}
\end{equation}
Therefore, using the inequality that $|\log (1+t)|\leq 2|t|$ for $|t|\leq 1/2$, it follows that
\begin{equation}\label{recurfai   13}
\begin{split}
\log(1-U(y))\leq C_3\frac{m}{N}
\end{split}
\end{equation}
for all $N\geq N_0$ and all $y>0$. Then, by taking the logarithm on both sides of equation \eqref{recuree1pm27} together with \eqref{recurfai   13}, we have the estimate in \eqref{recuree1pm24}. Thus the lemma is proved.

\end{proof}

Also, for technical reasons, we shall not work with $\overline{g}_{m,n}$ directly. Since the function $\overline{g}_{m,n}:\mathbb{R}\to(1,\infty)$ is an increasing homeomorphism, there exists $\overline{s}_{m,n}\in \mathbb{R}$ such that
\begin{equation}\label{trans 0}
\begin{split}
\overline{g}_{m,n}(\overline{s}_{m,n})=2.
\end{split}
\end{equation}
More specifically, when $\overline{\chi}_{m,n}(z)=(z+1)/2$, we have $g_{m,n}(\overline{s}_{m,n})=3$; when $\overline{\chi}_{m,n}(z)=z$, we have $g_{m,n}(\overline{s}_{m,n})=2$. For each pair $(m,2n)$, we shall work with the function $\overline{g}_{m,n}(z+\overline{s}_{m,n})$. Now, for the function $\overline{g}_{m,n}(z)$, we may use \eqref{recuree1pm25} or \eqref{recuree1pm25-fur} and slightly modify the proof of \cite[Lemma~2.4]{Zhang2024} to prove the following

\begin{lemma}\label{Lemma3 trans}
Let $r_0=-1.27846454\cdots$ be the unique real solution of the equation $e^{r_0}+r_0+1=0$. Then the constant $\overline{s}_{m,n}$ in \eqref{trans 0} satisfies
\begin{equation}\label{recuree1pm31=jgg}
\begin{split}
\overline{s}_{m,n}=\log N+r_0+O(\beta(k))\log N
\end{split}
\end{equation}
as $N\to\infty$, with $N=m+2n+1$.
\end{lemma}

\begin{proof}
We first consider the case $g_{m,n}(\overline{s}_{m,n})=2$. We write $\overline{s}_{m,n}=\log N+r$. Then equation \eqref{recuree1pm25-fur} yields
\begin{equation*}
\begin{split}
0=&\ -\log \binom{m+2n}{m}-\log N!+Ne^r+N\log N+Nr\\
&\ -2\log Q_m(Ne^r)-\log(1+e^r)+O(\beta(k)).
\end{split}
\end{equation*}
By Stirling's formula we have
\begin{equation}\label{trans 2}
\begin{split}
N(e^r+r+1)-\log(1+e^r)-2\log Q_m(Ne^r)=\log\binom{m+2n}{m}+\frac{1}{2}\log N+O(1).
\end{split}
\end{equation}
Note that
\begin{equation}\label{trans 3}
\begin{split}
0<\log \binom{m+2n}{m}=\log\frac{(m+2n)!}{m!(2n)!}\leq m\log N.
\end{split}
\end{equation}
Since $\sum_{i=0}^{\infty}A_i<\infty$, we see that $\log Q_m(Ne^r) \leq D_1m\log (Ne^{r})$ for all $r>0$ and some positive constant $D_1$ and also that $\log Q_m(Ne^r)\leq D_2m\log N$ for all $r\leq 0$ and some positive constant $D_2$. Moreover, by the estimate in \eqref{recuree1pm16  fajr-2}, as well as the estimate in \eqref{recuree1pm71}, we have $\log N_k=O(\log k)$ and thus $(m\log N)/N\to 0$ as $N\to\infty$. Then we obtain a contradiction from the identity in \eqref{trans 2} if we let $r\to\infty$ or $r\to-\infty$. This implies that $|r|=O(1)$ and further that $r=r_0+o(1)$ as $N\to\infty$. We write $r=r_0+t$ so that $t=o(1)$ as $N\to\infty$. We obtain from \eqref{trans 2} that
\begin{equation}\label{trans 4}
\begin{split}
0=-Nr_0t-\frac{1}{2}\log N+\log\binom{m+2n}{m}+2\log Q_m(Ne^r)+O(1)+O(Nt^2).
\end{split}
\end{equation}
By \eqref{trans 3} and the observation on $\log Q_m(Ne^r)$, this first yields that
\begin{equation*}
\begin{split}
r_0t=-\frac{1}{2}\frac{\log N}{N}+O(\beta(k))\log N+O\left(\frac{1}{N}\right)+O(t^2)
\end{split}
\end{equation*}
and hence
\begin{equation*}
\begin{split}
t=O(\beta(k))\log N.
\end{split}
\end{equation*}
For the case when $g_{m,n}(\overline{s}_{m,n})=3$, similar consideration also yields the estimate in \eqref{recuree1pm31=jgg}.
\end{proof}

For the two positive integers $M,N$ such that $M>N$, from now on we shall assume that there exists a constant $C>1$ such that $\mathfrak{m}+2\mathfrak{n}\leq C(m+2n)$. Then clearly
\begin{equation}\label{recuree1pm31}
\begin{split}
M\leq CN.
\end{split}
\end{equation}
The estimate in \eqref{recuree1pm73} implies that $C\to1$ as $N\to\infty$. Below we begin to analyse the asymptotic behaviors of $\overline{\phi}_k$ defined in \eqref{diffeo}.

For the choice of $(m_k)$ and $(n_k)$ we know that $\overline{\phi}_1=\phi_1$ is the identity in the plane and the asymptotic behaviors of $\overline{\phi}_2$ has been given in \eqref{recuree1pm12-new-1} and \eqref{recuree1pm12-new-2}. It is easy to see that all fixed point of $\overline{\phi}_2$ are near the origin. For the case $k\geq 3$, we have $\overline{\phi}_k=\phi_k$ and thus from the proof of \cite[Lemma~2.5]{Zhang2024} to obtain the following

\begin{lemma}\label{Lemma4}
Let $M,N$ be two integers as before. Then
\begin{enumerate}
  \item [(1)] if $\mathfrak{m}>m$, then $\overline{\phi}$ has at least one fixed point $p$ and each fixed point $p$ satisfies $p=\log N+O(1)$ as $k\to\infty$;
  \item [(2)] if $\mathfrak{m}=m$, then $\overline{\phi}$ has uniquely one fixed point $p$ such that $p=\log N+O(1)$ as $k\to\infty$;
  \item [(3)] if $\mathfrak{m}<m$, then $\overline{\phi}$ has uniquely one fixed point $p$ such that $p=(1-d)\log N+O(1)$ for a constant $d<1$ as $k\to\infty$;
  \item [(4)] in all of the above three cases, $\overline{\phi}(x)<x$ for $x<p+O(1)$ and $\overline{\phi}(x)>x$ for $x>p+O(1)$ and, moreover, we have $\overline{\phi}(x)\leq \log N+O(1)$ for all $x\leq\log N$.
\end{enumerate}

\end{lemma}

By the descriptions in Lemma~\ref{Lemma4}, we may present a uniform asymptotic relation for $\overline{\phi}(x)$, as well as for $|\overline{\phi}'(x)-1|$ for $x\geq\log N$ and all large $N$. The exact value of the constant $d$ in (3) of the above lemma is not important and we only need to know that $d\geq 0$. (In~\cite[Lemma~2.5]{Zhang2024}, the constant $d$ is chosen to be $1/(M-N)$ when $1<\gamma \leq 2$ and $d=0$ when $\gamma>2$, but the proof there shows that $d$ should be chosen a larger one when $1<\gamma\leq 2$.) Below we shall always assume that the terms $O(1)$ appearing are only dependent on the constant $C$, not on other variables. By the estimates in \eqref{recuree1pm25} or \eqref{recuree1pm25-fur}, we may slightly modify the proof of \cite[Lemma~2.6]{Zhang2024} to obtain the following

\begin{lemma}\label{Lemma5}
For $\mathfrak{m},\mathfrak{n},m,n,M,N$ and $\overline{\phi}:\mathbb{R}\to\mathbb{R}$ there exist positive constants $c_1,\cdots,c_8$ depending only on the constant $C$ in \eqref{recuree1pm31} such that
\begin{equation*}
\begin{split}
|\overline{\phi}(x)-x|\leq c_1e^{-x/2}\leq \frac{c_1}{N} \quad \text{for} \quad x>8\log N
\end{split}
\end{equation*}
and
\begin{equation*}
\begin{split}
\left|\overline{\phi}(x)-\frac{M}{N}x+\frac{1}{N}\log\frac{M!}{N!}\right|\leq c_2e^x \quad \text{for} \quad x<\log N.
\end{split}
\end{equation*}
Moreover, with $\overline{s}_{m,n}$ defined in \eqref{trans 0} we have
\begin{equation}\label{recuree1pm38}
\begin{split}
\left|\overline{\phi}(x)-x\right|\leq c_3 \quad \text{for} \quad x>\overline{s}_{m,n}
\end{split}
\end{equation}
and
\begin{equation*}
\begin{split}
\left|\overline{\phi}(x)-\frac{M}{N}x+\frac{1}{N}\log\frac{M!}{N!}\right|\leq c_4 \quad \text{for} \quad x<\log N.
\end{split}
\end{equation*}
Finally,
\begin{equation*}
\begin{split}
|\overline{\phi}'(x)-1|\leq c_5\left(e^{-x/2}+\beta(k)\right)\leq 2c_5\beta(k) \quad \text{for} \quad x>8\log N
\end{split}
\end{equation*}
and
\begin{equation*}
\begin{split}
\left|\overline{\phi}'(x)-\frac{M}{N}\right|\leq \frac{c_6e^x}{N} \quad \text{for} \quad x<\log N,
\end{split}
\end{equation*}
as well as
\begin{equation*}
\begin{split}
c_7\leq |\overline{\phi}'(x)|\leq c_8 \quad \text{for all} \quad x\in \mathbb{R}.
\end{split}
\end{equation*}

\end{lemma}

\subsection{Definition of a quasiregular map}\label{Definition of a quasiregular map-add} 

We begin to define a quasiregular map using the functions $\overline{g}_{m,n}$ defined in \eqref{recuree1pm1309hb-1-kjh-1}. By modifying these functions slightly, we obtain closely related functions $u_{m,n}$ and $v_{m,n}$ and then glue restrictions of these maps to half-strips along horizontal lines to obtain quasiregular maps $U$ and $V$ which are defined
in the right and left half-planes respectively. Then we will glue these functions along the imaginary axis to obtain a quasiregular map $G$ in the plane.

Let $(m_k)$ and $(n_k)$ be the two sequences from Lemma~\ref{Lemma1} and Lemma~\ref{Lemma2} respectively. Then we give explicit expression for $\overline{g}_{m,n}$ defined in \eqref{recuree1pm1309hb-1-kjh-1}. On each strip
\begin{equation*}
\begin{split}
\hat{\Xi}_k=\{x+iy: \quad y\geq0, \quad 2\pi\mathcal{N}_{k-1}<y< 2\pi\mathcal{N}_k\}
\end{split}
\end{equation*}
of the upper half-plane we have
\begin{equation}\label{recuree1-main eq1}
\begin{split}
\overline{g}_{m_k,n_k}(z)=g_{m_k,n_k}\left(\frac{x+iy-2i\pi\mathcal{N}_{k-1}}{N_k}\right),
\end{split}
\end{equation}
if $k=1,2$ and
\begin{equation}\label{recuree1-main eq2}
\begin{split}
\overline{g}_{m_k,n_k}(z)=
\frac{1}{2}\left[g_{m_k,n_k}\left(\frac{x+iy-2i\pi\mathcal{N}_{k-1}}{N_k}\right)+1\right]
\end{split}
\end{equation}
if $k\geq 3$. On the other hand, denoting $N_{0}=0$, $\mathcal{N}_{0}=0$ and $\mathcal{N}_{-k}=-\mathcal{N}_{k}$, on each strip
\begin{equation*}
\check{\Xi}_{-k}=\{z=x+iy: \quad y\leq 0, \quad 2\pi\mathcal{N}_{-k}\leq y\leq 2\pi \mathcal{N}_{1-k}\}
\end{equation*}
of the lower half-plane we have
\begin{equation}\label{recuree1-main eq3}
\begin{split}
\overline{g}_{m_k,n_k}(z)=g_{m_k,n_k}\left(\frac{x+iy-2i\pi\mathcal{N}_{1-k}}{N_{k}}\right).
\end{split}
\end{equation}
We will use the above expressions for $z$ in the left half-plane, that is, for $z=x+iy$ such that $x\leq 0$. For $z$ in the right half-plane, say in the strips $\hat{\Pi}_k=\{z=x+iy: 2\pi(k-1)\leq y\leq 2\pi k\}$ of the upper half-plane and the strips $\check{\Pi}_{-k}=\{z=x+iy: 2\pi(-k)\leq y\leq 2\pi (1-k)\}$ of the lower half-plane, we only need to define $\overline{g}_{m_k,n_k}(z)$ by replacing $N_k,\mathcal{N}_{k-1},\mathcal{N}_{1-k}$ in \eqref{recuree1-main eq1}, \eqref{recuree1-main eq2} and \eqref{recuree1-main eq3} by $1,k-1,1-k$ respectively. In view of the explicit expressions of $\overline{g}_{m_k,n_k}$ defined in \eqref{recuree1-main eq1}, \eqref{recuree1-main eq2} and \eqref{recuree1-main eq3}, the maps $U$, $V$ and $G$ will not commute with the complex conjugation. However, we shall only define them in the upper half-plane since the definitions in the lower half-plane are similar.

We begin by constructing the map $U$. Instead of $\overline{g}_{m,n}$ defined in \eqref{recuree1pm1309hb-1-kjh-1}, we consider the map
\begin{equation*}
\begin{split}
u_{m,n}: \{z\in\mathbb{C}: \text{Re}\ z\geq0\}\to\mathbb{C}, \quad u_{m,n}(z)=\overline{g}_{m,n}(z+\overline{s}_{m,n}).
\end{split}
\end{equation*}
Note that $\overline{g}_{m,n}$ is increasing on the real line and maps $[0,\infty)$ onto $[2,\infty)$. Put $U(z)=u_{m_k,n_k}(z)$ in the half-strip
\begin{equation*}
\begin{split}
\hat{\Pi}_k^{+}=\{x+iy: x>0, \quad  2\pi(k-1)<y<2\pi k\}.
\end{split}
\end{equation*}
The function $U$ will be discontinuous on the horizontal lines $L^{+}_{k-1}=\{z=x+iy: x>0, y=2\pi(k-1)\}$ and $L^{+}_k=\{z=x+iy: x>0, y=2\pi k\}$. In order to obtain a continuous function we consider the function $\overline{\psi}_k:[0,\infty)\to[0,\infty)$ defined by $u_{m_{k+1},n_{k+1}}(x)=u_{m_k,n_k}(\overline{\psi}_k(x))$. The function $\overline{\psi}_k$ is closely related to the function $\overline{\phi}$ considered in Lemma~\ref{Lemma4} and Lemma~\ref{Lemma5} and also the function $\overline{\phi}_2$ defined in \eqref{diffeo-s-2}. In fact, denoting by $\overline{\phi}_k$ the function $\overline{\phi}$ corresponding to $(m_k,2n_k)$ and $(m_{k+1},2n_{k+1})$, we have
\begin{equation}\label{recuree1-main eqiqu-uy-1}
\begin{split}
\overline{\psi}_k(x)=\overline{\phi}_k(x+\overline{s}_{m_{k+1},n_{k+1}})-\overline{s}_{m_{k},n_{k}}.
\end{split}
\end{equation}
We then define $U:\{z\in\mathbb{C}: \text{Re}\ z\geq0\}\to \mathbb{C}$ by interpolating between $u_{m_{k+1},n_{k+1}}$ and $u_{m_{k},n_{k}}$ as follows: if $2\pi (k-1)\leq y<2\pi k$, say $y=2\pi (k-1)+2\pi t$ where $0\leq t<1$, then we put
\begin{equation*}
\begin{split}
U(x+iy)=u_{m_{k},n_{k}}((1-t)x+t\overline{\psi}_k(x)+iy)=u_{m_{k},n_{k}}(x+iy+t(\overline{\psi}_k(x)-x)).
\end{split}
\end{equation*}

We now define a function $V$ in the left half-plane. In order to do so, we define
\begin{equation*}
\begin{split}
v_{m,n}:\{z\in\mathbb{C}: \text{Re}\ z< 0\}\to \mathbb{C},  \quad v_{m,n}(z)=\overline{g}_{m,n}\left(\frac{z}{m+2n+1}+\overline{s}_{m,n}\right).
\end{split}
\end{equation*}
Note that $v_{m,n}$ maps $(-\infty,0]$ monotonically onto $(1,2]$. Let $(m_k)$ and $(n_k)$ be as before and write $N_k=m_k+2n_k+1$. Put $\mathcal{N}_k=\sum_{j=1}^{k}N_j$. This time we would like to define $V(z)=v_{m_k,n_k}(z)$ in the half-strip
\begin{equation*}
\begin{split}
\hat{\Xi}_k^{-}=\{x+iy: x<0, \ 2\pi \mathcal{N}_{k-1}\leq y<2\pi \mathcal{N}_k\},
\end{split}
\end{equation*}
but again this function would be discontinuous on the horizontal lines $\mathcal{L}^{-}_{k-1}=\{z=x+iy: x<0, y=2\pi\mathcal{N}_{k-1}\}$ and $\mathcal{L}^{-}_k=\{z=x+iy: x<0, y=2\pi\mathcal{N}_{k}\}$. In order to obtain a continuous function we again interpolate between $v_{m_{k+1},n_{k+1}}$ and $v_{m_k,n_k}$. Similarly as before we consider the map $\overline{\psi}_k:(-\infty,0]\to (-\infty,0]$ defined by $v_{m_{k+1},n_{k+1}}(x)=v_{m_k,n_k}(\overline{\psi}_k(x))$. Denoting by $\overline{\phi}_k$ the function $\overline{\phi}$ corresponding to $(m_k,n_k)$ and $(m_{k+1},n_{k+1})$, we have
\begin{equation}\label{recuree1-main eqiqu-uy-2}
\begin{split}
\overline{\psi}_k(x)=N_k\overline{\phi}_k\left(\frac{x}{N_{k+1}}+\overline{s}_{m_{k+1},n_{k+1}}\right)-N_k\overline{s}_{m_{k},n_{k}}.
\end{split}
\end{equation}
Then we define $V:\{z\in \mathbb{C}:\text{Re}\ z\leq 0\}\to \mathbb{C}$ by interpolating between $v_{m_{k+1},n_{k+1}}$ and $v_{m_k,n_k}$ as follows: if $2\pi \mathcal{N}_{k-1}\leq y<2\pi \mathcal{N}_k$, say $y=2\pi \mathcal{N}_{k-1}+2\pi N_kt$ where $0\leq t<1$, then we put
\begin{equation*}
\begin{split}
V(x+iy)=v_{m_{k},n_{k}}((1-t)x+t\overline{\psi}_k(x)+iy)=v_{m_{k},n_{k}}(x+iy+t(\overline{\psi}_k(x)-x)).
\end{split}
\end{equation*}
This map $V$ is continuous on the horizontal lines $\mathcal{L}^{-}_k$ for any $k$ in the left half-plane.

Now we define our map $G$ by gluing $U$ and $V$ along the imaginary axis. In order to do this we note that by construction of $(m_k)$ and $(n_k)$ in Lemma~\ref{Lemma1} and Lemma~\ref{Lemma2} we have $U(iy)=V(ih(y))$ and thus $U(ig(y))=V(iy^{\gamma})$ for $y\geq 0$ with the maps $h$ and $g$ from Lemma~\ref{Lemma2}. Therefore, we will consider a homeomorphism $Q$ of the right half-plane
\begin{equation*}
\begin{split}
\mathbb{H}^{+}=\{z\in\mathbb{C}:\text{Re}\ z\geq 0\}
\end{split}
\end{equation*}
onto itself, satisfying $Q(\overline{z})=\overline{Q(z)}$, such that $Q(\pm iy)=\pm ig(y)$ for $y\geq 0$ while $Q(z)=z$ for $\text{Re}\ z\geq 1$. We thus have to define $Q(z)$ in the strip $\{z\in\mathbb{C}:0<\text{Re}\ z<1\}$. For $|\text{Im}\ z|\geq 1$ we define $Q$ in this strip by interpolation; that is, we put
\begin{equation*}
\begin{split}
Q(x\pm iy)=x\pm i((1-x)g(y)+xy)  \quad
 \text{if} \quad 0<x<1 \quad \text{and} \quad y\geq 1.
\end{split}
\end{equation*}
In the remaining part of the strip we define $Q$ by
\begin{equation*}
\begin{split}
Q(z)=\left\{
       \begin{array}{ll}
         z|z|^{\gamma-1}, & \text{if} \quad 0<|z|<1; \\
         z,               & \text{if} \quad |z|>1, \quad \text{but} \quad  0<\text{Re}\ z<1 \quad \text{and} \quad 0\leq |\text{Im}\ z|\leq 1.
       \end{array}
     \right.
\end{split}
\end{equation*}
Note that $h(y)=y$ for $0\leq y\leq 2\pi$ since $m_1=n_1=0$ implies that $N_1=1$, and thus $g(y)=y^{\gamma}$ for $0\leq y<1$. Thus $Q(\pm i y)=\pm ig(y)$ also for $|y|\leq 1$. Moreover, we have $g(1)=1$, meaning that the above expressions for $Q$ do indeed coincide for $|\text{Im}\ z|=1$. We conclude that the map $W=U\circ Q$ satisfies
\begin{equation*}
\begin{split}
W(\pm iy)=U(\pm ig(y))=V(\pm iy^{\gamma}) \quad \text{for} \quad y\geq 0.
\end{split}
\end{equation*}

Let now $\rho\in(1/2,1)$. We let $\gamma=1/(2\rho-1)$ and put $\sigma=\rho/(2\rho-1)$. The hypothesis that $1/2<\rho<1$ corresponds to $\gamma>1$ as well as $\sigma>1$. Then we define the map
\begin{equation}\label{Def-1}
\begin{split}
G(z)=\left\{
       \begin{array}{ll}
         W(z^{\rho}),         & \text{if} \quad |\arg z|\leq \frac{\pi}{2\rho}; \\
         V(-(-z)^{\sigma}),   & \text{if} \quad |\arg(-z)|\leq \frac{\pi}{2\sigma},
       \end{array}
      \right.
\end{split}
\end{equation}
which is continuous on the imaginary axis and on the horizontal lines $L^{+}_k$ in the right half-plane and $\mathcal{L}^{-}_{k}$ in the left half-plane for any integer $k$ respectively. Here, for $\mu>0$, we denote by $z^{\mu}$ the principal branch of the power which is defined in $\mathbb{C}\setminus(-\infty,0]$.

\subsection{Completion of the proof}\label{Completion of the proof-add} 

With the definition of $G$ in \eqref{Def-1}, here we first prove Theorem~\ref{maintheorem5} under the assumption that $\rho\in(1/2,1)$. With the expression of $\overline{\psi}_k$ in \eqref{recuree1-main eqiqu-uy-1}, we have $\overline{\psi}_{k}(x)-x=\overline{\phi}_{k}(x+\overline{s}_{m_{k+1},n_{k+1}})-(x+\overline{s}_{m_{k+1},n_{k+1}})+(\overline{s}_{m_{k+1},n_{k+1}}-\overline{s}_{m_k,n_k})$. Thus by Lemma~\ref{Lemma3 trans} and Lemma~\ref{Lemma5} we have $|\overline{\psi}_{k}(x)-x|=O(1)$ for $0\leq x\leq 8\log N_k$ and
\begin{equation*}
\begin{split}
|\overline{\psi}_{k}(x)-x|\leq \frac{c_1}{N_k}+|\log N_{k+1}-\log N_k|+O(\beta(k))\log N_k
\end{split}
\end{equation*}
for $x\leq 8\log N_k$. By the estimate in \eqref{recuree1pm16  fajr-2}, as well as the estimate in \eqref{recuree1pm73}, we see that
\begin{equation}\label{recuree1pm95-poi-1}
\begin{split}
|\overline{\psi}_{k}(x)-x|=O(\beta(k))\log k
\end{split}
\end{equation}
for $x>8\log N_k$. Similarly, with the expression of $\overline{\psi}_k$ in \eqref{recuree1-main eqiqu-uy-2} we have
\begin{equation*}
\begin{split}
\frac{1}{N_k}&\left|\overline{\psi}_k(x)-x\right|
=\left|\overline{\phi}_k\left(\frac{x}{N_{k+1}}+\overline{s}_{m_{k+1},n_{k+1}}\right)-\frac{x}{N_{k}}-\overline{s}_{m_{k},n_{k}}\right|\\
\leq &\ \left|\overline{\phi}_k\left(\frac{x}{N_{k+1}}+\overline{s}_{m_{k+1},n_{k+1}}\right)-\frac{N_{k+1}}{N_{k}}\left(\frac{x}{N_{k+1}}+\overline{s}_{m_{k+1},n_{k+1}}\right)+\frac{1}{N_k}\log\frac{N_{k+1}!}{N_k!}\right|\\
&+\left|\frac{N_{k+1}}{N_{k}}\overline{s}_{m_{k+1},n_{k+1}}-\overline{s}_{m_{k},n_{k}}-\frac{1}{N_k}\log\frac{N_{k+1}!}{N_k!}\right|.
\end{split}
\end{equation*}
By Lemma~\ref{Lemma3 trans} and Lemma~\ref{Lemma5} we have $|\overline{\psi}_{k}(x)-x|=O(1)$ for $-2N_k\log N_k\leq x\leq 0$ and, together with the estimate in \eqref{recuree1pm16  fajr-2} as well as the estimate in \eqref{recuree1pm73}, we easily show that
\begin{equation}\label{recuree1pm95-poi-2}
\begin{split}
\frac{1}{N_k}|\overline{\psi}_k(x)-x|=O(\beta(k))\log k
\end{split}
\end{equation}
for $x<-2N_k\log N_k$. With the error term $O(\beta(k))\log k$, instead of the error term $O(1/k^{\min\{1,\gamma-1\}})$, then we can follow exactly the same process as that in~\cite[Subsection~3.3]{Bergweilereremenko2019} and the comments in \cite[Subsection~2.3]{Zhang2024}, together with the inequalities in Lemma~\ref{Lemma5}, to estimate the dilatation of $G$ and show that $G$ is quasiregular in the plane and also satisfies the hypothesis of the Teichm\"uller--Wittich--Belinskii theorem.  This theorem, together with the existence theorem for quasiconformal mappings, yields that there exists a quasiconformal homeomorphism $\tau:\mathbb{C}\to\mathbb{C}$ and a meromorphic function $F$ such that
\begin{equation}\label{recuree1pm117}
\begin{split}
G(z)=F(\tau(z)) \quad \text{and} \quad \tau(z)\sim z \quad \text{as} \quad z\to\infty.
\end{split}
\end{equation}
Then $F=G\circ \tau^{-1}$ is a locally univalent meromorphic function. Define $E=F/F'$. Then $E$ is a Bank--Laine function. Clearly, $E=f_1f_2$ for two normalized solutions of the second order differential equation \eqref{bank-laine0} with an entire coefficient $A$.

Below we use similar arguments as in Subsection~\ref{Completion of the proof} to estimate the numbers of poles and zeros of $F$ respectively and also to determine the orders of $E$ and $A$. We need some facts about the construction of $U$ and $V$. Now we are ready to estimate the numbers of poles and zeros of $F$.

We first estimate the number of poles of $F$. We note from Lemma~\ref{Lemma1} that $\sum_{i=1}^{k}m_i\sim (2\pi)^{\delta\gamma-1}k^{\delta\gamma}$ when $0\leq \delta\gamma\leq 1$ and  $\sum_{i=1}^{k}m_i\sim (2\pi)^{\delta\gamma-1}(\log(k+2\pi))^{-2}k^{\delta\gamma}$ when $\delta\gamma>1$. Let $r>0$ and choose $k\in \mathbb{N}$ such that $2\pi(k-1)<r\leq 2\pi k$. It follows from the construction of $U$ and \eqref{recuree1pm95-poi-1} that the number of poles of $U$ satisfies
\begin{equation}\label{recuree1pm118o-0ijyt}
\begin{split}
\sum_{i=1}^{k}m_i \leq n(r,\infty,U)\leq 2\sum_{i=1}^{k}m_i
\end{split}
\end{equation}
for large $r$. Thus the double inequalities in \eqref{recuree1pm118o-0ijyt} yield $n(r,\infty,U)\asymp r^{\delta\gamma}$ when $0\leq \delta\gamma\leq 1$ and $n(r,\infty,U)\asymp (\log(r+2\pi))^{-2}r^{\delta\gamma}$ when $\delta\gamma>1$. Similarly, we choose $k\in \mathbb{N}$ such that $2\pi\mathcal{N}_{k-1} <r \leq 2\pi \mathcal{N}_k$. Recall that $\mathcal{N}_k\sim (2\pi)^{\gamma-1}k^{\gamma}$ as $k\to\infty$ by \eqref{recuree1pm109}. It follows from the construction of $V$ and \eqref{recuree1pm95-poi-2} that the number of poles of $V$ satisfies
\begin{equation}\label{recuree1pm118o-0ijyt-1}
\begin{split}
\sum_{i=1}^{k}m_i \leq n(r,\infty,V)\leq 2\sum_{i=1}^{k}m_i
\end{split}
\end{equation}
for large $r$. Then the double inequalities in \eqref{recuree1pm118o-0ijyt-1} imply that $n(r,\infty,V)\asymp r^{\delta}$ when $0\leq \delta\gamma\leq 1$ and $n(r,\infty,V)\asymp (\log(r+2\pi))^{-2}r^{\delta}$ when $\delta\gamma>1$. Denote $\lambda=\delta\rho\gamma$. Since $\sigma=\rho\gamma$, we conclude that the number of poles of $G$ satisfies $n(r,\infty,G)\asymp r^{\lambda}$ when $0\leq \delta\gamma\leq 1$ and $n(r,\infty,G)\asymp (\log(r+2\pi))^{-2}r^{\lambda}$ when $\delta\gamma>1$. By the relations in \eqref{recuree1pm117}, we have $n(r,\infty,F)\asymp r^{\lambda}$ when $0\leq \delta\gamma\leq 1$ and $n(r,\infty,F)\asymp (\log(r+2\pi))^{-2}r^{\lambda}$ when $\delta\gamma>1$. Thus $\lambda(f_1)=\lambda$.

Then we estimate the number of zeros of $F$. To this end, we first estimate the number of zeros of $g_{m_k,n_k}+1$ uniformly for all pairs $(m_k,2n_k)$. Recall from \cite[p.~7]{Hayman1964Meromorphic} that the Nevanlinna characteristic of the function $\exp(e^z)$ satisfies
\begin{equation}\label{characteristic N}
\begin{split}
T(r,\exp(e^z))\sim\frac{e^r}{(2\pi^3r)^{1/2}}, \quad r\to\infty.
\end{split}
\end{equation}
By the recursive formulas of the coefficients $A_i$ and $B_j$ in $g_{m,n}$ in \eqref{recuree1pm9}, we see that $\sum_{i=0}^{\infty}A_i<\infty$ and $\sum_{j=0}^{\infty}|B_j|<\infty$ and also that $\sum_{i=0}^{\infty}iA_i<\infty$ and $\sum_{j=0}^{\infty}|jB_j|<\infty$. For each pair $(m,2n)$, we write $g_{m,n}$ in \eqref{recuree1pm9} as $g_{m,n}(z)=\mathcal{R}(z)\exp(e^z)$ for a function $\mathcal{R}(z)=R_{m,n}(e^z)$. Then we have $T(r,g_{m,n}+1)=T(r,g_{m,n})+O(1)$ and thus
\begin{equation}\label{characteristic N-1}
\begin{split}
T(r,g_{m,n}+1)=T(r,\exp(e^z))+O(T(r,e^{mz}))+O(T(r,e^{nz}))+O(1).
\end{split}
\end{equation}
On the other hand, we also have the identity
\begin{equation*}
\begin{split}
\frac{1}{g_{m,n}(z)+1}=1-\frac{1}{\mathcal{R}'/\mathcal{R}+e^z}\frac{(g_{m,n}(z)+1)'}{g_{m,n}(z)+1},
\end{split}
\end{equation*}
which and the lemma on the logarithmic derivative and the first main theorem imply that
\begin{equation}\label{characteristic N-3}
\begin{split}
m\left(r,\frac{1}{g_{m,n}+1}\right)&=O(\log rT(2r,g_{m,n}))+O\left(T(r,\mathcal{R}'/\mathcal{R}+e^z)\right)+O(1)\\
&=O(\log rT(2r,g_{m,n}))+O((m+2n+1)T(r,e^{z}))+O(1).
\end{split}
\end{equation}
Recall that $N_k=m_k+2n_k+1\sim \gamma(2\pi k)^{\gamma-1}$ as $k\to\infty$ by \eqref{recuree1pm16  fajr-2} and $\mathcal{N}_k\sim (2\pi)^{\gamma-1}k^{\gamma}$ as $k\to\infty$ by \eqref{recuree1pm109}. Therefore, for the two sequences $(m_k)$ in Lemma~\ref{Lemma1} and $(n_k)$ in Lemma~\ref{Lemma2}, if $k\asymp r$ or $k^{\gamma}\asymp r$, as we assume in the following, then by \eqref{characteristic N} and \eqref{characteristic N-1} we have $\log T(r,g_{m_k,n_k}+1)\asymp r$ uniformly for all pairs $(m_k,2n_k)$. Together with \eqref{characteristic N-3}, the first main theorem implies that $\log N(r,0,g_{m_k,n_k}+1)\asymp r$, and hence $\log n(r,0,g_{m_k,n_k}+1)\asymp r$, uniformly for all pairs $(m_k,2n_k)$.

First, let $r>0$ and choose $k\in \mathbb{N}$ such that $2\pi(k-1)<r\leq 2\pi k$. Denote by $\hat{n}(r,0,\overline{g}_{m_k,n_k})$ the number of zeros of $\overline{g}_{m_k,n_k}$ in the first quadrant of the plane. Then, for all pairs $(m_k,2n_k)$, we have
\begin{equation}\label{recuree1pm118gkgy-1}
\begin{split}
\log\hat{n}(r,0,\overline{g}_{m_k,n_k})\asymp r.
\end{split}
\end{equation}
Denote by $\hat{n}(\ddot{r},0,U)$ the number of zeros of $U$ in the square region $\{z=x+iy:|x|+|y|\leq r\}$ intersecting the upper half-plane. It follows from the construction of $U$ that the number of zeros of $U$ in the upper half-plane satisfies
\begin{equation}\label{recuree1pm118gkgy-2}
\begin{split}
\sum_{i=1}^{\ddot{k}}\frac{1}{\ddot{k}}\hat{n}(\ddot{r}/2,0,u_{m_i,n_i})\leq \hat{n}(r,0,U)\leq \sum_{i=1}^{k}\frac{1}{k}\hat{n}(2\ddot{r},0,u_{m_i,n_i})
\end{split}
\end{equation}
for large $r$ and some integer $\ddot{k}$. Then by \eqref{recuree1pm118gkgy-1} we see that $\log(\sum_{i=1}^{k}\hat{n}(2\ddot{r},0,u_{m_i,n_i})/k))\asymp r$. For the quantity in the left-hand side of \eqref{recuree1pm118gkgy-2}, we have similar asymptotic relations. Then the double inequalities in \eqref{recuree1pm118gkgy-2} imply that $\log\hat{n}(r,0,U)\asymp r$. On the other hand, we choose $k\in \mathbb{N}$ such that $2\pi\mathcal{N}_{k-1} <r \leq 2\pi \mathcal{N}_k$. It follows from the construction of $V$ that the number of zeros of $V$ in the upper half-plane satisfies
\begin{equation}\label{recuree1pm118gkgy-3}
\begin{split}
\sum_{i=1}^{\ddot{k}}\frac{N_i}{\mathcal{N}_{\ddot{k}}}\hat{n}(\ddot{r}/2,0,v_{m_i,n_i})\leq \hat{n}(r,0,V)\leq \sum_{i=1}^{k}\frac{N_i}{\mathcal{N}_k}\hat{n}(2\ddot{r},0,v_{m_i,n_i})
\end{split}
\end{equation}
for large $r$ and some integer $\ddot{k}$. Recall that $N_k\sim \gamma(2\pi k)^{\gamma-1}$ as $k\to\infty$ and $\mathcal{N}_k\sim (2\pi)^{\gamma-1}k^{\gamma}$ as $k\to\infty$. Then by \eqref{recuree1pm118gkgy-1} and the double inequalities in \eqref{recuree1pm118gkgy-3} we have $\log\hat{n}(r,0,V)\asymp r/k^{\gamma-1}\asymp r^{1/\gamma}$. Since $\sigma=\rho\gamma$ and since $\mathcal{N}_k\sim (2\pi)^{\gamma-1}k^{\gamma}$ as $k\to\infty$, we conclude from the previous estimates for $\log\hat{n}(r,0,U)$ and $\log\hat{n}(r,0,V)$ that the number of zeros of $G$ in the upper half-plane satisfies $\log \hat{n}(r,0,G)\asymp r^{\rho}$. Similarly, we may estimate the number of zeros of $F$ in the lower half-plane. Let $r>0$ and choose $k\in \mathbb{N}$ such that $2\pi(k-1)<r\leq 2\pi k$. It follows from the construction of $U$ that the number of zeros of $U$ in the lower half-plane satisfies
\begin{equation}\label{recuree1pm118gkgy-4}
\begin{split}
2\sum_{j=1}^{k}n_j \leq \check{n}(r,0,U)\leq 4\sum_{j=1}^{k}n_j
\end{split}
\end{equation}
for large $r$. Since $2\sum_{j=1}^{k}n_j\sim \mathcal{N}_k$ as $k\to\infty$ by \eqref{recuree1pm72}, the double inequalities in \eqref{recuree1pm118gkgy-4} yield $\check{n}(r,0,U)\asymp r^{\gamma}$. On the other hand, if we choose $k\in \mathbb{N}$ such that $2\pi\mathcal{N}_{k-1} <r \leq 2\pi \mathcal{N}_k$, then it follows from the construction of $V$ that the number of zeros of $V$ in the lower half-plane satisfies
\begin{equation}\label{recuree1pm118gkgy-5}
\begin{split}
2\sum_{j=1}^{k}n_j \leq \check{n}(r,0,V)\leq 4\sum_{j=1}^{k}n_j
\end{split}
\end{equation}
for large $r$. Since $2\sum_{j=1}^{k}n_j\sim \mathcal{N}_k$ as $k\to\infty$ by \eqref{recuree1pm72}, the double inequalities in \eqref{recuree1pm118gkgy-5} yield that $\check{n}(r,0,V)\asymp r$. Since $\sigma=\rho\gamma$, we conclude that the number of zeros of $G$ in the lower half-plane satisfies $\check{n}(r,0,G)\asymp r^{\rho\gamma}$. By combining the two estimates for $\hat{n}(r,0,G)$ and $\check{n}(r,0,G)$ together, we conclude that $\log n(r,0,G)\asymp r^{\rho}$. Finally, by the relations in \eqref{recuree1pm117}, the number of zeros of $F$ satisfies $\log n(r,0,F)\asymp r^{\rho}$. Thus $\lambda(f_2)=\infty$.

Now we show that $\rho(A)=\rho$ by using similar arguments as in Subsection~\ref{Completion of the proof}. With the choice of the two sequences $(m_k)$ and $(n_k)$ in Lemma~\ref{Lemma1} and Lemma~\ref{Lemma2} and the explicit expressions of $\overline{g}_{m_k,n_k}$ in \eqref{recuree1-main eq1}, \eqref{recuree1-main eq2} and \eqref{recuree1-main eq3}, we may also suppose that there is a set $\Omega_1$ of finite linear measure such that the two estimates in \eqref{growth-1} and \eqref{growth-2} hold for $\overline{g}_{m_k,n_k}$ for all $z\in \mathbb{C}$ and $|z|\not\in \Omega_1$ and all pairs $(m_k,2n_k)$. Clearly, with the notation $L_k=m_k+2n_k$, this implies that there is a set $\Omega_2$ of finite linear measure such that
\begin{equation*}
\left\{
\begin{array}{ll}
|v_{m_k,n_k}(z)|\leq \exp((L_k+2\varepsilon)r/N_k)\left|\exp(e^{z/N_k})\right| & \text{for} \quad z\in \mathbb{H}^{-} \quad \text{and} \quad |z|\not\in \Omega_2,\\
|u_{m_k,n_k}(z)|\leq \exp((L_k+2\varepsilon)r)\left|\exp(e^{z})\right|     & \text{for} \quad z\in \mathbb{H}^{+} \quad \text{and} \quad |z|\not\in \Omega_2
\end{array}
\right.
\end{equation*}
for all pairs $(m,2n)$. Let $x+iy\in \mathbb{H}^{-}$ with $\text{Im}\ z\geq 0$ and $k\in \mathbb{N}$ with $2\pi \mathcal{N}_{k-1}\leq \text{Im} z<2\pi \mathcal{N}_k$. Then, with $t=(y-2\pi \mathcal{N}_{k-1})/{2\pi N_k}$ where $0\leq t<1$, we have a similar estimate as in \eqref{growth-4}. Since $N_k=L_k-1=m_k\sim\gamma(2\pi k)^{\gamma-1}$ as $k\to\infty$ by \eqref{recuree1pm16  fajr-2} and thus $N_k=O(k^{\gamma-1})=O(|z|^{\gamma-1})$ by the construction of $(m_k)$ in Lemma~\ref{Lemma1} and $(n_k)$ in Lemma~\ref{Lemma2}, this implies that
\begin{equation}\label{recuree1pm125}
\begin{split}
|V(z)|\leq \exp((O(r^{\gamma-1})+3\varepsilon)r)
\end{split}
\end{equation}
for $z\in \mathbb{H}^{-}$ and $|z|\not\in \Omega_2$. On the other hand, for $z\in \mathbb{H}^{+}$ and $|z|\not\in \Omega_1$ we have
\begin{equation*}
\begin{split}
|\overline{g}_{m_k,n_k}(z)|\leq \exp((L_k+2\varepsilon)r)\exp\left(e^{\text{Re}\ z}\right).
\end{split}
\end{equation*}
Again, assuming that $\text{Im}\ z\geq 0$, we choose $k\in \mathbb{N}$ such that $2\pi (k-1)\leq \text{Im}\ z <2\pi k$. Then, with $U(z)=u_{m_k,n_k}(q(z))=\overline{g}_{m_k,n_k}(q(z)+\overline{s}_{m_k,n_k})$ where $q(z)$ is defined by $q(x+iy)=x+iy+t(\overline{\psi}_k(x)-x)$, we have
\begin{equation}\label{recuree1pm126-jih}
\begin{split}
|U(z)|\leq \exp((L_k+2\varepsilon)\text{Re}(q(z))+\overline{s}_{m_k,n_k})\exp\left(e^{\text{Re}(q(z))+\overline{s}_{m_k,n_k}}\right)
\end{split}
\end{equation}
for $z\in \mathbb{H}^{+}$ and $|z|\not\in \Omega_2$. We have by Lemma~\ref{Lemma3 trans} that
\begin{equation}\label{recuree1pm128}
\begin{split}
\overline{s}_{m_k,n_k}=\log N_k+O(1)=O(\log |z|).
\end{split}
\end{equation}
Then, since $L_k=N_k-1$, we deduce from \eqref{recuree1pm38}, \eqref{recuree1pm126-jih} and \eqref{recuree1pm128} that
\begin{equation*}
\begin{split}
\log|U(z)| \leq&\ (L_k+2\varepsilon)\text{Re}(q(z))+\overline{s}_{m_k,n_k}+\exp(\text{Re}(q(z))+\overline{s}_{m_k,n_k})\\
\leq&\ \exp((1+o(1))|z|)+O(|z|^{\gamma})\leq \exp((1+o(1))|z|)
\end{split}
\end{equation*}
as $z\to\infty$ in $\mathbb{H}^{+}$ and $|z|\not\in \Omega_2$. Together with \eqref{recuree1pm125} we conclude that
\begin{equation*}
\begin{split}
\log|G(z)|\leq \exp((1+o(1))|z|^{\rho})
\end{split}
\end{equation*}
as $|z|\to\infty$ outside a set $\Omega_3$ of finite linear measure and hence \eqref{recuree1pm117} yields that
\begin{equation*}
\begin{split}
\log|F(z)|\leq \exp((1+o(1))|z|^{\rho})
\end{split}
\end{equation*}
as $|z|\to\infty$ outside a set $\Omega_4$ of finite linear measure. Together with \eqref{recuree1pm117}, the lemma on the logarithmic derivative now implies that $E=F/F'$ satisfies
\begin{equation*}
\begin{split}
m\left(r,\frac{1}{E}\right)=O(r^{\rho})
\end{split}
\end{equation*}
outside a set of finite linear measure. Since all zeros of $E$ are simple, the previous estimates on the poles and zeros of $F$ imply that $\log N(r,0,E)\asymp r^{\rho}$. Then the above estimate together with the first main theorem implies that $\log T(r,E)\asymp r^{\rho}$. Then the lemma on the logarithmic derivative, together with \eqref{recuree1pm7}, implies that
\begin{equation*}
\begin{split}
m(r,A)=2m\left(r,\frac{1}{E}\right)+O(\log rT(r,E))=O(r^{\rho}).
\end{split}
\end{equation*}
Thus $\rho(A)\leq \rho$. On the other hand, we note that $g_{1,1}$ is analytic in $\hat{\Pi}_1^{+}$ and in $\check{\Pi}_{-1}^{+}$ as well as in the positive real axis $\mathbb{R}^{+}$. Let $\overline{\eta}:\mathbb{H}^{+}\to \mathbb{H}^{+}$ be the map such that $\overline{\eta}=z^{\rho}$. Denote $\overline{Q}=Q+\overline{s}_{1,1}$. Note that $\overline{\eta}$ and $\overline{Q}$ both commute with the complex conjugation and also that are both analytic for $z$ such that $\text{Re}\ z\geq 1$ in the right half-plane. Then, for $\tau$ in \eqref{recuree1pm117}, by the same arguments as in the proof of Theorem~\ref{maintheorem3} we may choose a curve in $(\tau\circ \overline{\eta}^{-1}\circ \overline{Q}^{-1})(\hat{\Pi}_1^{+}\cup\check{\Pi}_{-1}^{+}\cup \mathbb{R}^{+})$, say $\mathcal{L}^{+}$, such that $(\overline{Q}\circ\overline{\eta}\circ\tau^{-1})(\mathcal{L}^{+})$ is the positive real axis and also that $\tau$ satisfies $|\tau'(z)|\leq 2|z|$ for all $z\in (\overline{\eta}^{-1}\circ \overline{Q}^{-1})(\mathcal{L}^{+})$ such that $|z|$ is large. Then, by the relation in \eqref{recuree1pm7} as well as the definitions of $u_{m,n}$ in Subsection~\ref{Definition of a quasiregular map-add} and the relation $F=G\circ \tau^{-1}$, we finally obtain that $\log |A(z_n)|\geq d|z_n|^{\rho}$ for some positive constant $d$ and an infinite sequence $(z_n)$ such that $z_n\in\mathcal{L}^{+}$ and $|z_n|$ is large, implying that $\rho(A)\geq \rho$. Hence $\rho(A)=\rho$.

To complete the proof under the assumption $\rho\in(1/2,1)$, we let $f_3$ be a solution linearly independent from $f_1$ of the second order differential equation \eqref{bank-laine0} with the constructed coefficient $A$. Then there are two nonzero constants $d_1$ and $d_2$ such that $f_3=d_1f_1+d_2f_2$ and thus
\begin{equation}\label{recover-arar}
\begin{split}
f_3=d_2f_1\left(F+\frac{d_1}{d_2}\right)=d_2f_1\left(G\circ\tau^{-1}+\frac{d_1}{d_2}\right),
\end{split}
\end{equation}
where $\tau$ is defined in \eqref{recuree1pm117}. By the explicit expressions of $\overline{g}_{m_k,n_k}$ in \eqref{recuree1-main eq1}, \eqref{recuree1-main eq2} and \eqref{recuree1-main eq3} we see that the number of zeros $\check{n}(r,0,f_3)$ of $f_3$ in the region $\tau(\check{\mathbb{H}})$, where $\check{\mathbb{H}}$ is the lower half-plane, satisfies $\log \check{n}(r,0,f_3)\geq d_3r^{\rho}$ for some positive constant $d_3$ and all large $r$, which implies that $\lambda(f_3)=\infty$. Thus we conclude that the Bank--Laine function $E_c=f_1(cf_1+f_2)$ satisfies $\lambda(E_c)=\infty$ for any constant $c$.

Finally, recalling that $m_k=n_k=0$ for $k=1,2,3$ in the two sequences $(m_k)$ and $(n_k)$, we may follow exactly the same process as that in \cite[Section~5]{Bergweilereremenko2019} together with the comments in Subsection~\ref{Completion of the proof-ex} to extend the above results to the case $\rho\in(n/2,n)$ for any positive integer $n$. Moreover, in certain region of the plane we may use similar arguments as for equation \eqref{recover-arar} to show that any other solution $f_3$ linearly independent from $f_1$ of the second order differential equation \eqref{bank-laine0} with the constructed coefficient $A$ cannot satisfy $\lambda(f_3)<\infty$. We shall omit the details. This completes the proof of Theorem~\ref{maintheorem5}.

\section{Proof of Theorem~\ref{maintheorem6}}\label{Proof of maintheorem6}

The proof for Theorem~\ref{maintheorem6} will be similar as that in Section~\ref{Proof of maintheorem4} and that in Section~\ref{Proof of maintheorem5} in many parts. So we shall keep the details minimal in the following.

Also, we first prove Theorem~\ref{maintheorem6} for the case $n=1$. We let $\gamma\in[0,1]$ and $\delta\in[0,1]$ be two numbers and write $\lambda=\delta\gamma$. We also write $\lambda_1=\lambda$ and $\lambda_2=\gamma$. Then $\lambda_1\leq \lambda_2\in[0,1]$ and thus we shall choose two sequences $(m_k)$ and $(n_k)$ associated with $\lambda_1$ and $\lambda_2$ respectively in the same way as in Section~\ref{Proof of maintheorem4}, namely,
\begin{enumerate}
  \item [(I)] if $\lambda_1,\lambda_2\in[0,1)$, then we choose the sequence $(m_k)$ and $(n_k)$ constructed in Subsection~\ref{Preliminary lemmas} respectively;
  \item [(II)] if $\lambda_1<\lambda_2=1$, then we choose the sequence $(m_k)$ constructed in Subsection~\ref{Preliminary lemmas} and $(n_k)$ in the way that $n_k=0$ for $k=1,2,3$ and $n_k=1$ for all $k\geq 4$ respectively;
  \item [(III)] if $\lambda_1=\lambda_2=1$, then we choose the sequence $(m_k)$ and $(n_k)$ in the way that $m_k=0$ for $k=1,2,3$ and $m_k=1$ for all $k\geq 4$ and $n_k=0$ for $k=1,2,3$ and $n_k=1$ for all $k\geq 4$ respectively.
\end{enumerate}
With the definitions of $\overline{g}_{m,n}$ in \eqref{recuree1pm1309hb-1-kjh-1}, for each pair $(m,2n)$ and each pair $(\mathfrak{m},2\mathfrak{n})$, we shall consider the function $\overline{\phi}:\mathbb{R}\to\mathbb{R}$ defined in \eqref{diffeo}, i.e.,
\begin{equation*}
\begin{split}
(\overline{\chi}_{\mathfrak{m},\mathfrak{n}}\circ g_{\mathfrak{m},\mathfrak{n}})(x)=\overline{g}_{\mathfrak{m},\mathfrak{n}}(x)=(\overline{g}_{m,n}\circ\overline{\phi})(x)=(\overline{\chi}_{m,n}\circ g_{m,n}\circ\overline{\phi})(x).
\end{split}
\end{equation*}
Put $N=m+2n+1$ and $M=\mathfrak{m}+2\mathfrak{n}+1$. We will consider the functions $\overline{\phi}$ for the case that $m=m_k$, $n=n_k$, $\mathfrak{m}=m_{k+1}$ and $\mathfrak{n}=n_{k+1}$, that is, $M=N_{k+1}$ and $N=N_k$. In order to simplify the formulas in most situations below we suppress the dependence of $\overline{\phi}$ from $m,n$ and $\mathfrak{m},\mathfrak{n}$ from the notation.

In all the three cases (I), (II) and (III), we have $N\leq 4$ in any case and thus the asymptotic behaviors of $\overline{\phi}$ described in Lemma~\ref{Lemma 0} apply to all $\overline{\phi}_k$ in the case $\overline{\chi}_{m_{k+1},n_{k+1}}=\overline{\chi}_{m_k,n_k}$ and the asymptotic behaviors of $\overline{\phi}$ described in \eqref{recuree1pm12-new-1} and \eqref{recuree1pm12-new-2} apply to all $\overline{\phi}_k$ in the case $\overline{\chi}_{m_{k+1},n_{k+1}}(z)=(z+1)/2$ and $\overline{\chi}_{m_k,n_k}(z)=z$. Moreover, the asymptotic behaviors of $\overline{\phi}$ for the case $\overline{\chi}_{m_{k+1},n_{k+1}}=z$ and $\overline{\chi}_{m_k,n_k}(z)=(z+1)/2$ can be easily obtained as for \eqref{recuree1pm12-new-1} and \eqref{recuree1pm12-new-2}. We shall omit the details. Also, since $N\leq 4$ in any case, for the constant $\overline{s}_{m,n}\in \mathbb{R}$ such that $\overline{g}_{m,n}(\overline{s}_{m,n})=2$ defined in \eqref{trans 0}, we have $\overline{s}_{m,n}=O(1)$ for all $m,n$, which can be proved by slightly modifying the proof of Lemma~\ref{Lemma b}.

Note that $m_k=0$ for $k=1,2,3$ and $n_k=0$ for $k=1,2,3$ for each two sequences $(m_k)$ and $(n_k)$ in (I), (II) and (III). We begin to define a quasiregular maps $U$, $V$ and $G$ by gluing the functions $\overline{g}_{m,n}$ defined in \eqref{recuree1pm1309hb-1-kjh-1}, but with the new sequences $(m_k)$ and $(n_k)$ chosen above. Also, we have $U(iy)=V(iy)$ on the imaginary axis. Again, we first work with the map
\begin{equation*}
\overline{g}_{m,n}\circ\chi_{m,n}=\overline{\chi}_{m,n}\circ g_{m,n}\circ\chi_{m,n}=\overline{\chi}_{m,n}\circ \tilde{g}_{m,n},
\end{equation*}
where $\chi_{m,n}$ is defined in \eqref{diffeo  fajr  ayt}. The explicit expressions for the functions $\overline{g}_{m,n}\circ\chi_{m,n}$ with different pairs $(m_k,2n_k)$ are given as follows. With the definitions of $\tilde{g}_{m_k,n_k}$ defined in \eqref{diffeo  fajr  a-j}, on each strip
\begin{equation*}
\begin{split}
\hat{\Xi}_k=\{x+iy: \quad y\geq0, \quad 2\pi\mathcal{N}_{k-1}<y< 2\pi\mathcal{N}_k\}
\end{split}
\end{equation*}
of the upper half-plane we have
\begin{equation}\label{recuree1-main eq1--nmb4}
\begin{split}
(\overline{g}_{m_k,n_k}\circ\chi_{m_k,n_k})(z)=\left\{
\begin{array}{ll}
\frac{1}{2}\left[g_{m_k,n_k}\left(\frac{x}{l}+i\frac{y-2\pi\mathcal{N}_{k-1}}{N_k}\right)+1\right], & x\geq 0, \\
\frac{1}{2}\left[g_{m_k,n_k}\left(\frac{x+iy-2i\pi\mathcal{N}_{k-1}}{N_k}\right)+1\right],
& x\leq 0
\end{array}
\right.
\end{split}
\end{equation}
if $n_k=1$ and
\begin{equation}\label{recuree1-main eq2--nmb5}
\begin{split}
(\overline{g}_{m_k,n_k}\circ\chi_{m_k,n_k})(z)=\left\{
\begin{array}{ll}
g_{m_k,n_k}\left(\frac{x}{l}+i\frac{y-2\pi\mathcal{N}_{k-1}}{N_k}\right), & x\geq 0, \\
g_{m_k,n_k}\left(\frac{x+iy-2i\pi\mathcal{N}_{k-1}}{N_k}\right),          & x\leq 0
\end{array}
\right.
\end{split}
\end{equation}
if $n_k=0$. On the other hand, denoting $N_{0}=0$, $\mathcal{N}_{0}=0$ and $\mathcal{N}_{-k}=-\mathcal{N}_{k}$, on each strip
\begin{equation*}
\check{\Xi}_{-k}=\{z=x+iy: \quad y\leq 0, \quad 2\pi\mathcal{N}_{-k}\leq y\leq 2\pi \mathcal{N}_{1-k}\}
\end{equation*}
of the lower half-plane we have
\begin{equation}\label{recuree1-main eq3--nmb6}
\begin{split}
(\overline{g}_{m_k,n_k}\circ\chi_{m_k,n_k})(z)=\left\{
\begin{array}{ll}
g_{m_k,n_k}\left(\frac{x}{l}+i\frac{y-2\pi\mathcal{N}_{1-k}}{N_{k}}\right), & x\geq 0, \\
g_{m_k,n_k}\left(\frac{x+iy-2i\pi\mathcal{N}_{1-k}}{N_k}\right),          & x\leq 0.
\end{array}
\right.
\end{split}
\end{equation}
In the above new definitions, we shall choose $l=1$ in the case (I) and $l=3$ in the case (II) and $l=4$ in the case (III). Note that $n_k=0$ for $k=1,2,3$ and thus we have the expressions in \eqref{recuree1-main eq2--nmb5} in $\hat{\Xi}_1$, $\hat{\Xi}_2$ and $\hat{\Xi}_3$. Also note that the two expressions in each of the definition of $\overline{g}_{m_k,n_k}\circ\chi_{m_k,n_k}$ in \eqref{recuree1-main eq1--nmb4}, \eqref{recuree1-main eq2--nmb5} or \eqref{recuree1-main eq3--nmb6} coincide on the imaginary axis.

In view of these definitions, the maps $U$, $V$ and $G$ will not commute with the complex conjugation, but we shall only define them in the upper half-plane since the definitions in the lower half-plane are similar.

The definitions for $U$ and $V$ will be similar as that in Subsection~\ref{Definition of a quasiregular map}, except that we shall use the explicit expressions of $\overline{g}_{m,n}\circ\chi_{m_k,n_k}=\overline{\chi}_{m,n}\circ \tilde{g}_{m,n}$ defined in \eqref{recuree1-main eq1--nmb4} and \eqref{recuree1-main eq2--nmb5}, instead of the functions $\tilde{g}_{m,n}$ defined in \eqref{diffeo  fajr  a-j}. To construct the map $U$, we define
\begin{equation*}
\begin{split}
u_{m,n}: \{z\in\mathbb{C}: \text{Re}\ z\geq0 \}\to\mathbb{C}, \quad u_{m,n}(z)=\overline{g}_{m,n}(\chi_{m,n}(z)+\overline{s}_{m,n}),
\end{split}
\end{equation*}
where $\chi_{m,n}$ is defined in \eqref{diffeo  fajr  ayt}.
Note that $u_{m,n}$ is increasing on the real line and maps $[0,\infty)$ onto $[2,\infty)$.
For each two chosen sequences $(m_k)$ and $(n_k)$, we write $N_k=m_k+2n_k+1$. Put $U(z)=u_{m_k,n_k}(z)$ in the half-strip
\begin{equation*}
\begin{split}
\hat{\Xi}_k^{+}=\{x+iy: x>0, \quad  2\pi\mathcal{N}_{k-1}<y<2\pi\mathcal{N}_k\}.
\end{split}
\end{equation*}
Then we consider the map
\begin{equation*}
\begin{split}
u_{m_k,n_k}: \{z\in\mathbb{C}: \text{Re}\ z\geq0 \}\to\mathbb{C}, \quad u_{m_k,n_k}(z)=\overline{g}_{m_k,n_k}(\chi_{m_k,n_k}(z)+\overline{s}_{m_k,n_k})
\end{split}
\end{equation*}
when $n_k=0$ and
\begin{equation*}
\begin{split}
u_{m_k,n_k}: \{z\in\mathbb{C}: \text{Re}\ z\geq0 \}\to\mathbb{C}, \quad u_{m_k,n_k}(z)=\frac{1}{2}\left[\overline{g}_{m_k,n_k}(\chi_{m_k,n_k}(z)+\overline{s}_{m_k,n_k})+1\right]
\end{split}
\end{equation*}
when $n_k=1$. Consider the function $\overline{\psi}_k:[0,\infty)\to[0,\infty)$ defined by $u_{m_{k+1},n_{k+1}}(x)=u_{m_k,n_k}(\overline{\psi}_k(x))$. Then, denoting by $\phi_k$ the function $\phi$ corresponding to $(m_k,2n_k)$ and $(m_{k+1},2n_{k+1})$, $\overline{\psi}_k$ takes the form
\begin{equation*}
\begin{split}
\overline{\psi}_k(x)=l\overline{\phi}_k\left(\frac{x}{l}+\overline{s}_{m_{k+1},n_{k+1}}\right)-l\overline{s}_{m_{k},n_{k}}.
\end{split}
\end{equation*}
Then we may interpolate in the same way as in Subsection~\ref{Definition of a quasiregular map} to obtain a map $U$ which is continuous on the horizontal lines $\mathcal{L}^{+}_{k}=\{z=x+iy, x>0, y=2\pi\mathcal{N}_k\}$ for any $k$ in the right half-plane: if $2\pi \mathcal{N}_{k-1}\leq y<2\pi\mathcal{N}_k$, say $y=2\pi \mathcal{N}_{k-1}+2\pi N_kt$ where $0\leq t<1$, then we put
\begin{equation*}
\begin{split}
U(x+iy)=u_{m_{k},n_{k}}((1-t)x+t\psi_k(x)+iy)=u_{m_{k},n_{k}}(x+iy+t(\psi_k(x)-x)).
\end{split}
\end{equation*}
On the other hand, to construct a function $V$ in the left half-plane, we define
\begin{equation*}
\begin{split}
v_{m,n}:\{z\in\mathbb{C}: \text{Re}\ z< 0\}\to \mathbb{C},  \quad v_{m,n}(z)=\overline{g}_{m,n}\left(\chi_{m,n}(z)+\overline{s}_{m,n}\right),
\end{split}
\end{equation*}
where $\chi_{m,n}$ is defined in \eqref{diffeo  fajr  ayt}. Note that $v_{m,n}$ maps $(-\infty,0]$ monotonically onto $(1,2]$. Let $(m_k)$ and $(n_k)$ be the two sequences as before. This time we would like to define $V(z)=v_{m_k,n_k}(z)$ in the half-strip
\begin{equation*}
\begin{split}
\hat{\Xi}_k^{-}=\{x+iy: x<0, \ 2\pi \mathcal{N}_{k-1}\leq y<2\pi \mathcal{N}_k\}.
\end{split}
\end{equation*}
Then we consider the map
\begin{equation*}
\begin{split}
v_{m_k,n_k}: \{z\in\mathbb{C}: \text{Re}\ z\leq0 \}\to\mathbb{C}, \quad v_{m_k,n_k}(z)=\overline{g}_{m_k,n_k}(\chi_{m_k,n_k}(z)+\overline{s}_{m_k,n_k})
\end{split}
\end{equation*}
when $n_k=0$ and
\begin{equation*}
\begin{split}
v_{m_k,n_k}: \{z\in\mathbb{C}: \text{Re}\ z\leq0 \}\to\mathbb{C}, \quad v_{m_k,n_k}(z)=\frac{1}{2}\left[\overline{g}_{m_k,n_k}(\chi_{m_k,n_k}(z)+\overline{s}_{m_k,n_k})+1\right]
\end{split}
\end{equation*}
when $n_k=1$. Consider the map $\overline{\psi}_k:(-\infty,0]\to (-\infty,0]$ defined by $v_{m_{k+1},n_{k+1}}(x)=v_{m_k,n_k}(\overline{\psi}_k(x))$. Then, denoting by $\phi_k$ the function $\phi$ corresponding to $(m_k,2n_k)$ and $(m_{k+1},2n_{k+1})$, $\overline{\psi}_k$ takes the form
\begin{equation*}
\begin{split}
\overline{\psi}_k(x)=N_k\overline{\phi}_k\left(\frac{x}{N_{k+1}}+\overline{s}_{m_{k+1},n_{k+1}}\right)-N_k\overline{s}_{m_{k},n_{k}}.
\end{split}
\end{equation*}
Then we may interpolate in the same way as in Subsection~\ref{Definition of a quasiregular map} to obtain a map $V$ which is continuous on the horizontal lines $\mathcal{L}^{-}_{k}=\{z=x+iy, x<0, y=2\pi\mathcal{N}_k\}$ for any $k$ in the left half-plane: if $2\pi \mathcal{N}_{k-1}\leq y<2\pi \mathcal{N}_k$, say $y=2\pi \mathcal{N}_{k-1}+2\pi N_kt$ where $0\leq t<1$, then we put
\begin{equation*}
\begin{split}
V(x+iy)=v_{m_{k},n_{k}}((1-t)x+t\psi_k(x)+iy)=v_{m_{k},n_{k}}(x+iy+t(\psi_k(x)-x)).
\end{split}
\end{equation*}

Note that $U$ and $V$ coincide on the imaginary axis. With the definitions of the maps $U$ and $V$, we define the map
\begin{equation}\label{Def-2}
\begin{split}
G(z)=\left\{
       \begin{array}{ll}
         U(z),   & \text{if} \quad |\arg z|\leq \frac{\pi}{2}; \\
         V(z),   & \text{if} \quad |\arg(-z)|\leq \frac{\pi}{2},
       \end{array}
      \right.
\end{split}
\end{equation}
which is continuous on the horizontal lines $\mathcal{L}^{+}_k$ in the right half-plane and $\mathcal{L}^{-}_{k}$ in the left half-plane for any integer $k$ respectively.

Then we can follow exactly the same process as in Subsection~\ref{Estimation of the dilatation} to show that the map $G$ defined in \eqref{Def-2} is quasiregular in the plane and satisfies the hypothesis of the Teichm\"uller--Wittich--Belinskii. This theorem, together with the existence theorem for quasiconformal mappings, yields that there exists a quasiconformal homeomorphism $\tau:\mathbb{C}\to\mathbb{C}$ and a meromorphic function $F$ such that
\begin{equation}\label{recuree1pm117 final-kjp-1}
\begin{split}
G(z)=F(\tau(z)) \quad  \text{and} \quad \tau(z)\sim z \quad \text{as} \quad z\to\infty.
\end{split}
\end{equation}
Then $F=G\circ \tau^{-1}$ is a locally univalent meromorphic function. Define $E=F/F'$. Then $E$ is a Bank--Laine function. Clearly, $E=f_1f_2$ for two normalized solutions of the second order differential equation \eqref{bank-laine0} with an entire coefficient $A$.

We now estimate the numbers of poles and zeros of $F$ respectively. For the number of poles of $F$, we may use the same arguments as in Subsection~\ref{Completion of the proof} to obtain that $n(r,\infty,F)\asymp (\log r)^2$ when $\lambda=0$ and $n(r,\infty,F)\asymp r^{\lambda}$ when $0<\lambda\leq 1$. Thus $\lambda(f_1)=\lambda$.

To estimate the number of zeros of $F$, we let $r>0$ and choose $k\in \mathbb{N}$ such that $2\pi\mathcal{N}_{k-1}<r\leq 2\pi \mathcal{N}_k$. Note that $\log n(r,0,\overline{g}_{m_k,n_k})\asymp r$ for all pairs $(m_k,2n_k)$, as is shown in Subsection~\ref{Completion of the proof-add}. Denote by $\hat{n}(r,0,\overline{g}_{m_k,n_k})$ the number of zeros of $\overline{g}_{m_k,n_k}$ in the upper half-plane. The estimates in Subsection~\ref{Completion of the proof-add} imply that the number of zeros of $U$ in the upper half-plane satisfies $\log \hat{n}(r,0,U)\leq c_1r$ for some positive constant $c_1$. On the other hand, for a given pari $(m,n)$, denoting by $n(\ddot{r},0,\overline{g}_{m,n})$ the number of zeros of $\overline{g}_{m,n}$ in the square region $\{z=x+iy:|x|+|y|\leq r\}$, we also have
\begin{equation*}
\begin{split}
n(\ddot{r}/2,0,\overline{g}_{m,n})\leq n(r,0,\overline{g}_{m,n})\leq n(2\ddot{r},0,\overline{g}_{m,n}).
\end{split}
\end{equation*}
By the periodicity of $\overline{g}_{m,n}$, we easily see that the number $n(\dot{r},0,\overline{g}_{m,n})$ of zeros of $\overline{g}_{m,n}$ in only one fixed strip $\Pi_k=\{x+iy: x>0, 2\pi(k-1)<y<2\pi k\}$ satisfies $n(\ddot{r},0,\overline{g}_{m,n})=[r/\pi+O(1)]n(\dot{r},0,\overline{g}_{m,n})$ and thus $\log n(\dot{r},0,\overline{g}_{m,n})\asymp r$. This implies that $\log \hat{n}(r,0,U)\geq c_2r$ for some positive constant $c_2$. Thus the number of zeros of $U$ in the upper half-plane satisfies $\log\hat{n}(r,0,U)\asymp r$. Similarly, the number of zeros of $V$ in the upper half-plane satisfies $\log\hat{n}(r,0,V)\asymp r$ and so the number of zeros of $G$ in the upper half-plane satisfies $\log\hat{n}(r,0,G)\asymp r$. For the number of zeros of $G$ in the lower half-plane, we can use the same arguments as in Subsection~\ref{Completion of the proof} to obtain that $\check{n}(r,0,G)\asymp (\log r)^2$ when $\gamma=0$ and $\check{n}(r,0,G)\asymp r^{\gamma}$ when $0<\gamma\leq 1$. By combining the estimates for $\hat{n}(r,0,G)$ and $\check{n}(r,0,G)$ together, we conclude that $\log n(r,0,G)\asymp r$. Finally, by the relations in \eqref{recuree1pm117 final-kjp-1}, the number of zeros of $F$ satisfies $\log n(r,0,F)\asymp r$. Thus $\lambda(f_2)=\infty$.

To determine the order of $A$, we may use similar arguments as in Subsection~\ref{Completion of the proof} and in Subsection~\ref{Completion of the proof-add} to estimate $\log|G|$ and $\log|F|$ from above and conclude that the resulting Bank--Laine function $E$ satisfies $m(r,1/E)=O(r)$ and further that $\rho(A)\leq 1$. Moreover, by similar arguments as there we may also consider $\log|E|$ along a curve in $(\tau\circ\tilde{\chi}_{1,1}^{-1})(\hat{\Xi}_1^{+}\cup\check{\Xi}_{-1}^{+}\cup \mathbb{R}^{+})$, say $\mathcal{L}^{+}$, such that $(\tilde{\chi}_{1,1}\circ\tau^{-1})(\mathcal{L}^{+})$ is the positive real axis, and conclude that $\rho(A)\geq 1$. Hence $\rho(A)=1$. Moreover, for any other solution $f_3$ linearly independent from $f_1$ of the second order differential equation \eqref{bank-laine0} with the constructed coefficient $A$, we may consider the equation \eqref{recover-arar} and show that $f_3$ satisfies $\lambda(f_3)=\infty$. We omit the details. This completes the proof for the case $n=1$.

Finally, the above results can be extended to the case for any positive integer $n$ in the same way as that in \cite[Section~5]{Bergweilereremenko2019} together with the comments in Subsection~\ref{Completion of the proof-ex} and also that in Subsection~\ref{Completion of the proof-add}. We shall also omit the details. This completes the proof of Theorem~\ref{maintheorem6}.

\section{Acknowledgements}

The author would like to greatly thank professor Walter Bergweiler for having enlightening discussions on the Bank--Laine conjecture and the quasiconformal surgery during his visit to the University of Kiel in July~2024. In particular, during the discussions the author specified the main obstacles in the proof of Theorem~\ref{maintheorem4} and Theorem~\ref{maintheorem5}. Some ideas for overcoming these obstacles were proposed during the discussions.

\end{document}